\documentclass[11pt, oneside]{article}

\usepackage{epsfig}
\usepackage{tikz}
\usepackage{url}
\usetikzlibrary{fadings,shapes,arrows.meta,positioning}

\usepackage{hhline}
\usepackage[title]{appendix}

\usepackage{amsmath,amstext,amssymb,amsfonts,amsthm}
\usepackage{mathrsfs}
\usepackage{xcolor}
\usepackage{mathtools}
\usepackage{graphicx}
\usepackage{booktabs}
\usepackage{multicol}
\usepackage{multirow}
\usepackage[labelfont=bf]{caption}
\usepackage{subcaption}
\usepackage{siunitx}

\newtheorem{lem}{Lemma}[section]
\newtheorem{thm}{Theorem}[section]

\newtheorem{rem}{Remark}[section]

\newtheorem{exmp}{Example}

\numberwithin{equation}{section}
\numberwithin{figure}{section}

\newcommand{\jl}{[\![}
\newcommand{\jr}{]\!]}
\newcommand{\la}{\langle}
\newcommand{\ra}{\rangle}
\newcommand{\jpl}{|\![}
\newcommand{\jpr}{]\!|}

\tikzfading[name=fade right, left color=transparent!0, right color=transparent!100]
\tikzfading[name=fade left, right color=transparent!0, left color=transparent!100]
\tikzfading[name=fade down, bottom color=transparent!0, top color=transparent!100]
\tikzfading[name=fade up, top color=transparent!0, bottom color=transparent!100]
\allowdisplaybreaks[4]

\title{An Interior Penalty Discontinuous Galerkin Method for an Interface Model of Flow in Fractured Porous Media\footnotemark[1]}
\author{
Yong Liu\footnotemark[2]
\and
Ziyao Xu\footnotemark[3]
}

\date{}

\begin{document}

\maketitle
\renewcommand{\thefootnote}{\fnsymbol{footnote}}
\footnotetext[1]{The first author is funded in part by Strategic Priority Research Program of the Chinese Academy of Sciences under the Grant No. XDB0640000, China Natural Science Foundation under grant 12201621, 12288201 and the Youth Innovation Promotion Association (CAS).}
\footnotetext[2]{LSEC, Institute of Computational Mathematics, Academy of Mathematics and Systems Science, Chinese Academy of Sciences, Beijing 100190, P.R. China.  E-mail: yongliu@lsec.cc.ac.cn}
\footnotetext[3]{Department of Applied and Computational Mathematics and Statistics,
University of Notre Dame, Notre Dame, IN 46556, USA. E-mail: zxu25@nd.edu}

\begin{center}
\small
\begin{minipage}{0.9\textwidth}
\textbf{Abstract.}
Discrete fracture models with reduced-dimensional treatment of conductive and blocking fractures are widely used to simulate fluid flow in fractured porous media. Among these, numerical methods based on interface models are intensively studied, where the fractures are treated as co-dimension one manifolds in a bulk matrix with low-dimensional governing equations. In this paper, we propose a simple yet effective treatment for modeling the fractures on fitted grids in the interior penalty discontinuous Galerkin (IPDG) methods without introducing any additional degrees of freedom or equations on the interfaces. 
We conduct stability and {\em hp}-analysis for the proposed IPDG method, deriving optimal a priori error bounds concerning mesh size ($h$) and sub-optimal bounds for polynomial degree ($k$) in both the energy norm and the $L^2$ norm. Numerical experiments involving published benchmarks validate our theoretical analysis and demonstrate the method's robust performance. Furthermore, we extend our method to two-phase flows and use numerical tests to confirm the algorithm's validity.

\medskip
\textbf{Key words.} discrete fracture model, interior penalty discontinuous Galerkin (IPDG), interface problem, stability analysis, error estimates, two-phase flow

\medskip

\textbf{AMS classification}.
65N12, 65N30, 65M60
\end{minipage}
\end{center}
\setlength{\parindent}{2em}

\pagenumbering{arabic}

\section{Introduction}\label{Sect:intro}

As a result of geological processes, fractures are common in rocks. 
Depending on the materials filling them, fractures can act either as preferred conductive paths or as blocking barriers for subsurface flows. 
Their permeability may be higher or lower than that of the surrounding bulk matrix by several orders of magnitude. 
Because of their significant impact on the hydraulic properties of rocks, accurate and efficient simulation of fluid flow in fractured porous media is desired in many applications, 
such as oil recovery in naturally fractured reservoirs, geothermal extraction, and hydraulic fracturing.

Due to the small scale of fractures, directly simulating fractured rocks as a heterogeneous medium with extremely refined grids in fracture regions is often unaffordable for computational resources.
To reduce the computational cost, the discrete fracture models (DFMs) that use hybrid-dimensional approaches to treat fractures as low-dimensional geometries have been actively studied over the past decades.
A common approach to account for conductive fractures in DFMs is to superpose the fracture flows onto the flow in the porous matrix, using the aperture of the fracture multiplied by the fracture flow as the dimensional homogeneity factor.
Examples of this approach include the finite element (FE) \cite{kim2000finite, karimi2001numerical,zhang2013accurate} and the vertex-centered finite volume (Box) \cite{monteagudo2004control, reichenberger2006mixed, monteagudo2007comparison} discrete fracture models on fitted meshes. 
In these models, fracture flows are integrated over the matrix cells, eliminating the need to compute flux transfer between the matrix and fractures \cite{hoteit2006compositionalb}.
Alternatively, fractures can be assigned additional degrees of freedom and discretized with individual equations with flux exchange to the surrounding bulk matrix.
For instance, the cell-centered finite volume (CCFV) discrete fracture models based on the two-point flux approximation (TPFA) \cite{karimi2004efficient} and the multi-point flux approximation (MPFA) \cite{sandve2012efficient, ahmed2015control,glaser2017discrete}, the mortar-mixed finite element method \cite{boon2018robust}, as well as the extended Box discrete fracture model (EBox-DFM) \cite{glaser2022comparison}. 
To relieve the constraints on body-fitted grids, various non-fitted DFMs have been proposed, such as the embedded discrete fracture model (EDFM) \cite{li2008efficient, moinfar2014development}, which computes the flux exchange between matrix cells and fracture cells based on their average distance, and its extension, pEDFM \cite{ctene2017projection, jiang2017improved} to include blocking barriers.
Also available are the Lagrange multiplier methods \cite{koppel2019lagrange, schadle20193d} for coupling the matrix and fracture equations via the Lagrange multiplier, 
the extended finite element method (XFEM) \cite{d2012mixed, flemisch2016review, schwenck2015dimensionally} and immersed finite element method (iFEM) \cite{zhao2024discrete} to embed the information of fractures into the local function space of matrix elements.
Recently, a reinterpreted discrete fracture model (RDFM) \cite{xu2020hybrid,xu2023hybrid,fu2023hybridizable} was proposed to describe fractures as Dirac-$\delta$ coefficients in the total permeability tensor of porous media.
A general interface model accounting for arbitrary types of fractures was proposed in \cite{martin2005modeling} with the mixed finite element discretization.
This model reduces fractures to interfaces governed by a low-dimensional version of Darcy’s law and establishes appropriate coupling conditions with the bulk matrix. 
The interface models have been the basis for the development of many discrete fracture models, based on finite volume \cite{angot2009asymptotic}, mimetic finite difference \cite{antonietti2016mimetic}, discontinuous Galerkin \cite{antonietti2019discontinuous, antonietti2020ME,ma2021discontinuous}, weak Galerkin \cite{wang2018weak}, hybrid high-order \cite{chave2018hybrid}, and mixed finite element methods \cite{arraras2019mixed}, among others.

The Discontinuous Galerkin (DG) method is widely adopted in computing fluid flow in porous media due to its advantages in high-order accuracy, flexibility for complex computational domains, and ease of $hp$ adaptation and parallelization. 
Various types of DG methods can be applied to porous media flow \cite{rivie2000part,hoteit2006compositionala,chuenjarern2019high,tao2023oscillation} and interface problems \cite{huynh2013high, cangiani2018adaptive, huang2020high}.
In this work, we propose an Interior Penalty Discontinuous Galerkin (IPDG) \cite{riviere2002discontinuous,riviere2008discontinuous} discrete fracture model discretizing an interface problem on fitted meshes, without introducing additional degrees of freedom and equations on interfaces, thereby reducing computational costs.
The degrees of freedom and the sparsity of the stiffness matrix in the proposed DG scheme are identical to those in the original IPDG method, which is a main advantage of our method compared to existing work based on DG methods. 
The proposed scheme can be implemented in enriched Galerkin \cite{kadeethum2020flow} spaces to further reduce the computational cost.
Numerical analyses have been conducted on DG methods for various scenarios: single-phase flow in fractured porous media \cite{antonietti2019discontinuous, antonietti2020ME,mozolevski2021high}, two-phase flow within a single domain \cite{Epshteyn2009JCAM,Kou2013SINUM}, and a hybrid-dimensional fracture model for two-phase flow \cite{chen2023discontinuous}. 
Additionally, Arnold et al. \cite{Arnold2002SINUM} offer a unified analysis of DG methods applied to elliptic problems. 
This paper presents an {\em hp}-analysis of IPDG methods for single-phase flow in both conductive and blocking fractures. 
Through the employment of {\em hp} inverse estimates, interpolation techniques, and duality arguments, we derive both optimal a priori error bounds in terms of mesh size $h$ and sub-optimal bounds regarding polynomial degree $k$, assessed in both the energy norm and $L^2$ norm.
The extension to two-phase flows is also established and tested in the paper.

The remainder of the paper is organized as follows: In Section \ref{Sect:model}, we describe the interface model of fluid flow in fractured porous media involving both conductive and blocking fractures. In Section \ref{Sect:scheme}, we establish an IPDG scheme for the model. The stability analysis and error estimates of the proposed scheme are conducted in Section \ref{Sect:theory}. 
Numerical tests involving published benchmarks are presented in Section \ref{Sect:tests} to validate the theoretical analysis and demonstrate the effectiveness of the method. 
We extend our approach to two-phase flows and use a numerical test to demonstrate the validity of the algorithm in Section \ref{Sect:two-phase}. Finally, we conclude with remarks in Section \ref{Sect:summary}.

\section{An interface model of flow in fractured porous media}\label{Sect:model}
Generally, the fluid flow in heterogeneous porous media is governed by the Poisson equation
\begin{equation*}
-\nabla\cdot(\mathbf{K}\nabla p)=q, \quad (x,y)\in\Omega,
\end{equation*}
where $\mathbf{K}$ is the permeability tensor of the porous media, $p$ is the pressure of the fluid and $q$ is the source term.
In the presence of fractures, $\mathbf{K}(x,y)$ has fine structures in the domain and extremely refined grids in fracture regions are desired to resolve the small scale of fractures.
Such an approach is referred to as the equi-dimensional method.
To avoid grid refinement, hybrid-dimensional approaches have been developed.

\begin{figure}[!htbp]
  \centering
  \begin{tikzpicture}[scale=1.0]
    \draw[thick] (-4,-2) -- (4,-2) -- (4,2) -- (-4,2) -- cycle;
    
    \draw[thick] (-2,-2) -- (-1,2); 
    \draw[thick] (2,-2) -- (1,2); 
    
    \draw[-Latex] (-1.25,1) -- ++(-1,0.25) node[above right] {$\mathbf{n}_1^+$};
    \draw[-Latex] (-1.75,-1) -- ++(1,-0.25) node[below left] {$\mathbf{n}_1^-$};
    \draw[-Latex] (-1.8,-0.4) -- ++(0.25,1.0) node[below left] {$\boldsymbol{\nu}_1$};    
    \draw[-Latex] (1.25,1) -- ++(-1,-0.25) node[above left] {$\mathbf{n}_2^+$};
    \draw[-Latex] (1.8,-0.4) -- ++(-0.25,1) node[below right] {$\boldsymbol{\nu}_2$};
    \draw[-Latex] (1.75,-1) -- ++(1,0.25) node[below right] {$\mathbf{n}_2^-$};
    
    \node at (-3,0) {$\Omega_1$};
    \node at (0,0) {$\Omega_2$};
    \node at (3,0) {$\Omega_3$};
    
    \node at (-1.4,0) {$\gamma_1$};
    \node at (1.4,0) {$\gamma_2$};
  \end{tikzpicture}
  \caption{The geometry of an interface model involving fractures.}
  \label{fig:hybridDomain}
\end{figure}
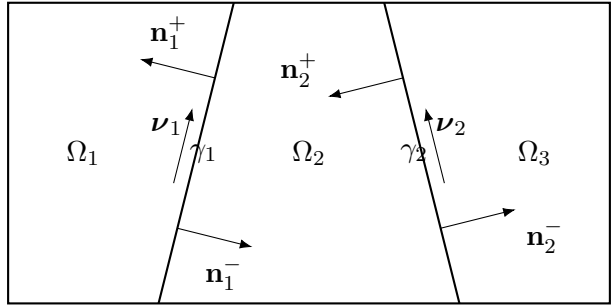

In this paper, we consider the following interface problem for modeling the fluid flow in fractured media with both conductive and blocking types of fractures,
\begin{equation}\label{eq:PDEmodel}
\begin{cases}
-\nabla\cdot(\mathbf{K}_m\nabla p)=q, &(x,y)\in\Omega\setminus\left(\gamma_1\cup \gamma_2\right)\\
\mathbf{u}^{-}\cdot\mathbf{n}_1^-+\mathbf{u}^{+}\cdot\mathbf{n}_1^+=\mathbf{0}, & (x,y)\in\gamma_1\\
\mathbf{u}^{-}\cdot\mathbf{n}_1^-=-k_b\frac{p^+-p^-}{a}, & (x,y)\in\gamma_1\\
p^+-p^-=0, & (x,y) \in \gamma_2\\
-\frac{\partial }{\partial \boldsymbol{\nu}_2}\left( ak_f\frac{\partial p^+}{\partial \boldsymbol{\nu}_2}\right) = { q_f+}
\mathbf{u}^{-}\cdot\mathbf{n}_2^-+\mathbf{u}^{+}\cdot\mathbf{n}_2^+, & (x,y) \in \gamma_2
\end{cases}
\end{equation}
subject to the boundary conditions
\begin{equation}\label{eq:BdrCond}
p=g_D ~\text{on}~ \Gamma_D, \quad (\mathbf{K}_m\nabla p)\cdot\mathbf{n}= g_N ~\text{on}~\Gamma_N:=\partial\Omega\setminus\Gamma_D,
\end{equation}
where $\mathbf{K}_m$ is the permeability tensor of the bulk matrix, $p$ is the pressure, $q$ { and $q_f$ are source terms}, $\Omega$ is the domain of interest, $\gamma_1$ is the blocking barrier interface, $\gamma_2$ is the conductive fracture interface, $\mathbf{u}^{\pm}:=-\mathbf{K}_m^{\pm}\nabla p^{\pm}$ are Darcy velocities evaluated from either side of $\gamma_i,i=1,2$,
$\mathbf{n}_i^{\pm}$ are unit outer normal vectors on $\gamma_i$ from either side, $\boldsymbol{\nu}_i$ are unit tangential vectors on $\gamma_i$, $k_b (\ll 1)$ is the permeability of the blocking barrier, $k_f (\gg1)$ is the permeability of the conductive fracture, $a (\ll 1)$ are the apertures of the barrier and fracture, $\Gamma_D$ and $\Gamma_{N}$ are the Dirichlet and Neumann boundaries, respectively, and $\mathbf{n}$ is the unit outer normal on the boundary. Here we assume that $\gamma_1$ and $\gamma_2$ are disjoint and divide the domain $\Omega$ into three subdomains $\Omega_i,i=1,2,3$, see Figure \ref{fig:hybridDomain}. We also denote $\Omega_{\gamma_1}^{-}=\Omega_1$, $\Omega_{\gamma_1}^{+}=\Omega_2$, $\Omega_{\gamma_2}^{-}=\Omega_2$, and $\Omega_{\gamma_2}^{+}=\Omega_3$.

The interface model considered here, which neglects tangential flow in blocking barriers and pressure discontinuities across conductive fractures, can be regarded as the limit case of the general interface model established in \cite{martin2005modeling}.
We refer the readers to the discussions about equations (4.5) and (4.6) therein. 
It is important to note that the model in this paper is not suitable to fractures with high tangential permeability but low normal permeability, as considered in \cite{martin2005modeling}, a situation that occurs less frequently in reality.

\section{The interior penalty discontinuous Galerkin methods}\label{Sect:scheme}
In this section, we establish the interior penalty discontinuous Galerkin methods for the interface problem \eqref{eq:PDEmodel} - \eqref{eq:BdrCond}.

We assume that the computational domain $\Omega$ is partitioned into a triangular fitted mesh, denoted as $\mathcal{T}_h=\{T\}$, where $T$ are triangular cells and the barriers and fractures $\gamma_i,i=1,2,$ are aligned with the sides of $T$ (refer to the Figure \ref{fig:jump_of_vertices}).
As is customary, we assume the triangles $T$ to be shape-regular. Let $h_T$ stand for the diameter of $T$ and $h=\max_{T\in \mathcal{T}_h} h_T$.
We denote collections of different types of edges by
$\mathcal{E}^D=\{e:e=\partial T \cap \Gamma_D \,, \forall T \in \mathcal{T}_h\}$, 
$\mathcal{E}^N=\{e:e=\partial T \cap \Gamma_N \,, \forall T \in \mathcal{T}_h\}$, 
$\mathcal{E}=\{e:e=\partial T \cap \partial T' \,, \forall T, T' \in \mathcal{T}_h\}$,
$\mathcal{E}^{\gamma_i}=\{e:e=\partial T \cap \gamma_i \,, \forall T \in \mathcal{T}_h\}, i=1,2$, 
and $\mathcal{E}^{\circ}=\mathcal{E}\setminus (\mathcal{E}^{\gamma_1}\cup \mathcal{E}^{\gamma_2})$. 
Let $\mathcal{V}=\{(x^\star,y^\star)\in\gamma: (x^\star,y^\star) \text{ is one of vertices of }T\,, \forall T \in \mathcal{T}_h\}$ be all of vertices on the fracture $\gamma_2$.
We denote collections of different types of vertices by $\mathcal{V}^{\circ}\cup\mathcal{V}^D\cup \mathcal{V}^N=\mathcal{V}$, where $\mathcal{V}^{\circ}=\{(x^\star,y^\star)\in \mathcal{V}: (x^\star,y^\star)\in \Omega^{\circ} \}$, $\mathcal{V}^{D}=\{(x^\star,y^\star)\in \mathcal{V}: (x^\star,y^\star)\in \bar{\Gamma}_D \}$, and $\mathcal{V}^{N}=\{(x^\star,y^\star)\in \mathcal{V}: (x^\star,y^\star)\in \bar{\Gamma}_N \}$. 
Furthermore, we denote $\mathcal{T}_h^{\gamma_2}=\mathcal{T}_h^{\gamma_2,-}\cup \mathcal{T}_h^{\gamma_2,+}$, where $\mathcal{T}_h^{\gamma_2,\pm}:=\{T\in \mathcal{T}_h\cap \Omega_{\gamma_2}^{\pm}: \text{ there exists } e\in \mathcal{E}^{\gamma_2} \text{ s.t. } e \in \partial T\}$.

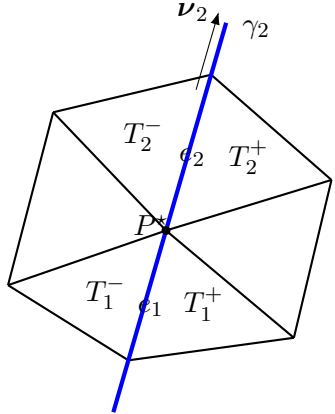
\begin{figure}[!htbp]
 \centering
  \begin{tikzpicture}[scale=1.0]
    \draw[thick] (0,-0.073) -- (0.6,2);
    \draw[thick] (0,-0.073) -- (2.2,0.6);
    \draw[thick] (0,-0.073) -- (1.7,-1.5);
    \draw[thick] (0,-0.073) -- (-0.5,-1.8);
    \draw[thick] (0,-0.073) -- (-2.1,-0.8);
    \draw[thick] (0,-0.073) -- (-1.5,1.5);
    \draw[thick] (0.6,2) -- (2.2,0.6);
    \draw[thick] (2.2,0.6) -- (1.7,-1.5);
    \draw[thick] (1.7,-1.5) -- (-0.5,-1.8);
    \draw[thick] (-0.5,-1.8) -- (-2.1,-0.8);
    \draw[thick] (-2.1,-0.8) -- (-1.5,1.5);
    \draw[thick] (-1.5,1.5) -- (0.6,2);
    \draw[thick, path fading=fade down] (0.6,2) -- (-0.1,2.7);
    \draw[thick, path fading=fade down] (0.6,2) -- (0.8, 2.6909);
    \draw[thick, path fading=fade down] (0.6,2) -- (1.5,2.2);
    \draw[thick, path fading=fade right] (2.2,0.6) -- (2.6,1.5);
    \draw[thick, path fading=fade right] (2.2,0.6) -- (3.3,-0.2);
    \draw[thick, path fading=fade right] (1.7,-1.5) -- (3.0,-1.0);
    \draw[thick, path fading=fade right] (1.7,-1.5) -- (2.9,-2.5);
    \draw[thick, path fading=fade up] (1.7,-1.5) -- (1.3,-2.6);
    \draw[thick, path fading=fade up] (-0.5,-1.8) -- (0.2,-2.7);
    \draw[thick, path fading=fade up] (-0.5,-1.8) -- (-1.8,-2.5);
    \draw[thick, path fading=fade up] (-0.5,-1.8) -- (-0.7,-2.49); 
    \draw[thick, path fading=fade up] (-2.1,-0.8) -- (-2.2,-2.0);
    \draw[thick, path fading=fade up] (-2.1,-0.8) -- (-3.2,-1.0);
    \draw[thick, path fading=fade down] (-2.1,-0.8) -- (-3.3,0.5);
    \draw[thick, path fading=fade left] (-1.5,1.5) -- (-2.8,1.0);
    \draw[thick, path fading=fade left] (-1.5,1.5) -- (-2.8,2.0);
    \draw[thick, path fading=fade down] (-1.5,1.5) -- (-1.0,2.4);            
    \draw[ultra thick, blue] (-0.7,-2.49) -- (0.8, 2.6909);
    \node at (1.2,2.6) {$\gamma_2$};
    \node at (-0.2,-1.1) {$e_1$};
    \node at (0.35,0.9) {$e_2$};
    \node at (-0.3,1.2) {$T_{2}^{-}$};
    \node at (1.1,0.9) {$T_{2}^{+}$};
    \node at (-0.8,-0.9) {$T_{1}^{-}$};
    \node at (0.5,-1.05) {$T_{1}^{+}$};
    \node at (-0.2,0.0) {$P^{\star}$};
    \node at (0,-0.073) [circle,fill,inner sep=1.2pt] {};
    \draw[-Latex] (0.4,1.8) -- ++(0.3,1.036) node[left] {$\boldsymbol{\nu}_2$};
  \end{tikzpicture}  
  \caption{Local area of a body-fitted mesh near the fracture $\gamma_2$.}\label{fig:jump_of_vertices}
 \end{figure}
 
The DG function space is defined as
\begin{equation*}
V_{h,k}^{DG}(\Omega):=\{v\in L^2(\Omega): v|_{T} \in P^k(T), \quad \forall T  \in \mathcal{T}_h\},
\end{equation*}
where $P^k(T)$ is the space of polynomials of degree no greater than $k$ on $T$.

For $v\in V_{h,k}^{DG}$, we define the average $\{v\}$ and the jump $\jl v \jr$ of $v$ on edges of elements as follows. 
Let $e\in \mathcal{E}$ be an interior edge shared by elements $T_1$ and $T_2$. Define the unit normal vectors $\mathbf{n}_1$ and $\mathbf{n}_2$ on $e$ pointing exterior of $T_1$ and $T_2$, respectively. We set
\begin{align*}
    \{v\}=\frac 12 (v_1+v_2), \quad \jl v\jr=v_1\mathbf{n}_1+v_2\mathbf{n}_2 \quad \text{ on } e \in \mathcal{E},
\end{align*}
where $v_i:=v|_{\partial T_i}$. For $e\in \mathcal{E}^D\cup \mathcal{E}^N$, we set
\begin{align*}
\{v\}=v,\quad \jl v \jr = v\mathbf{n} \quad \text{ on } e \in \mathcal{E}^D\cup \mathcal{E}^N.
\end{align*}
Likewise, we define the average $\{\mathbf{v}\}$ and the jump $\jl \mathbf{v} \jr$ of a vector-valued function $\mathbf{v}\in [V_{h,k}^{DG}]^2$ on edges as 
\begin{align*}
    \{\mathbf{v}\}=\frac 12 (\mathbf{v}_1+\mathbf{v}_2), \quad \jl \mathbf{v}\jr=\mathbf{v}_1\cdot\mathbf{n}_1+\mathbf{v}_2\cdot\mathbf{n}_2 \quad \text{ on } e \in \mathcal{E},
\end{align*}
and 
\begin{align*}
\{\mathbf{v}\}=\mathbf{v},\quad \jl \mathbf{v} \jr = \mathbf{v}\cdot\mathbf{n} \quad \text{ on } e \in \mathcal{E}^D\cup \mathcal{E}^N.
\end{align*}
We also define the average $\{v\}_{P^{\star}_{\pm}}$ and jump $\jl v\jr_{P^{\star}_{\pm}}$ of $v$ at vertices $P^{\star}$ on $\gamma_2$ (refer to Figure \ref{fig:jump_of_vertices}), for discretizing the interface condition of the conductive fractures. 
Assume that $P^\star\in\mathcal{V}^{\circ}$ is an interior vertex shared by four elements $T_1^{\pm}$ and $T_2^\pm$, where $T_1^{\pm}$ share a common edge $e_1$, and $T_2^{\pm}$ share a common edge $e_2$. 
The unit tangential vector $\boldsymbol{\nu}_2$ points from $e_1$ to $e_2$.
 We set
\begin{align*}
    \{v\}_{P^{\star}_{\pm}}=\frac{1}{2}(v|_{ T_1^{\pm}}+v|_{T_2^{\pm}}), \quad \jl v\jr_{P^{\star}_{\pm}}=v|_{T_1^{\pm}}-v|_{T_2^{\pm}} \quad \text{ at } P^\star\in\mathcal{V}^{\circ}.
\end{align*}
For $P^\star\in \mathcal{V}^D$, we set
\begin{align*}
    \{v\}_{P^{\star}_{\pm}}=v|_{T^{\pm}}, \quad \jl v\jr_{P^{\star}_{\pm}}=\text{sign}(\boldsymbol{\nu}_2\cdot \mathbf{n})v|_{T^{\pm}} \quad \text{ at } P^\star\in \mathcal{V}^D,
\end{align*}
where $T^{\pm}\in \mathcal{T}_h^{\gamma_2,\pm}$, which are such that $P^{\star} \in \partial T^{-} \cap \partial T^{+}$, and $T^{\pm}$ share a common edge on $\gamma_2$. 
The $\text{sign}(\cdot)$ is the sign function, i.e., $\text{sign}(x)=1$ if $x\geq 0$ and $\text{sign}(x)=-1$ if $x<0$. 


The IPDG scheme for the interface model \eqref{eq:PDEmodel} - \eqref{eq:BdrCond} is to find $p_h\in V_{h,k}^{DG}$, such that
\begin{equation}\label{eq:scheme}
a_h(p_h, \xi)+b_h(p_h,\xi)=F_h(\xi ) +G_h(\xi) \quad \forall\xi\in V_{h,k}^{DG},
\end{equation}
where
\begin{equation}\label{eq:a}
\begin{split}
    a_h(p_h,\xi)=&(\mathbf{K}_m\nabla p_h,\nabla \xi)_{\mathcal{T}_h}-\la \{\mathbf{K}_m\nabla p_h\},\jl\xi\jr \ra_{\mathcal{E}^{0}\cup \mathcal{E}^D\cup \mathcal{E}^{\gamma_2}}+\sigma \la 
    \jl p_h\jr , \{\mathbf{K}_m\nabla\xi\} \ra_{\mathcal{E}^0\cup \mathcal{E}^D\cup \mathcal{E}^{\gamma_2}}\\
    &+\la \alpha \jl p_h\jr, \jl \xi\jr \ra_{\mathcal{E}^0\cup \mathcal{E}^D\cup \mathcal{E}^{\gamma_2}}+\la\frac{k_b}{a}\jl p_h\jr,\jl\xi\jr \ra_{\mathcal{E}^{\gamma_1}},
\end{split}
\end{equation}
\begin{equation}\label{eq:b}
\begin{split}
    b_h(p_h,\xi)=&\frac 12  \la ak_f \frac{\partial p_h^-}{\partial \boldsymbol{\nu}_2}, \frac{\partial \xi^-}{\partial \boldsymbol{\nu}_2}\ra_{\mathcal{E}^{\gamma_2}}+\frac 12 \la  ak_f \frac{\partial p_h^+}{\partial \boldsymbol{\nu}_2}, \frac{\partial \xi^+}{\partial \boldsymbol{\nu}_2}\ra_{\mathcal{E}^{\gamma_2}}\\
    &-\frac 12[ak_f\{\frac{\partial p_h}{\partial \boldsymbol{\nu}_2}\}_{P_{-}^{\star}}, \jl \xi \jr_{P_{-}^\star}]_{\mathcal{V}^\circ \cup \mathcal{V}^D} -\frac 12[ak_f\{\frac{\partial p_h}{\partial \boldsymbol{\nu}_2}\}_{P_{+}^{\star}}, \jl \xi \jr_{P_{+}^\star}]_{\mathcal{V}^\circ \cup \mathcal{V}^D} \\
    &+\frac \sigma 2[ak_f\{\frac{\partial \xi }{\partial \boldsymbol{\nu}_2}\}_{P_{-}^{\star}}, \jl p_h \jr_{P_{-}^\star}]_{\mathcal{V}^\circ \cup \mathcal{V}^D} +\frac \sigma 2[ak_f\{\frac{\partial \xi}{\partial \boldsymbol{\nu}_2}\}_{P_{+}^{\star}}, \jl p_h \jr_{P_{+}^\star}]_{\mathcal{V}^\circ \cup \mathcal{V}^D} \\
    &+ [\tilde{\alpha}\jl p_h \jr_{P_{-}^\star},\jl \xi \jr_{P_{-}^\star}]_{\mathcal{V}^\circ \cup \mathcal{V}^D} +[\tilde{\alpha}\jl p_h \jr_{P_{+}^\star},\jl \xi \jr_{P_{+}^\star}]_{\mathcal{V}^\circ \cup \mathcal{V}^D},
\end{split}
\end{equation}
and
\begin{equation}\label{eq:F}
F_h(\xi)=(q,\xi)_{\mathcal{T}_h}+\la g_N, \xi \ra_{\mathcal{E}^N} +\sigma \la g_D, \mathbf{K}_m\nabla\xi\cdot\mathbf{n}\ra_{\mathcal{E}^D} +\la \alpha g_D,\xi \ra_{\mathcal{E}^D},
\end{equation}
\begin{equation}\label{eq:G}
G_h(\xi)={ \la q_f, \{\xi\} \ra_{\mathcal{E}^{\gamma_2}}}+\frac{\sigma}{2}[ak_f(\{\frac{\partial \xi}{\partial \boldsymbol{\nu}_2}\}_{P_{-}^\star}+\{\frac{\partial \xi}{\partial \boldsymbol{\nu}_2}\}_{P_{+}^\star}),\text{sign}(\boldsymbol{\nu}_2\cdot \mathbf{n})g_D]_{\mathcal{V}^D}+[\tilde{\alpha}(\xi_{P_{-}^\star}+\xi_{P_{+}^\star}),g_D]_{\mathcal{V}^D},
\end{equation}
where $\sigma=-1,1,0$ corresponds to the symmetric interior penalty Galerkin (SIPG), nonsymmetric interior penalty Galerkin (NIPG), and incomplete interior penalty Galerkin (IIPG) methods, respectively.
The $\alpha\in L^{\infty}(\mathcal{E}^0\cup \mathcal{E}^D\cup \mathcal{E}^{\gamma_2})$ is a penalty function that is defined by
\begin{align*}
    \alpha|_e=\alpha_0 k^2 h_e^{-1} \quad \forall e\in \mathcal{E}^0\cup \mathcal{E}^D \cup \mathcal{E}^{\gamma_2}
\end{align*}
where the mesh function $h_e\in L^{\infty}(\mathcal{E}^0\cup\mathcal{E}^D\cup \mathcal{E}^{\gamma_2})$ is defined by $h_e|_{e}=|e|$,  the $|e|$ denotes the length of $e$ and $\alpha_0$ is a positive number. The $\tilde{\alpha} \in L^\infty(\mathcal{E}^{\gamma_2})$ is also a penalty function that is defined by 
\begin{align*}
\tilde{\alpha}|_{P^\star}=
\begin{cases}
\tilde{\alpha}_0k^2h_\star^{-1} \quad \forall P^\star \in \mathcal{V}^\circ\\
\tilde{\alpha}_0k^2h_e^{-1} \quad \forall P^\star \in \mathcal{V}^D\\
\end{cases}
\end{align*}
where $h_\star=\min(|e_1|,|e_2|)$, the edges $e_1$ and $e_2$ share a common vertex $P^\star$ on $\gamma_2$.
Let $(\mathbf{u},\mathbf{v})_{\mathcal{T}_h}=\sum_{T\in \mathcal{T}_h}\int_T \mathbf{u} \cdot \mathbf{v} \, dxdy$ and $\la \mathbf{u}, \mathbf{v}\ra_{\tilde{\mathcal{E}}}=\sum_{e\in \tilde{\mathcal{E}}}\int_e \mathbf{u}\cdot \mathbf{v} \, ds$ for some $\tilde{\mathcal{E}}\subset \mathcal{E}\cup \mathcal{E}^D\cup \mathcal{E}^N$, and $[u_{P_{\pm}^\star },v_{P_{\pm}^\star }]_{\tilde{\mathcal{V}}}=\sum_{P^\star \in \tilde{\mathcal{V}}}u(P_{\pm}^\star)v(P_{\pm}^\star)$ for some $\tilde{\mathcal{V}}\subset \mathcal{V}$.



\begin{rem}
Compared with the original IPDG scheme for elliptic equations \cite{riviere2008discontinuous}, we drop the penalty term $\la \alpha \jl p_h\jr, \jl \xi\jr \ra_{\mathcal{E}^{\gamma_1}}$ on the barrier interfaces ${\mathcal{E}^{\gamma_1}}$ to accommodate pressure discontinuities. 
Furthermore, we replace the flux term $\la \{\mathbf{K}_m\nabla p_h\},\jl\xi\jr \ra_{\mathcal{E}^{\gamma_1}}$ 
with $\la\frac{k_b}{a}\jl p_h\jr,\jl\xi\jr \ra_{\mathcal{E}^{\gamma_1}}$
to account for the interface condition of the barriers.
For the conductive fracture interfaces $\mathcal{E}^{\gamma_2}$, we replace the flux term $\la \jl \mathbf{K}_m\nabla p_h \jr, \{\xi\}\ra_{\mathcal{E}^{\gamma_2}}$ with $\la  \frac{\partial^2 p_h}{\partial \boldsymbol{\nu}_2^2}, \{\xi\}\ra_{\mathcal{E}^{\gamma_2}}$ to account for the interface condition of the fractures. 
Then we apply the integration by parts following the original one-dimensional IPDG scheme on $\gamma_2$ to obtain the bilinear form $b_h(\cdot,\cdot)$.
\end{rem}

\begin{rem}
In the above setup, we assume the fracture and barrier interfaces split the entire domain $\Omega$ into separate pieces, as shown in Figure \ref{fig:hybridDomain}. 
It is easy to adjust the scheme to model interfaces immersed in $\Omega$ by removing the terms evaluated at the tips of the fracture in \eqref{eq:b} and \eqref{eq:G}. 
When the computational domain contains a fracture (barrier) network, each fracture (barrier) is accounted for independently, and no special condition is imposed at intersections. 
When a conductive fracture and blocking barrier intersect, one should either treat the fracture as two separate, non-connecting pieces split by the barrier or treat the barrier as two separate, non-connecting pieces split by the fracture, depending on the geological reality.
See Example \ref{ex:complex} in the numerical section for additional illustrations.
\end{rem}

\begin{rem}
The enriched Galerkin (EG) spaces can be adopted in the variational formulation \eqref{eq:scheme} to reduce the degrees of freedom. We define the EG spaces as
\begin{equation*}
V_{h,k}^{EG}=V_{h,k}^{C}+V_{h,0}^{DG}\subset V_{h,k}^{DG},
\end{equation*}
where $V_{h,k}^{C}:=\{v\in C(\Omega\setminus\gamma_1): v|_{T} \in P^k(T), \quad \forall T  \in \mathcal{T}_h \}$. Here, $C(\Omega\setminus\gamma_1)$ denotes the space of continuous functions on $\Omega\setminus\gamma_1$, which are allowed to have discontinuities only on $\gamma_1$. The EG scheme is then obtained simply by replacing the trial and test function spaces $V_{h,k}^{DG}$ with $V_{h,k}^{EG}$ in \eqref{eq:scheme}.
\end{rem}

\section{Stability analysis and error estimates}\label{Sect:theory}
In this section, we provide the stability analysis and error estimates for the IPDG scheme \eqref{eq:scheme}. 
We define the broken Sobolev spaces, for $s\geq1$,
\begin{align*}
    &H^s(\mathcal{T}_h):=\{v\in L^2(\Omega):v|_{T}\in H^s(T) \,, \forall T\in \mathcal{T}_h\},\\
    &H^s(\cup_{i=1}^3\Omega_i):=\{v\in L^2(\Omega):v|_{\Omega_i}\in H^s(\Omega_i),\, i=1,2,3\},\\
    &W^{s,\infty}(\cup_{i=1}^3\Omega_i):=\{v\in L^2(\Omega):v|_{\Omega_i}\in W^{s,\infty}(\Omega_i),\, i=1,2,3\}.
\end{align*}
Clearly, we have $W^{s,\infty}(\cup_{i=1}^3\Omega_i) \subset H^s(\cup_{i=1}^3\Omega_i)\subset H^s(\mathcal{T}_h)$. 
For any $v\in H^1(\mathcal{T}_h)$, we define the DG norm associated with the bilinear forms \eqref{eq:a} and \eqref{eq:b} as follows
\begin{align}\label{eq:DGnorm}
    \|v\|_{DG}^2:=&\|\nabla v\|_{\mathcal{T}_h}^2 +\|\alpha^{\frac 12}\jl v \jr \|_{\mathcal{E}^0\cup \mathcal{E}^D\cup \mathcal{E}^{\gamma_2}}^2+\|(\frac{k_b}{a})^{\frac 12} \jl v \jr\|_{\mathcal{E}^{\gamma_1}}^2\nonumber\\
    &+\|\frac{\partial v^-}{\partial \boldsymbol{\nu}_2 }\|_{\mathcal{E}^{\gamma_2}}^2+\|\frac{\partial v^+}{\partial \boldsymbol{\nu}_2 }\|_{\mathcal{E}^{\gamma_2}}^2+|\![\tilde{\alpha}^{\frac 12}\jl v \jr _{P_{-}^\star}]\!|_{\mathcal{V}^\circ \cup \mathcal{V}^D}^2+|\![\tilde{\alpha}^{\frac 12}\jl v\jr _{P_{+}^\star}]\!|_{\mathcal{V}^\circ \cup \mathcal{V}^D}^2,
\end{align}
where $\|\mathbf{v}\|_{\mathcal{T}_h}^2=\sum_{T\in \mathcal{T}_h}\int_T \mathbf{v}\cdot \mathbf{v}\, dxdy$, $\| v\|_{\tilde{\mathcal{E}}}^2:=\la v,v\ra_{\tilde{\mathcal{E}}}$ for some $\tilde{\mathcal{E}}\subset \mathcal{E}\cup \mathcal{E}^D\cup \mathcal{E}^N$, and $\jpl \tilde{\alpha}^{\frac 12} \jl v\jr_{P_{\pm}^\star}\jpr_{\tilde{\mathcal{V}}}^2:=[\tilde \alpha \jl v\jr_{P_{\pm}^\star},\jl v\jr_{P_{\pm}^\star}]_{\tilde{\mathcal{V}}}$ for some $\tilde{\mathcal{V}}\subset \mathcal{V}$. 

We first recall the classical {\em hp} inverse inequality in \cite[Theorem 3.92, Theorem 4.76]{Schwab1998hp}
\begin{lem}\label{lem:inv}
Let $I=(c,d)$ be a bounded interval and $h_I=d-c$. Then for every $v\in P^k(I)$ it holds that
\begin{align*}
    &\|v'\|_{L^2(I)}\leq C k^2 h_I^{-1}\|v\|_{L^2(I)}, \\
    &\|v\|_{L^\infty} \leq C k h_I^{-1/2} \|v\|_{L^2(I)}.
\end{align*}
    Let $T\in \mathcal{T}_h$. Then there exists a constant $C$ independent of $h_T$ and $k$ such that
    \begin{align*}
        &\|\nabla v\|_{L^2(T)}\leq C k^2 h_T^{-1} \|v\|_{L^2(T)} \quad \forall v \in P^k(T),\\
        &\|v\|_{L^2(\partial T)}\leq C k h_T^{-1/2} \|v\|_{L^2(T)} \quad \forall v \in P^k(T).
    \end{align*}
\end{lem}

We have the following stability result.
\begin{thm}
    Let $u,v \in V_{h,k}^{DG}(\Omega)$, and consider the bilinear forms $a_h(\cdot,\cdot)$ and $b_h(\cdot,\cdot)$ as defined in \eqref{eq:a}-\eqref{eq:b}. Assuming that $\theta \mathbf{I}\leq \mathbf{K}_m\leq \Theta \mathbf{I}$ { and $\theta\leq ak_f\leq\Theta$} for some $\theta, \Theta >0$ and that $\alpha_0, \tilde{\alpha}_0$ are sufficiently large, we then have
    \begin{align}
    & |a_h(u,v)+b_h(u,v)|\leq C_1\|u\|_{DG}\|v\|_{DG},\label{neq:cont}\\
    & a_h(u,u)+b_h(u,u) \geq C_2\|u\|_{DG}^2,\label{neq:coer}
    \end{align}
    where constants $C_1$, $C_2$ are independent of $k$ and $h$. 
\end{thm}
\begin{proof}
    By using the Cauchy-Schwarz inequality and inverse inequalities in Lemma \ref{lem:inv}, we have
    \begin{align*}
        |a_h(u,v)|\leq &C\|\nabla u\|_{\mathcal{T}_h} \|\nabla v\|_{\mathcal{T}_h} +C \|\nabla u\|_{\mathcal{T}_h}\|k h_e^{-1/2}\jl v \jr \|_{\mathcal{E}^0\cup \mathcal{E}^D\cup \mathcal{E}^{\gamma_2}}\\
        &+C \|\nabla v\|_{\mathcal{T}_h}\|k h_e^{-1/2}\jl u \jr \|_{\mathcal{E}^0\cup \mathcal{E}^D\cup \mathcal{E}^{\gamma_2}}+\|\alpha^{\frac 12}\jl u \jr \|_{\mathcal{E}^0\cup \mathcal{E}^D\cup \mathcal{E}^{\gamma_2}}\|\alpha^{\frac 12}\jl v \jr \|_{\mathcal{E}^0\cup \mathcal{E}^D\cup \mathcal{E}^{\gamma_2}}\\
        &+\|(\frac{k_b}{a})^{\frac 12} \jl u \jr \|_{\mathcal{E}^{\gamma_1}}\|(\frac{k_b}{a})^{\frac 12} \jl v \jr \|_{\mathcal{E}^{\gamma_1}}\\
        \leq & { \frac{C_1}{2}} \|u\|_{DG}\|v\|_{DG}.
    \end{align*}
For \eqref{neq:coer}, again by using inverse inequalities in Lemma \ref{lem:inv}, we have
\begin{align*}
    a_h(u,u)\geq& C\|\nabla u\|_{\mathcal{T}_h}^2 -C \|\nabla u\|_{\mathcal{T}_h}\|k h_e^{-1/2}\jl u \jr\|_{\mathcal{E}^0\cup \mathcal{E}^D\cup \mathcal{E}^{\gamma_2}}+\|\alpha^{\frac 12}\jl u \jr\|_{\mathcal{E}^0\cup \mathcal{E}^D\cup \mathcal{E}^{\gamma_2}}^2+\|(\frac{k_b}{a})^{\frac 12} \jl u \jr \|_{\mathcal{E}^{\gamma_1}}^2\\
    \geq & \frac{C}{2}\|\nabla u\|_{\mathcal{T}_h}^2+(\alpha_0-\frac{C}{2})\|kh_e^{-1/2}\jl u \jr \|_{\mathcal{E}^0\cup \mathcal{E}^D\cup \mathcal{E}^{\gamma_2}}^2 +\|(\frac{k_b}{a})^{\frac 12} \jl u \jr \|_{\mathcal{E}^{\gamma_1}}^2\\
    \geq & \min(\frac{C}{2},\frac 12) (\|\nabla u\|_{\mathcal{T}_h}^2 +\|\alpha^{\frac 12}\jl u \jr \|_{\mathcal{E}^0\cup \mathcal{E}^D\cup \mathcal{E}^{\gamma_2}}^2+\|(\frac{k_b}{a})^{\frac 12} \jl u \jr\|_{\mathcal{E}^{\gamma_1}}^2),
\end{align*}
where we choose the parameter $\alpha_0$ sufficiently large such that $\alpha_0>C$. Again by using the Cauchy-Schwarz inequality and inverse inequalities in Lemma \ref{lem:inv}, we have
\begin{align*}
    |b_h(u,v)|\leq &C \|\frac{\partial u^-}{\partial \boldsymbol{\nu}_2 }\|_{\mathcal{E}^{\gamma_2}}\|\frac{\partial v^-}{\partial \boldsymbol{\nu}_2 }\|_{\mathcal{E}^{\gamma_2}}+C\|\frac{\partial u^+}{\partial \boldsymbol{\nu}_2 }\|_{\mathcal{E}^{\gamma_2}}\|\frac{\partial v^+}{\partial \boldsymbol{\nu}_2 }\|_{\mathcal{E}^{\gamma_2}}\\
    &+C\|\frac{\partial u^-}{\partial \boldsymbol{\nu}_2 }\|_{\mathcal{E}^{\gamma_2}}|\![kh^{-1/2}\jl v \jr _{P_{-}^\star}]\!|_{\mathcal{V}^\circ \cup \mathcal{V}^D}+C\|\frac{\partial u^+}{\partial \boldsymbol{\nu}_2 }\|_{\mathcal{E}^{\gamma_2}}|\![kh^{-1/2}\jl v \jr _{P_{+}^\star}]\!|_{\mathcal{V}^\circ \cup \mathcal{V}^D}\\
    &+C\|\frac{\partial v^-}{\partial \boldsymbol{\nu}_2 }\|_{\mathcal{E}^{\gamma_2}}|\![kh^{-1/2}\jl u \jr _{P_{-}^\star}]\!|_{\mathcal{V}^\circ \cup \mathcal{V}^D}+C\|\frac{\partial v^+}{\partial \boldsymbol{\nu}_2 }\|_{\mathcal{E}^{\gamma_2}}|\![kh^{-1/2}\jl u \jr _{P_{+}^\star}]\!|_{\mathcal{V}^\circ \cup \mathcal{V}^D}\\
    &+|\![\tilde{\alpha}^{1/2}\jl u \jr _{P_{-}^\star}]\!|_{\mathcal{V}^\circ \cup \mathcal{V}^D}|\![\tilde{\alpha}^{1/2}\jl v \jr _{P_{-}^\star}]\!|_{\mathcal{V}^\circ \cup \mathcal{V}^D}+|\![\tilde{\alpha}^{1/2}\jl u \jr _{P_{+}^\star}]\!|_{\mathcal{V}^\circ \cup \mathcal{V}^D}|\![\tilde{\alpha}^{1/2}\jl v \jr _{P_{+}^\star}]\!|_{\mathcal{V}^\circ \cup \mathcal{V}^D}\\
    \leq&{ \frac{C_1}{2}}\|u\|_{DG}\|v\|_{DG}.
\end{align*}
Similar to derive the coercivity of $a_h(\cdot,\cdot)$, we have
\begin{align*}
    b_h(u,u)\geq C_2(\|\frac{\partial u^-}{\partial \boldsymbol{\nu}_2 }\|_{\mathcal{E}^{\gamma_2}}^2+\|\frac{\partial u^+}{\partial \boldsymbol{\nu}_2 }\|_{\mathcal{E}^{\gamma_2}}^2+|\![\tilde{\alpha}^{\frac 12}\jl u \jr _{P_{-}^\star}]\!|_{\mathcal{V}^\circ \cup \mathcal{V}^D}^2+|\![\tilde{\alpha}^{\frac 12}\jl u\jr _{P_{+}^\star}]\!|_{\mathcal{V}^\circ \cup \mathcal{V}^D}^2),
\end{align*}
where we also require the parameter $\tilde{\alpha}_0$ sufficiently large. Thus, we have
\begin{align*}
    a_h(u,u)+b_h(u,u) \geq C_2\|u\|_{ DG}^2.
\end{align*}
This completes the proof.
\end{proof}

We now prove error estimates for the IPDG method by using the properties of consistency, boundedness, stability, and approximation of suitable interpolations. We first recall the classical {\em hp} interpolation \cite[Lemma 4.5]{Babuska1987}, \cite[Lemma B.3]{Melenk2010MC}.
\begin{lem}\label{lem:hp_inter}
    Let $s>1$ and $T\in \mathcal{T}_h$. There exist hp-interpolation operator $\pi_T^{hp}:H^s(T) \rightarrow P^k(T)$ such that for $0\leq t \leq s$, $v\in H^s(T)$,
    \begin{align}\label{eq:hperr1}
        \|v-\pi_T^{hp}(v)\|_{H^t(T)}\leq C\frac{h_T^{\nu-t}}{k^{s-t}} \|v\|_{H^s(T)},
    \end{align}
    where $\nu=\min(k+1,s)$. Moreover, if $s>\frac 32$, for $t=0,1$, we have
    \begin{align}\label{eq:hperr2}
        &\|D^t(v-\pi_T^{hp}(v)\|_{L^2(\partial T)}\leq C \frac{h_T^{\nu -t-1/2}}{k^{s-t-1/2}}\|v\|_{H^s(T)},\\
        &\|D^t(v-\pi_T^{hp}(v))\|_{L^{\infty}(T)}\leq C\frac{h_T^{\nu-t-1}}{k^{s-t-1}} \|v\|_{H^s(T)}, \label{eq:hperrLinf}
    \end{align}
    Here the constant $C$ is independent of $k$ and $h_T$.
\end{lem}
Then we define the interpolation operator $\Pi_{hp}:\, H^s(\cup_{i=1}^3\Omega_i)\rightarrow V_{h,k}^{DG}(\Omega)$, $\Pi_{hp} v|_T=\pi_T^{hp}(v)$. By the definition of the DG norm \eqref{eq:DGnorm} and Lemma \ref{lem:hp_inter}, we have
\begin{align}\label{err:apro}
    \|v-\Pi_{hp}v\|_{DG}\leq C \frac{h^{\nu-1}}{k^{s-\frac 32}}\|v\|_{H^s(\cup_{i=1}^3\Omega_i)} \quad \forall v \in H^s(\cup_{i=1}^3\Omega_i).
\end{align}
The following theorem is our main result in this section.
\begin{thm}
    Let $p\in W^{s,\infty}(\cup_{i=1}^3\Omega_i),\, s>2,$ be the exact solution of the problem \eqref{eq:PDEmodel} and assume that $\theta \mathbf{I}\leq \mathbf{K}_m\leq \Theta \mathbf{I}$ { and $\theta\leq ak_f\leq\Theta$} for some $\theta, \, \Theta >0$. The numerical solution $p_h$ of the IPDG scheme \eqref{eq:scheme} satisfies
    \begin{align}\label{err:energy}
        \|p-p_h\|_{DG}\leq C\frac{h^{\nu-1}}{k^{s-2}}\|p\|_{W^{s,\infty}(\cup_{i=1}^3\Omega_i)},
    \end{align}
    where $\nu=\min(k+1,s)$ and the constant $C$ is independent of $k$ and $h$. Furthermore, the following optimal order $L^2$-error estimates of SIPG methods hold
    \begin{align}\label{err:L2}
    \|p-p_h\|_{L^2(\Omega)}\leq C\frac{h^{\nu}}{k^{s-2}}\|p\|_{W^{s,\infty}(\cup_{i=1}^3\Omega_i)}.
    \end{align}
\end{thm}
\begin{proof}
    By the consistency of the IPDG scheme, we have the Galerkin orthogonality:
    \begin{align}\label{eq:ortho}
        a_h(p-p_h,\xi)+b_h(p-p_h,\xi)=0 \quad \forall \xi \in V_{h,k}^{DG}(\Omega).
    \end{align}
    We rewrite $p-p_h=(p-\Pi_{hp}p)-(p_h-\Pi_{hp}p):=\eta-\xi$, by using the coercivity \eqref{neq:coer} and the error estimate of the interpolation \eqref{err:apro}, we have
    \begin{align*}
        \|\xi\|_{DG}^2\leq& C (a_h(\xi,\xi)+b_h(\xi,\xi))\\
        =&C(a_h(\eta,\xi)+b_h(\eta,\xi))
    \end{align*}
Let's first estimate $a_h(\eta,\xi)$, by using the error estimate of the interpolation and inverse inequality for $\xi$, we have
    \begin{align*}
        a_h(\eta,\xi)\leq & C\frac{h^{\nu-1}}{k^{s-1}}\|p\|_{H^s(\cup_{i=1}^3\Omega_i)}\|\nabla \xi\|_{\mathcal{T}_h}+C\frac{h^{\nu-1}}{k^{s-1/2}}\|p\|_{H^s(\cup_{i=1}^3\Omega_i)}\|\alpha^{1/2}\jl \xi \jr\|_{\mathcal{E}^0\cup \mathcal{E}^D\cup \mathcal{E}^{\gamma_2}}\\
        &+C\frac{h^{\nu-1}}{k^{s-3/2}}\|p\|_{H^s(\cup_{i=1}^3\Omega_i)}\|\nabla \xi\|_{\mathcal{T}_h}+C\frac{h^{\nu-1/2}}{k^{s-1/2}}\|p\|_{H^s(\cup_{i=1}^3\Omega_i)}\|(\frac{k_b}{a})^{1/2}\jl \xi\jr\|_{\mathcal{E}^{\gamma_1}}\\
        \leq& C\frac{h^{\nu-1}}{k^{s- 3/2}}\|p\|_{H^s(\cup_{i=1}^3\Omega_i)}\|\xi\|_{DG},
    \end{align*}
For $b_h(\eta,\xi)$, we have the following estimate
\begin{align*}
    b_h(\eta,\xi)=I+II+III,
 \end{align*}   
 where
 \begin{align*}
   I&=\frac 12  \la ak_f \frac{\partial \eta^-}{\partial \boldsymbol{\nu}_2}, \frac{\partial \xi^-}{\partial \boldsymbol{\nu}_2}\ra_{\mathcal{E}^{\gamma_2}}+\frac 12 \la  ak_f \frac{\partial \eta^+}{\partial \boldsymbol{\nu}_2}, \frac{\partial \xi^+}{\partial \boldsymbol{\nu}_2}\ra_{\mathcal{E}^{\gamma_2}}\\
   &\leq C\sum_{T\in \mathcal{T}_h^{\gamma_2,-}} \frac{h^{\nu-3/2}}{k^{s- 3/2}}\|p\|_{H^s(T)}\|\frac{\partial \xi^-}{\partial \boldsymbol{\nu}_2}\|_{L^2(\partial T\cap \gamma_2)}+C\sum_{T\in \mathcal{T}_h^{\gamma_2,+}} \frac{h^{\nu-3/2}}{k^{s- 3/2}}\|p\|_{H^s(T)}\|\frac{\partial \xi^+}{\partial \boldsymbol{\nu}_2}\|_{L^2(\partial T\cap \gamma_2)}\\
   &\leq C\frac{h^{\nu-1}}{k^{s- 3/2}}\|p\|_{W^{s,\infty}(\Omega_{\gamma_2}^{-}\cup \Omega_{\gamma_2}^+)}(\|\frac{\partial \xi^-}{\partial \boldsymbol{\nu}_2}\|_{L^2(\mathcal{E}^{\gamma_2})}+\|\frac{\partial \xi^+}{\partial \boldsymbol{\nu}_2}\|_{L^2(\mathcal{E}^{\gamma_2})}).
\end{align*}
In the last inequality, we used the $\|p\|_{H^s(T)}\leq C h\|p\|_{W^{s,\infty}(T)}$ and the Cauchy-Schwarz inequality.
For $II$, we use the inverse inequality in Lemma \ref{lem:inv} for $\frac{\partial \xi^\pm}{\partial \boldsymbol{\nu}_2}$ and the estimate of $\eta$ to get
\begin{align*}
    II=&-\frac 12[ak_f\{\frac{\partial \eta}{\partial \boldsymbol{\nu}_2}\}_{P_{-}^{\star}}, \jl \xi \jr_{P_{-}^\star}]_{\mathcal{V}^\circ \cup \mathcal{V}^D} -\frac 12[ak_f\{\frac{\partial \eta}{\partial \boldsymbol{\nu}_2}\}_{P_{+}^{\star}}, \jl \xi \jr_{P_{+}^\star}]_{\mathcal{V}^\circ \cup \mathcal{V}^D} \\
    &+\frac \sigma 2[ak_f\{\frac{\partial \xi }{\partial \boldsymbol{\nu}_2}\}_{P_{-}^{\star}}, \jl \eta \jr_{P_{-}^\star}]_{\mathcal{V}^\circ \cup \mathcal{V}^D} +\frac \sigma 2[ak_f\{\frac{\partial \xi}{\partial \boldsymbol{\nu}_2}\}_{P_{+}^{\star}}, \jl \eta \jr_{P_{+}^\star}]_{\mathcal{V}^\circ \cup \mathcal{V}^D} \\
    &+C\frac{h^{\nu-3/2}}{k^{s- 2}}\|p\|_{H^s(\mathcal{T}_h^{\gamma_2,-})}\|\frac{\partial \xi^-}{\partial \boldsymbol{\nu}_2}\|_{L^2(\mathcal{E}^{\gamma_2})})+C\frac{h^{\nu-3/2}}{k^{s- 2}}\|p\|_{H^s(\mathcal{T}_h^{\gamma_2,+})}\|\frac{\partial \xi^+}{\partial \boldsymbol{\nu}_2}\|_{L^2(\mathcal{E}^{\gamma_2})})\\
    &\leq C\frac{h^{\nu-1}}{k^{s- 2}}\|p\|_{W^{s,\infty}(\Omega_{\gamma_2}^{-}\cup \Omega_{\gamma_2}^+)}(\jpl\tilde{\alpha}^{1/2} \jl \xi \jr _{P_{-}^\star} \jpr_{\mathcal{V}^\circ \cup \mathcal{V}^D}+\jpl\tilde{\alpha}^{1/2} \jl \xi \jr _{P_{+}^\star} \jpr_{\mathcal{V}^\circ \cup \mathcal{V}^D}+\|\frac{\partial \xi^-}{\partial \boldsymbol{\nu}_2}\|_{L^2(\mathcal{E}^{\gamma_2})})+\|\frac{\partial \xi^+}{\partial \boldsymbol{\nu}_2}\|_{L^2(\mathcal{E}^{\gamma_2})})).
\end{align*}
Finally, we estimate $III$ by using the Cauchy-Schwarz inequality and the error estimation of the interpolation to obtain
\begin{align*}
    III=& [\tilde{\alpha}\jl \eta \jr_{P_{-}^\star},\jl \xi \jr_{P_{-}^\star}]_{\mathcal{V}^\circ \cup \mathcal{V}^D} +[\tilde{\alpha}\jl \eta \jr_{P_{+}^\star},\jl \xi \jr_{P_{+}^\star}]_{\mathcal{V}^\circ \cup \mathcal{V}^D}\\
    \leq& C\frac{h^{\nu-3/2}}{k^{s- 2}}\|p\|_{H^s(\mathcal{T}_h^{\gamma_2,-})}\jpl\tilde{\alpha}^{1/2} \jl \xi \jr _{P_{-}^\star} \jpr_{\mathcal{V}^\circ \cup\mathcal{V}^D}+C\frac{h^{\nu-3/2}}{k^{s- 2}}\|p\|_{H^s(\mathcal{T}_h^{\gamma_2,+})}\jpl\tilde{\alpha}^{1/2} \jl \xi \jr _{P_{+}^\star}\jpr_{\mathcal{V}^\circ \cup\mathcal{V}^D}\\
    \leq &C\frac{h^{\nu-1}}{k^{s- 2}}\|p\|_{W^{s,\infty}(\Omega_{\gamma_2}^{-}\cup \Omega_{\gamma_2}^+)}(\jpl\tilde{\alpha}^{1/2} \jl \xi \jr _{P_{-}^\star} \jpr_{\mathcal{V}^\circ \cup\mathcal{V}^D}+\jpl\tilde{\alpha}^{1/2} \jl \xi \jr _{P_{+}^\star} \jpr_{\mathcal{V}^\circ \cup\mathcal{V}^D}).
\end{align*}
Therefore, we have
\begin{align*}
    \|\xi\|_{DG}\leq C \frac{h^{\nu-1}}{k^{s- 2}}\|p\|_{W^{s,\infty}(\cup_{i=1}^{3}\Omega_i)}
\end{align*}
This completes the proof of \eqref{err:energy} by the triangle inequality and \eqref{err:apro}. For the $L^2$-error estimates of SIPG methods, we use the standard duality argument. We define the auxiliary function $\psi$ as the solution of the adjoint problem
\begin{equation}\label{eq:adjointmodel}
\begin{cases}
        -\nabla \cdot(\mathbf{K}_m \nabla \psi)=p-p_h & \text{ in } \Omega \setminus (\gamma_1\cup \gamma_2),\\
        \mathbf{K}_m^-\nabla \psi^- \cdot \mathbf{n}^-+\mathbf{K}_m^+\nabla \psi^+ \cdot \mathbf{n}^+=0 &  \text{ on } \gamma_1,\\
        \mathbf{K}_m^-\nabla \psi^- \cdot \mathbf{n}^-=k_b\frac{\psi^+-\psi^-}{a} &  \text{ on } \gamma_1,\\
        \psi^--\psi^+=0 & \text{ on } \gamma_2,\\
         \mathbf{K}_m^-\nabla \psi^- \cdot \mathbf{n}^-+\mathbf{K}_m^+\nabla \psi^+ \cdot \mathbf{n}^+=ak_f\frac{\partial^2 \psi^+} {\partial \boldsymbol{\nu}_2^2} & \text{ on } \gamma_2, \\
        \psi =0 & \text{ on } \partial \Omega 
\end{cases}
\end{equation}
    Since the $a_h(\cdot,\cdot)$ satisfies the adjoint consistency condition, that is
    \begin{align}\label{eq:adj}
        a_h(v,\psi)+b_h(v,\psi)=(p-p_h,v)_{\mathcal{T}_h} \quad \forall v\in H^2(\mathcal{T}_h),
    \end{align}
    we take $\psi_I\in V_{h,1}^{DG}(\Omega)$ to be a piecewise linear interpolant of $\psi$. Then, we choose $v=p-p_h$ in \eqref{eq:adj} and use \eqref{neq:cont}, \eqref{err:energy}, and \eqref{eq:ortho} to obtain
    \begin{align*}
        \|p-p_h\|_{0,\Omega}^2&=a_h(p-p_h,\psi)+b_h(p-p_h,\psi)=a_h(p-p_h,\psi-\psi_I)+b_h(p-p_h,\psi-\psi_I)\\
        &\leq C\frac{h^{\nu-1}}{k^{s-2}}\|p\|_{W^{s,\infty}(\cup_{i=1}^3\Omega_i)}h|\psi|_{H^2(\cup_{i=1}^3\Omega_i)}.
    \end{align*}
    Finally, the regularity assumption of the adjoint problem gives $|\psi|_{H^2(\cup_{i=1}^3\Omega_i)}\leq C_r\|p-p_h\|_{L^2(\Omega)}$ with $C_r$ depending only on the domain $\Omega$. Hence, we get the desired optimal error estimate \eqref{err:L2}.
\end{proof}

\begin{rem}
    For the NIPG and IIPG methods, we only obtain the suboptimal order convergence in $L^2$ due to losing the adjoint consistency. However, one can recover the optimal rate of convergence in the $L^2$-norm by using super penalties, refer to \cite{Arnold2002SINUM}. If we assume $\Omega_i,i=1,2,3$ are convex Lipschitz domains, the elliptic regularity theories give us the regularity assumption of the adjoint problem.
\end{rem}

\section{Numerical experiments}\label{Sect:tests}

In this section, we present numerical experiments to test the scheme established in Section \ref{Sect:scheme}. 
The first example is a convergence test to validate the theoretical analysis in Section \ref{Sect:theory}. The $P^k$-DG spaces with $k=1,2,3$ are used in this example. 
The remaining numerical tests are widely adopted benchmarks presented in increasing order of complexity to demonstrate the performance of our method.
The $P^1$-DG space is adopted.
The source terms are set to zero and the flow is driven by the boundary conditions in these examples.
When a reference solution is available for the test, we do not include the numerical results of other discrete fracture models in the presentation for comparison, as our numerical results exhibit excellent agreement with the reference solution. 
That said, we refer the readers to \cite{flemisch2018benchmarks, glaser2022comparison} for a comprehensive comparison with other discrete fracture models.
To save space, all computations are based on the SIPG scheme.

\begin{exmp}\label{ex:convergence}
\textbf{convergence test}

In this example, we test the convergence of our scheme \eqref{eq:scheme} for the interface problem \eqref{eq:PDEmodel} - \eqref{eq:BdrCond} in the domain $\Omega=(0,1)^2$ with an interface  $\gamma=\{(x,y):x=0.5,0\leq y\leq 1\}$.
The permeability of bulk matrix is $\mathbf{K}_m=\mathbf{I}$, where $\mathbf{I}$ is the identity matrix.
Two cases are considered in this example.

In the case (a), we investigate a conductive fracture with aperture $a=10^{-4}$ and permeability $k_f=10^{4}$.
One can verify that the following setting,
\begin{equation}\label{eq:sol_source_fracture}
\begin{cases}
p^{-}=\sin(x)\sin(y),~ q^{-}=2\sin(x)\sin(y), &(x,y)\in\Omega^{-}=(0,\frac12)\times(0,1),\\
p^{+}=p^{-}+\sin(\frac12)(x-\frac12)\sin(y),~ q^{+}=q^{-}+\sin(\frac12)(x-\frac12)\sin(y), &(x,y)\in\Omega^{+}=(\frac12,1)\times(0,1),\\
{ q_f=0,}& { (x,y)\in\gamma=\{\frac12\}\times[0,1]},
\end{cases}
\end{equation}
along with the corresponding Dirichlet boundary conditions, provides a solution to the problem \eqref{eq:PDEmodel} - \eqref{eq:BdrCond} that involves only a conductive fracture interface.

In the case (b), we investigate a blocking barrier with aperture $a=10^{-4}$ and permeability $k_b=10^{-4}$.
One can verify that the following settings,
\begin{equation}\label{eq:sol_source_barrier}
\begin{cases}
p^{-}=\sin(x)\sin(y),~ q^{-}=2\sin(x)\sin(y), &(x,y)\in\Omega^{-}=(0,\frac12)\times(0,1),\\
p^{+}=p^{-}+\cos(\frac12)\sin(y),~ q^{+}=q^{-}+\cos(\frac12)\sin(y), &(x,y)\in\Omega^{+}=(\frac12,1)\times(0,1),
\end{cases}
\end{equation}
along with the corresponding Dirichlet boundary conditions, provides a solution to the problem \eqref{eq:PDEmodel} - \eqref{eq:BdrCond} that involves only a blocking barrier interface.

We perform the computation on triangular grids with different levels of refinement and various orders of DG space. 
See an illustration of the grid with $h=\frac18$ in Figure \ref{fig:grid_converge_test}.
{We take the penalty parameters $\alpha_0=\tilde{\alpha}_0=10$ in case (a), and $\alpha_0=10$ in case (b).}
We present the errors $\|p-p_h\|_{L^2}$, $\|\nabla(p-p_h)\|_{\mathcal{T}_h}$ and $\|p-p_h\|_{DG}$ in Table \ref{tab:accuracy_fracture} and Table \ref{tab:accuracy_barrier} for the case (a) and case (b), respectively. 
From the tables, we can clearly observe an optimal convergence order of $(k+1)$ in the $L^2$-norm, and an optimal convergence order of $k$ in the broken semi-$H^1$-norm and energy norm.

\begin{figure}[!hbpt]
 \centering
 \includegraphics[width=0.45\textwidth]{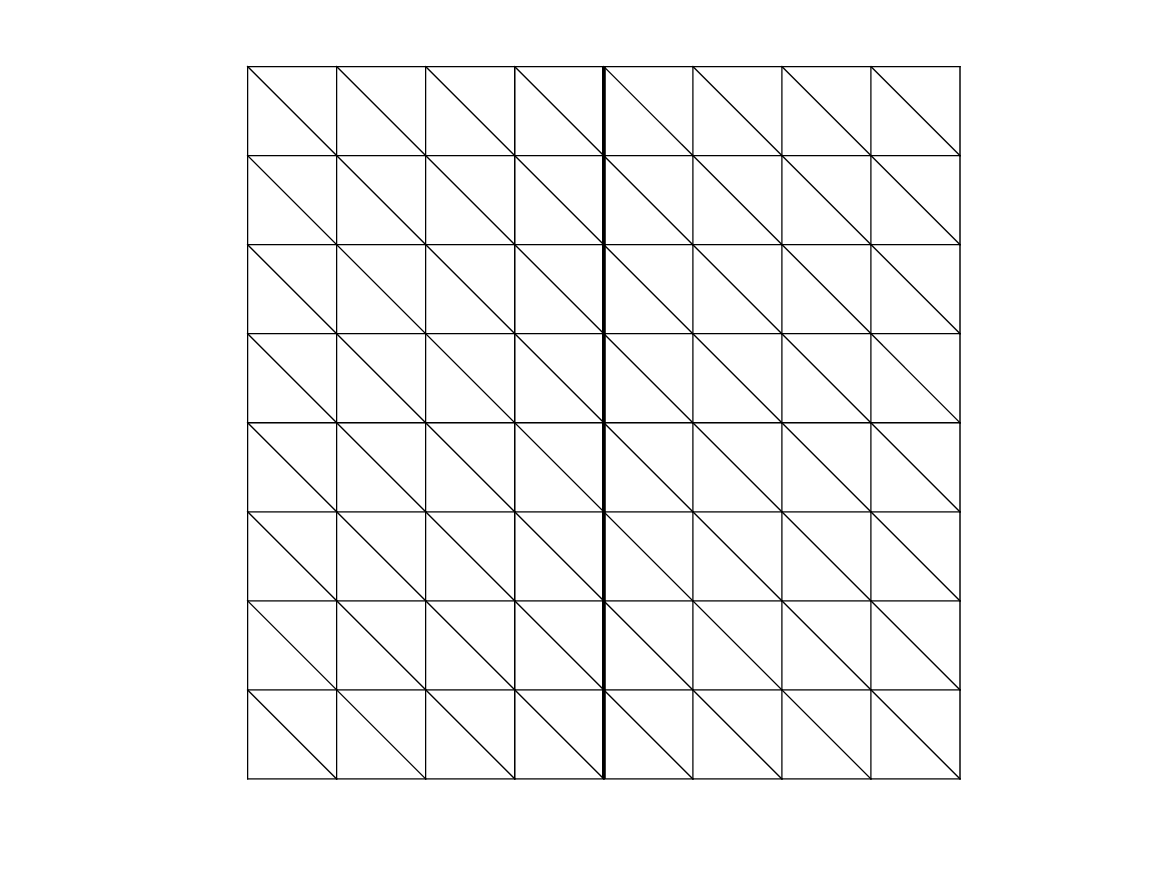}
  \setlength{\abovecaptionskip}{2pt} 
 \caption{\textbf{Example \ref{ex:convergence}: convergence test.} The distribution of the fracture (indicated by the black thick line segment) and the grid with $h=1/8$ used in the computation.}
 \label{fig:grid_converge_test}
\end{figure}

\begin{table}[!htbp]
\centering
\begin{tabular}{ccccccccc}
\toprule[1.5pt]
\multicolumn{8}{c}{$P^1$-SIPG} \\
\cline{2-7} 
$h$ & $\|p-p_h\|_{L^2}$ & Order & $\|\nabla (p-p_h)\|_{\mathcal{T}_h}$ & Order & $\|p-p_h\|_{DG}$ & Order &\\
\midrule
$1/16$ & 3.81E-04  & - & 3.06E-02 & - & 3.66E-02 & -\\
$1/32$ & 9.65E-05 & 1.98 & 1.52E-02 & 1.01 & 1.82E-02 & 1.01 \\
$1/64$ & 2.43E-05 & 1.99 & 7.59E-03 & 1.00 & 9.06E-03 & 1.00 \\
$1/128$ & 6.09E-06 & 2.00 & 3.79E-03 & 1.00 & 4.52E-03 & 1.00 \\
$1/256$ & 1.53E-06 & 2.00 & 1.89E-03 & 1.00 & 2.26E-03 & 1.00 \\
\midrule
\multicolumn{8}{c}{$P^2$-SIPG} \\
\cline{2-7} 
$h$ & $\|p-p_h\|_{L^2}$ & Order & $\|\nabla (p-p_h)\|_{\mathcal{T}_h}$ & Order & $\|p-p_h\|_{DG}$ & Order &\\
\midrule
$1/8$ & 8.64E-06 & - & 6.32E-04 & - & 8.08E-04 & - \\
$1/16$ & 1.09E-06 & 2.99 & 1.59E-04 & 1.99 & 2.00E-04 & 2.01 \\
$1/32$ & 1.37E-07 & 2.99 & 3.98E-05 & 2.00 & 4.98E-05 & 2.01 \\
$1/64$ & 1.72E-08 & 3.00 & 9.97E-06 & 2.00 & 1.24E-05 & 2.00 \\
$1/128$ & 2.15E-09 & 3.00 & 2.49E-06 & 2.00 & 3.10E-06 & 2.00 \\
\midrule
\multicolumn{8}{c}{$P^3$-SIPG} \\
\cline{2-7} 
$h$ & $\|p-p_h\|_{L^2}$ & Order & $\|\nabla (p-p_h)\|_{\mathcal{T}_h}$ & Order & $\|p-p_h\|_{DG}$ & Order &\\
\midrule
$1/4$ & 4.84E-06 & - & 2.00E-04 & - & 2.11E-04 & -\\
$1/8$ & 2.92E-07 & 4.05 & 2.45E-05 & 3.02 & 2.57E-05 & 3.03 \\
$1/16$ & 1.78E-08 & 4.03 & 3.04E-06 & 3.01 & 3.17E-06 & 3.02 \\
$1/32$ & 1.10E-09 & 4.02 & 3.78E-07 & 3.01 & 3.94E-07 & 3.01 \\
$1/64$ & 6.85E-11 & 4.01 & 4.72E-08 & 3.00 & 4.91E-08 & 3.00 \\
\bottomrule[1.5pt]
\end{tabular}
\caption{
\textbf{\textbf{Example \ref{ex:convergence}: convergence test.}} 
Error table of the case (a) that involves a conductive fracture interface.}
\label{tab:accuracy_fracture}
\end{table}

\begin{table}[!htbp]
\centering
\begin{tabular}{ccccccccc}
\toprule[1.5pt]
\multicolumn{8}{c}{$P^1$-SIPG} \\
\cline{2-7} 
$h$ & $\|p-p_h\|_{L^2}$ & Order & $\|\nabla (p-p_h)\|_{\mathcal{T}_h}$ & Order & $\|p-p_h\|_{DG}$ & Order &\\
\midrule
$1/16$ & 3.45E-04 & - & 2.63E-02 & - & 3.08E-02 & -\\
$1/32$ & 8.84E-05 & 1.96 & 1.32E-02 & 1.00 & 1.54E-02 & 1.00 \\
$1/64$ & 2.24E-05 & 1.98 & 6.60E-03 & 1.00 & 7.67E-03 & 1.00 \\
$1/128$ & 5.63E-06 & 1.99 & 3.30E-03 & 1.00 & 3.83E-03 & 1.00 \\
$1/256$ & 1.41E-06 & 2.00 & 1.65E-03 & 1.00 & 1.91E-03 & 1.00 \\
\midrule
\multicolumn{8}{c}{$P^2$-SIPG} \\
\cline{2-7} 
$h$ & $\|p-p_h\|_{L^2}$ & Order & $\|\nabla (p-p_h)\|_{\mathcal{T}_h}$ & Order & $\|p-p_h\|_{DG}$ & Order &\\
\midrule
$1/8$ & 1.07E-05 & - & 7.12E-04 & - & 8.31E-04 & -\\
$1/16$ & 1.36E-06 & 2.98 & 1.80E-04 & 1.98 & 2.05E-04 & 2.00 \\
$1/32$ & 1.71E-07 & 2.99 & 4.53E-05 & 1.99 & 5.10E-05 & 2.00 \\
$1/64$ & 2.15E-08 & 2.99 & 1.14E-05 & 2.00 & 1.27E-05 & 2.00 \\
$1/128$ & 2.69E-09 & 3.00 & 2.84E-06 & 2.00 & 3.17E-06 & 2.00 \\
\midrule
\multicolumn{8}{c}{$P^3$-SIPG} \\
\cline{2-7} 
$h$ & $\|p-p_h\|_{L^2}$ & Order & $\|\nabla (p-p_h)\|_{\mathcal{T}_h}$ & Order & $\|p-p_h\|_{DG}$ & Order &\\
\midrule
$1/4$ & 4.31E-06 & - & 1.76E-04 & - & 1.87E-04 & -\\
$1/8$ & 2.63E-07 & 4.03 & 2.20E-05 & 3.00 & 2.30E-05 & 3.02 \\
$1/16$ & 1.63E-08 & 4.02 & 2.76E-06 & 3.00 & 2.86E-06 & 3.01 \\
$1/32$ & 1.01E-09 & 4.01 & 3.44E-07 & 3.00 & 3.57E-07 &  3.00\\
$1/64$ & 6.29E-11 & 4.00 & 4.31E-08 & 3.00 & 4.46E-08 &  3.00\\
\bottomrule[1.5pt]
\end{tabular}
\caption{
\textbf{\textbf{Example \ref{ex:convergence}}: convergence test.} 
Error table of the case (b) that involves a blocking barrier interface.}
\label{tab:accuracy_barrier}
\end{table}

\end{exmp}

\begin{exmp}\label{ex:single}
\textbf{single fracture}

In this problem, we test the scenario involving a single immersed fracture.
The computational domain is set to $\Omega=(0,1)^2$ and the permeability of the bulk matrix is $\mathbf{K}_m=\mathbf{I}$.
Two cases with different conductivities of the fracture are tested in this example.

In the case (a), we investigate a conductive fracture with the aperture $a=10^{-3}$ and permeability $k_f=10^{8}$.
We consider two different distributions of the fracture: a vertical fracture extending from $(0.5, 0.5)$ to $(0.5, 1)$ and a slanted fracture extending from $(0.25, 0.75)$ to $(0.75, 0.25)$. 
See an illustration of the fractures and the corresponding grids in Figure \ref{fig:single_grid_medium}.
In this case, the top and bottom boundaries are set to Dirichlet conditions with $g_D=1$ and $g_D=0$, respectively, while the left and right boundaries are impermeable, i.e., $g_N=0$. 

In the case (b), we investigate a blocking fracture with the aperture $a=10^{-3}$ and permeability $k_b=10^{-8}$.
The distributions of the fracture and the grids used in computation are the same as in case (a).
In this case, the left and right boundaries are set to Dirichlet conditions with $g_D=0$ and $g_D=1$, respectively, and the top and bottom boundaries are impermeable. 

We perform the computation for both cases on the grids shown in Figure \ref{fig:single_grid_medium}.
{We take the penalty parameters $\alpha_0=\tilde{\alpha}_0=5$ in case (a), and $\alpha_0=10$ in case (b).}
The numerical results are presented in Figure \ref{fig:single_fracture_P1} and Figure \ref{fig:single_barrier_P1} for the cases of conductive fractures and blocking barriers, respectively.
It can be seen from the slices of the pressure that our results match very well with the reference solutions obtained from the box method discrete fracture model (Box-DFM) \cite{xu2024extension} for the interface problem \eqref{eq:PDEmodel} - \eqref{eq:BdrCond} on refined grids.
\begin{figure}[!htbp]
 \centering
 \begin{subfigure}[b]{0.45\textwidth}
  \includegraphics[width=\textwidth]{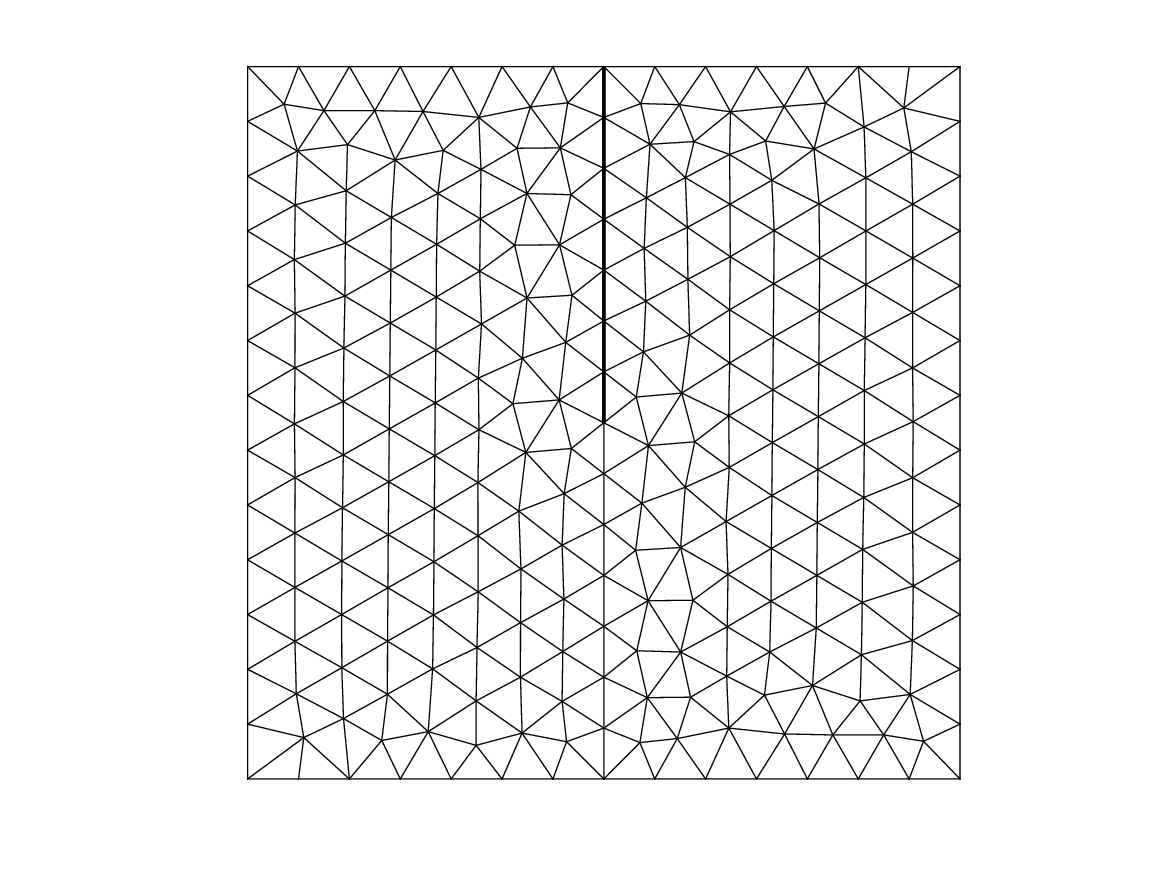}
  \caption{Vertical fracture}
 \end{subfigure}
 \begin{subfigure}[b]{0.45\textwidth}
  \includegraphics[width=\textwidth]{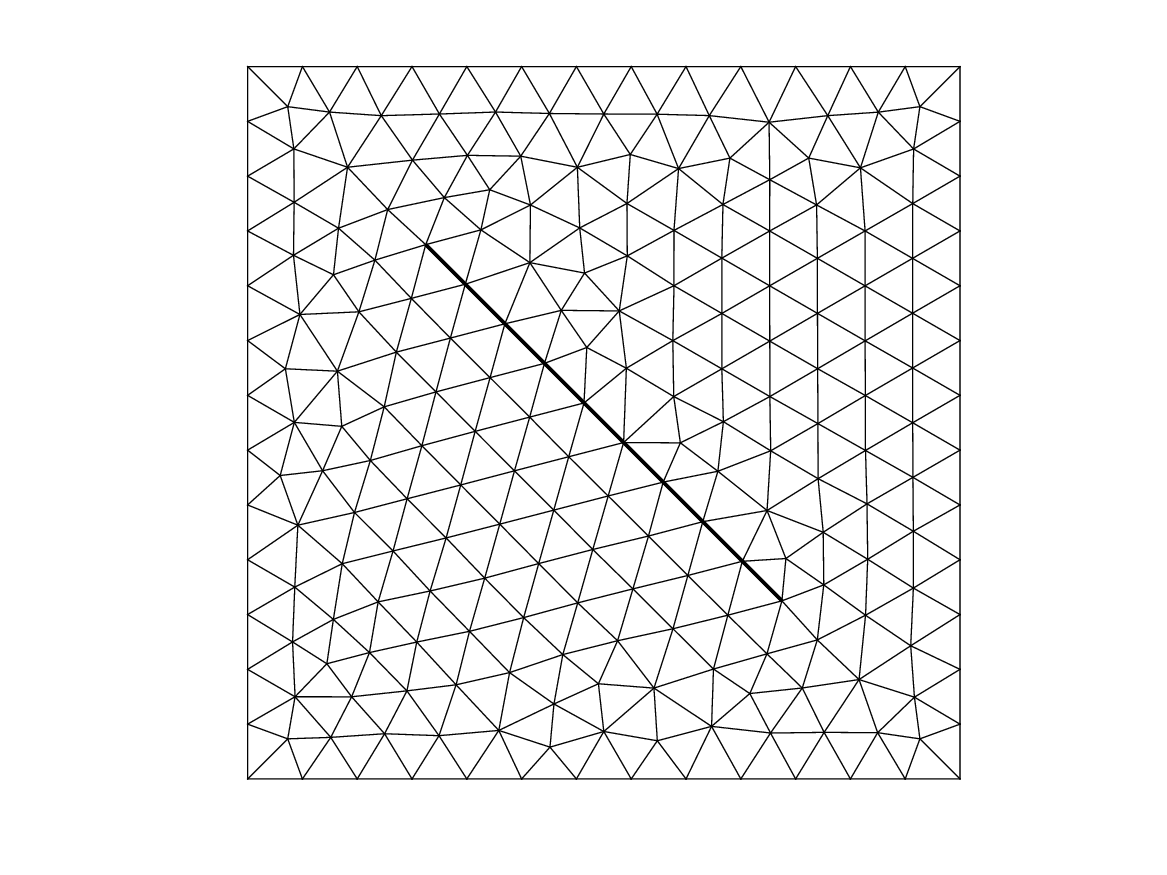}
  \caption{Slanted fracture}
 \end{subfigure}
 \caption{\textbf{Example \ref{ex:single}: single fracture.} 
 The distributions of the fracture (indicated by the black thick line segments) and the grids used in the computation.
 The grid for the vertical fracture contains $450$ cells, and the grid for the slanted fracture contains $404$ cells.}
 \label{fig:single_grid_medium}
\end{figure}

\begin{figure}[!htbp]
 \centering
 \begin{subfigure}[b]{0.45\textwidth}
  \includegraphics[width=\textwidth]{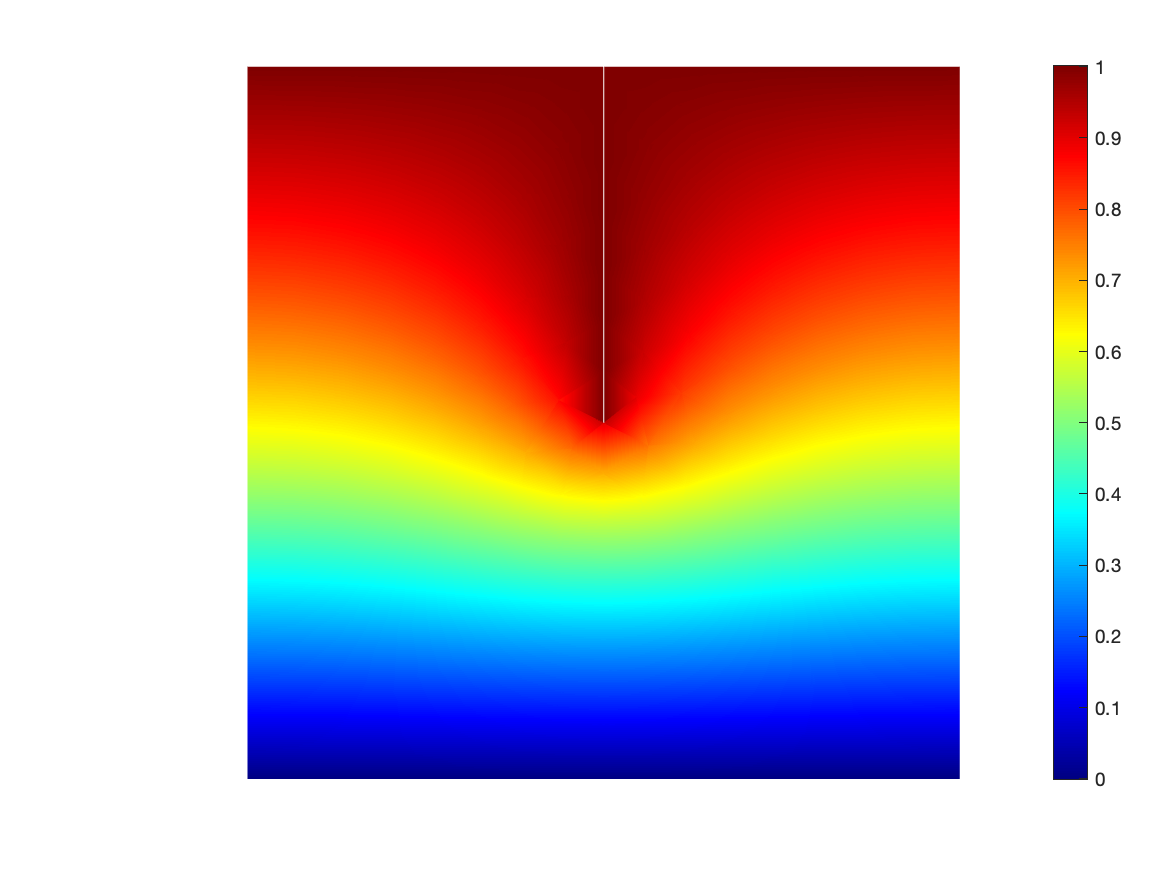}
  \caption{Contour of pressure - $1$}
 \end{subfigure}
 \begin{subfigure}[b]{0.45\textwidth}
  \includegraphics[width=\textwidth]{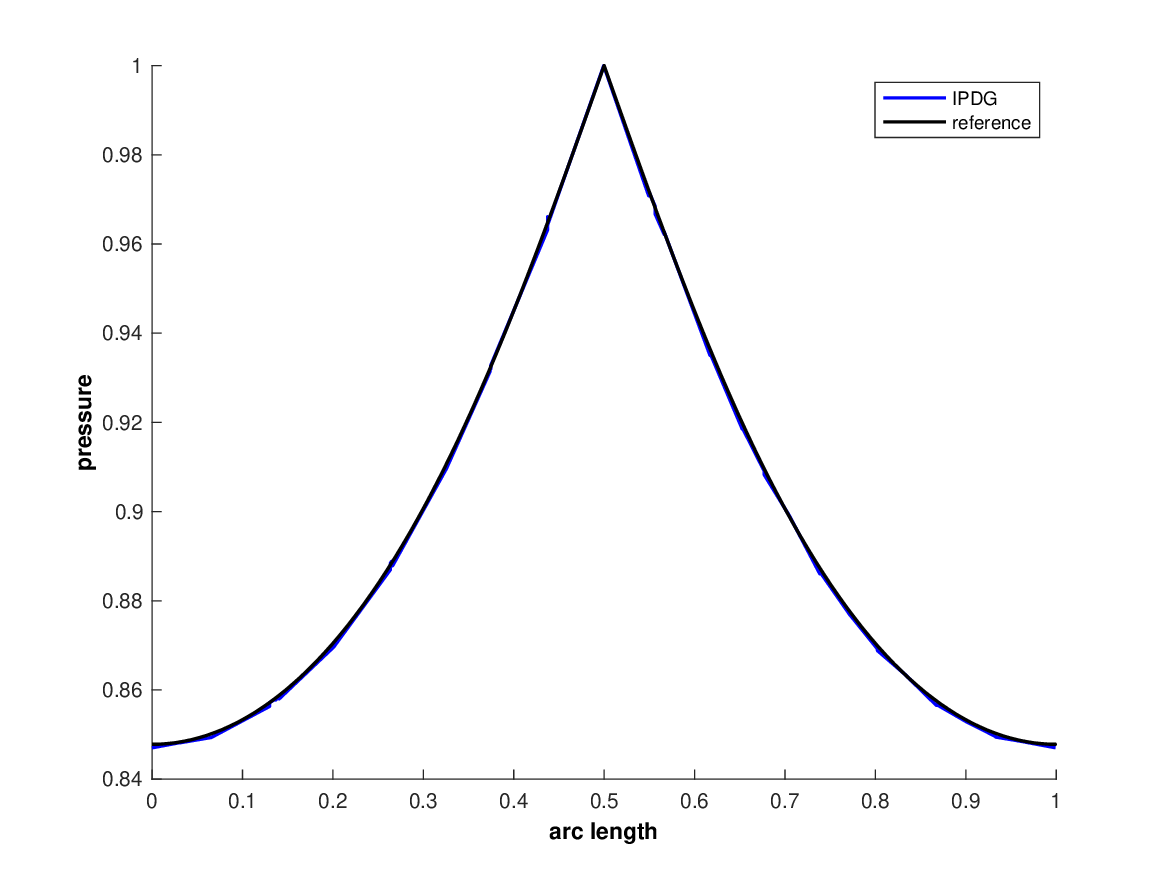}
  \caption{Slice of pressure - $1$}
 \end{subfigure}\\
 \begin{subfigure}[b]{0.45\textwidth}
  \includegraphics[width=\textwidth]{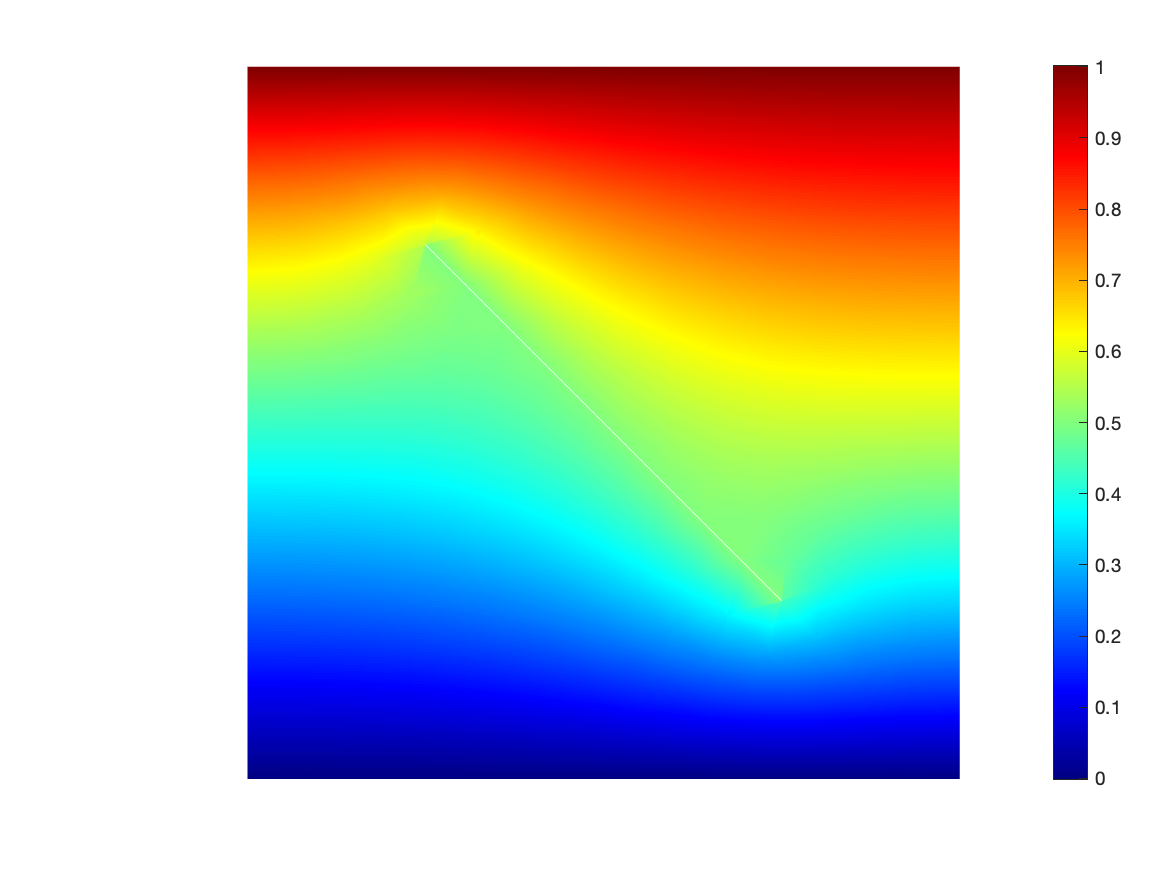}
  \caption{Contour of pressure - $2$}
 \end{subfigure}
 \begin{subfigure}[b]{0.45\textwidth}
  \includegraphics[width=\textwidth]{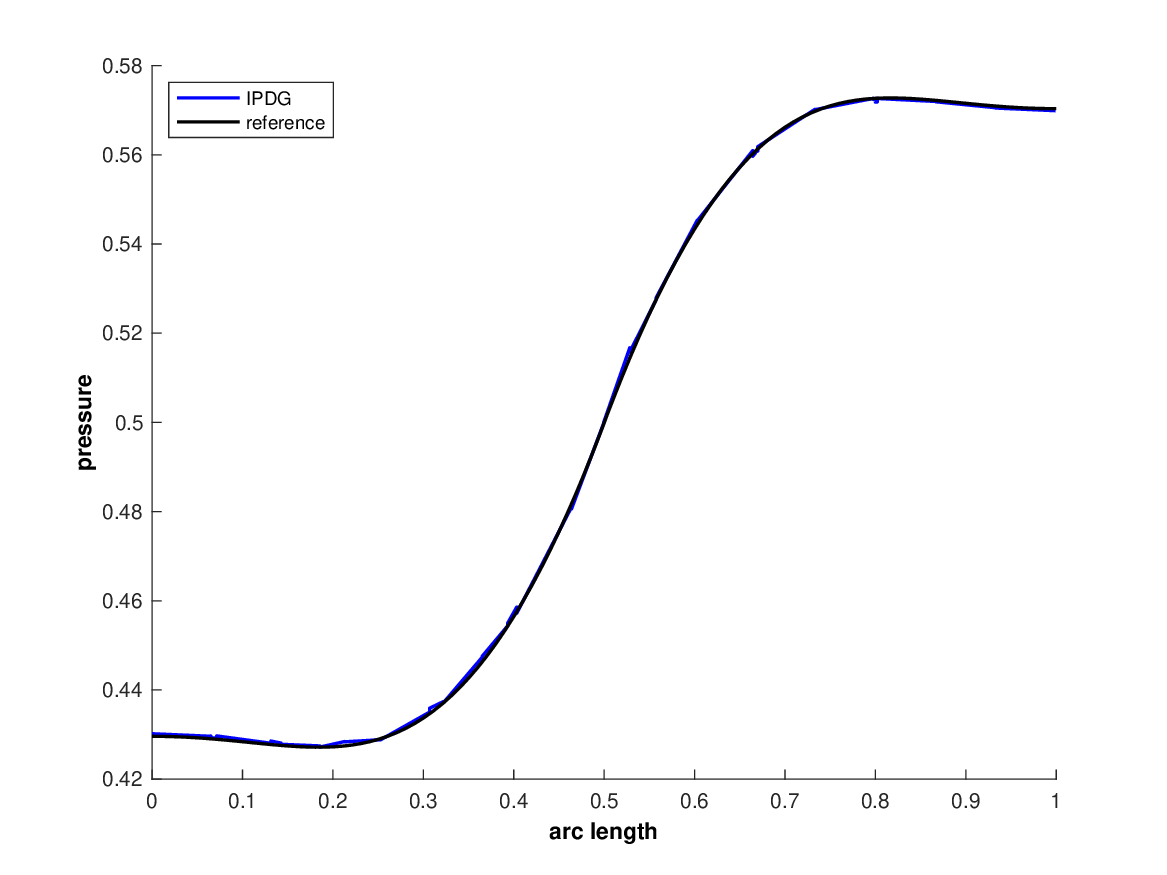}
  \caption{Slice of pressure - $2$}
 \end{subfigure}\\ 
 \caption{\textbf{Example \ref{ex:single}: single fracture.} 
 Results of the $P^1$-SIPG method for the single conductive fracture computed on the grids are shown in Figure \ref{fig:single_grid_medium}.
The slices of pressure in (b) and (d) are taken along $y=0.75$ and $y=0.5$, respectively.
The reference solutions are obtained from the Box-DFM \cite{xu2024extension} with $23, 306$ and $23, 455$ cells for the vertical and slanted fractures, respectively.}
 \label{fig:single_fracture_P1}
\end{figure}

\begin{figure}[!htbp]
 \centering
 \begin{subfigure}[b]{0.45\textwidth}
  \includegraphics[width=\textwidth]{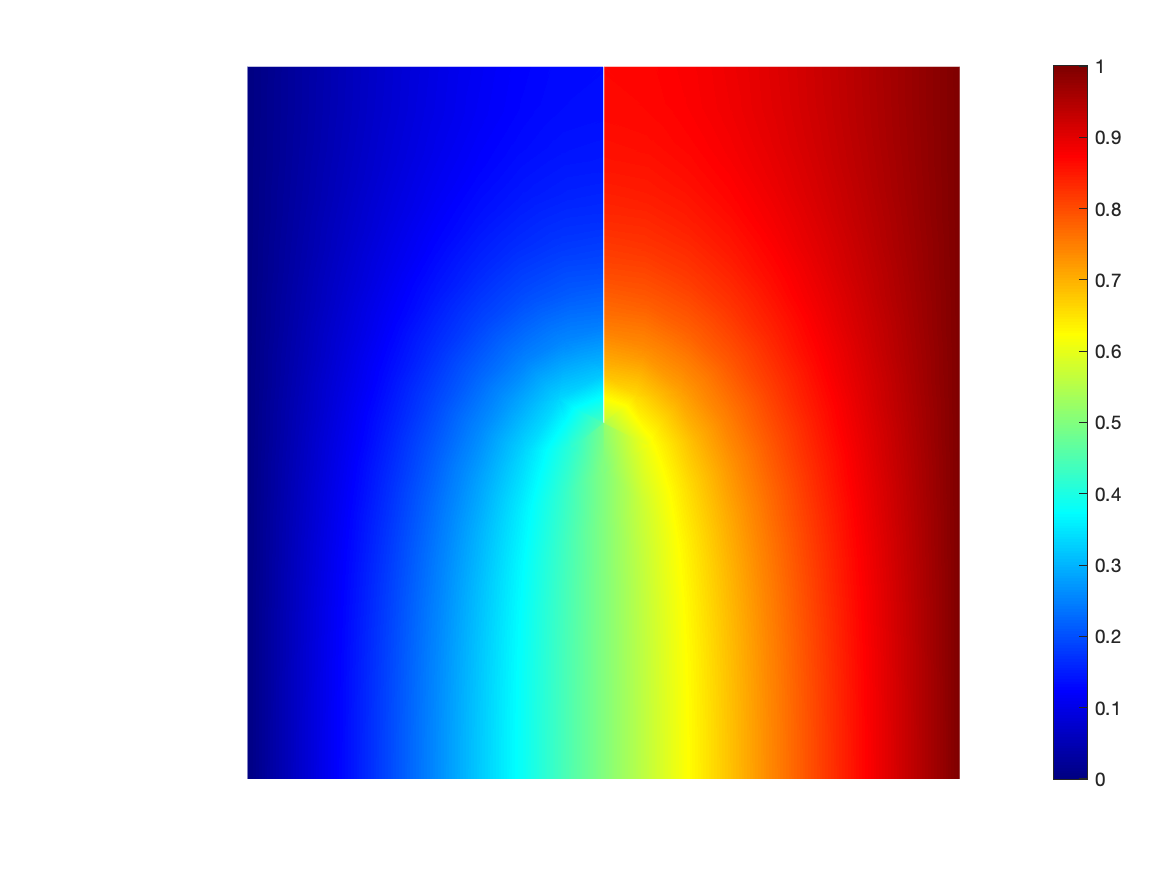}
  \caption{Contour of pressure - $1$}
 \end{subfigure}
 \begin{subfigure}[b]{0.45\textwidth}
  \includegraphics[width=\textwidth]{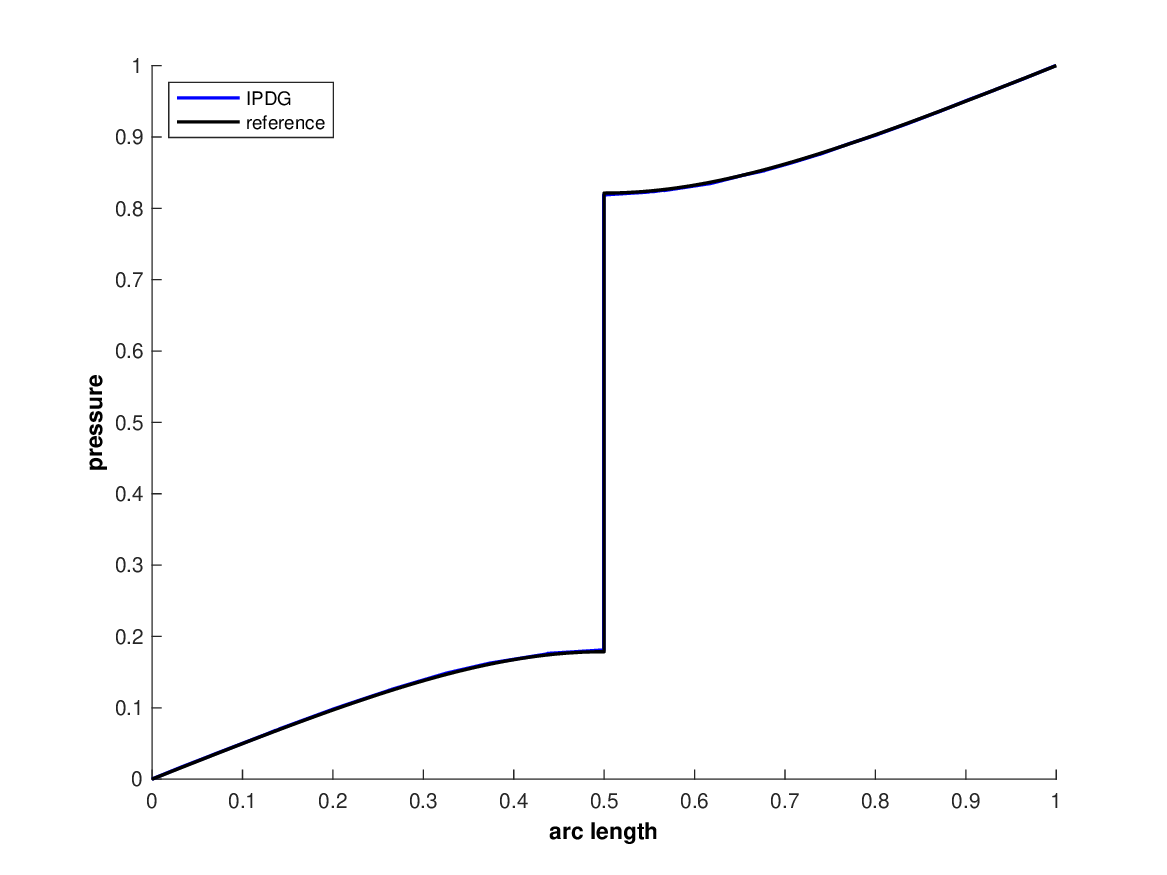}
  \caption{Slice of pressure - $1$}
 \end{subfigure}\\
 \begin{subfigure}[b]{0.45\textwidth}
  \includegraphics[width=\textwidth]{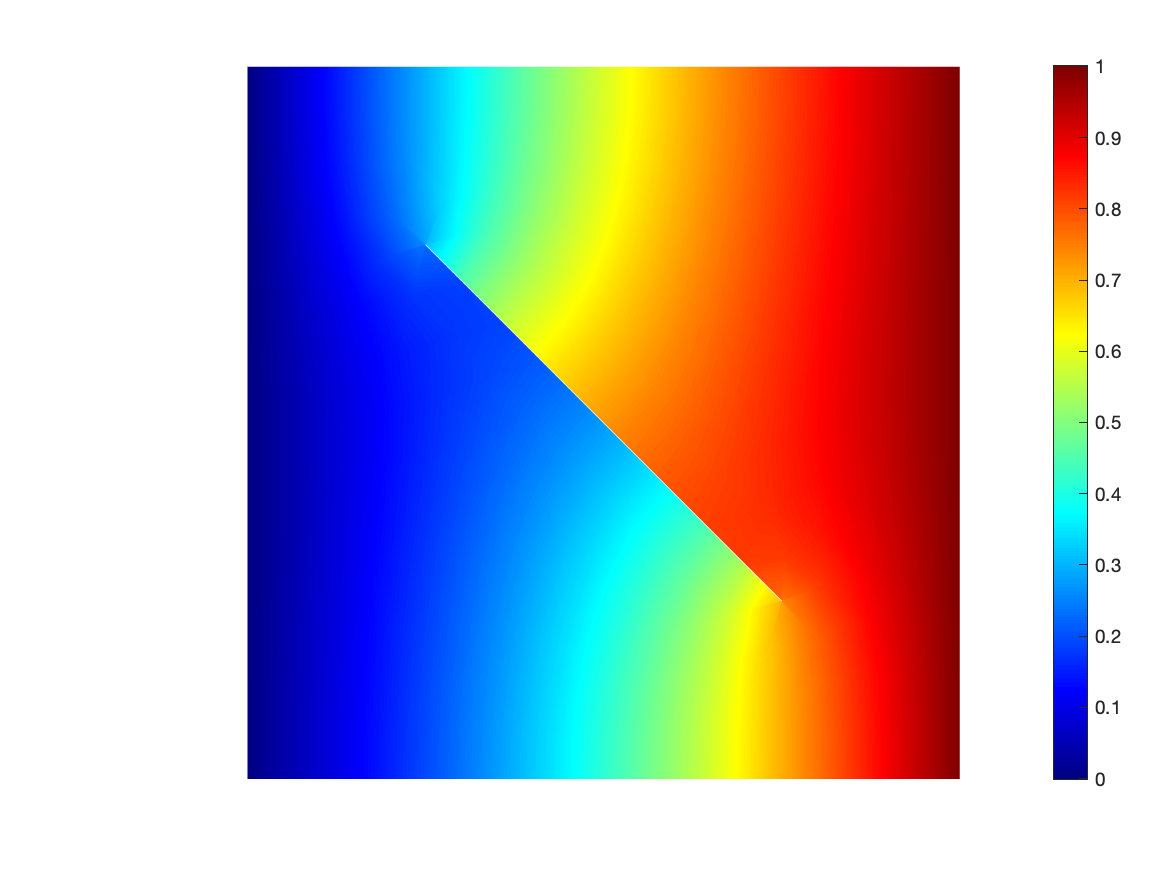}
  \caption{Contour of pressure - $2$}
 \end{subfigure}
 \begin{subfigure}[b]{0.45\textwidth}
  \includegraphics[width=\textwidth]{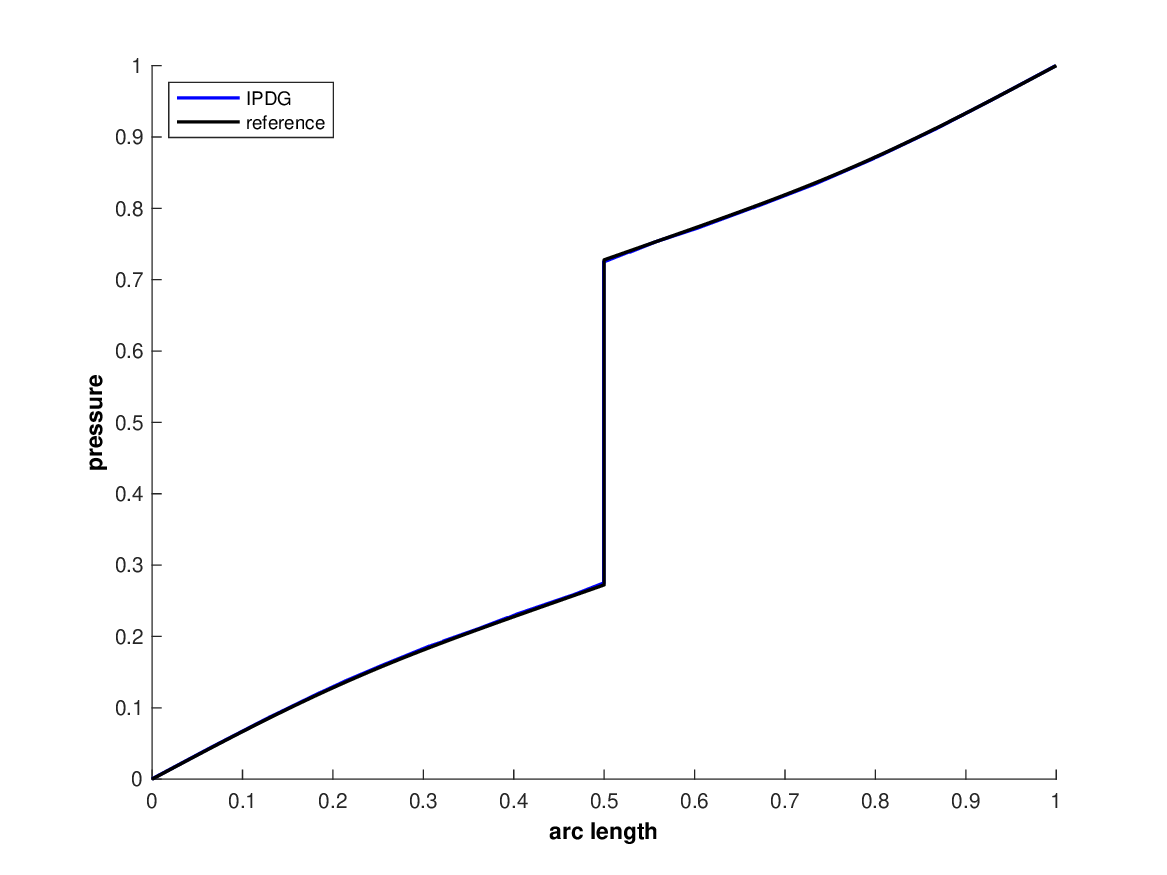}
  \caption{Slice of pressure - $2$}
 \end{subfigure}\\ 
 \caption{\textbf{Example \ref{ex:single}: single fracture.} 
 Results of the $P^1$-SIPG method for the single blocking barrier computed on the grids are shown in Figure \ref{fig:single_grid_medium}.
The slices of pressure in (b) and (d) are taken along $y=0.75$ and $y=0.5$, respectively.
The reference solutions are obtained from the Box-DFM \cite{xu2024extension} with $23, 306$ and $23, 455$ cells for the vertical and slanted barriers, respectively.}
 \label{fig:single_barrier_P1}
\end{figure}
\end{exmp}

\begin{exmp}\label{ex:regular}
\textbf{regular fracture network}

In this example, we test a regular fracture network as studied in \cite{flemisch2018benchmarks}.
The computational domain is set to $\Omega=(0,1)^2$ with impermeable boundary condition on the top and bottom, Neumann boundary condition $g_N=1$ on the left, and Dirichlet boundary condition $g_D=1$ on the right.
The permeability of the bulk matrix is $\mathbf{K}_m=\mathbf{I}$. 
Six fractures with a uniform aperture $a=10^{-4}$ are regularly distributed in the domain, with the exact coordinates detailed in \cite{flemisch2018benchmarks}, as also illustrated in Figure \ref{fig:grid_regular_P1}. 
Two cases are tested. 
In the case (a), the fractures are conductive with a uniform permeability $k_f=10^{4}$.
In the case (b), the fractures are blocking with a uniform permeability $k_b=10^{-4}$.

We perform the computation for both cases on the grid shown in Figure \ref{fig:grid_regular_P1}.
{We take the penalty parameters $\alpha_0=10^{4}, \tilde{\alpha}_0=10$ in case (a), and $\alpha_0=10$ in case (b).}
The numerical results are presented in Figure \ref{fig:regular_fractures_P1} and Figure \ref{fig:regular_barriers_P1} for the cases of conductive fractures and blocking barriers, respectively.
It can be observed from the figures that our numerical results exhibit excellent agreement with the reference solutions provided by the authors of \cite{flemisch2018benchmarks}.

\begin{figure}[!hbpt]
 \centering
 \includegraphics[width=0.45\textwidth]{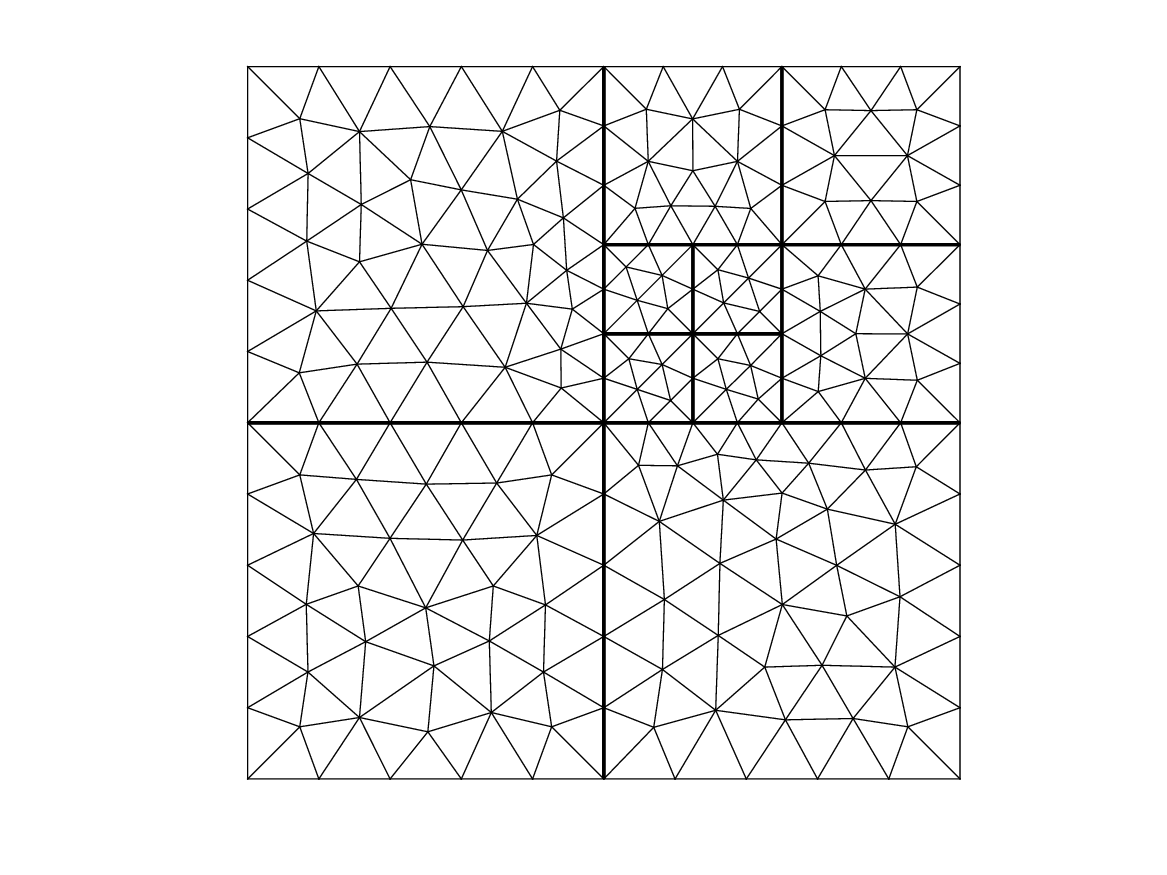}
  \setlength{\abovecaptionskip}{2pt} 
 \caption{\textbf{Example \ref{ex:regular}: regular fracture network.} The distribution of fractures (indicated by the black thick line segments) and the grid containing $366$ cells used in the computation.}
 \label{fig:grid_regular_P1}
\end{figure}

\begin{figure}[!htbp]
 \centering
 \begin{subfigure}[b]{0.45\textwidth}
  \includegraphics[width=\textwidth]{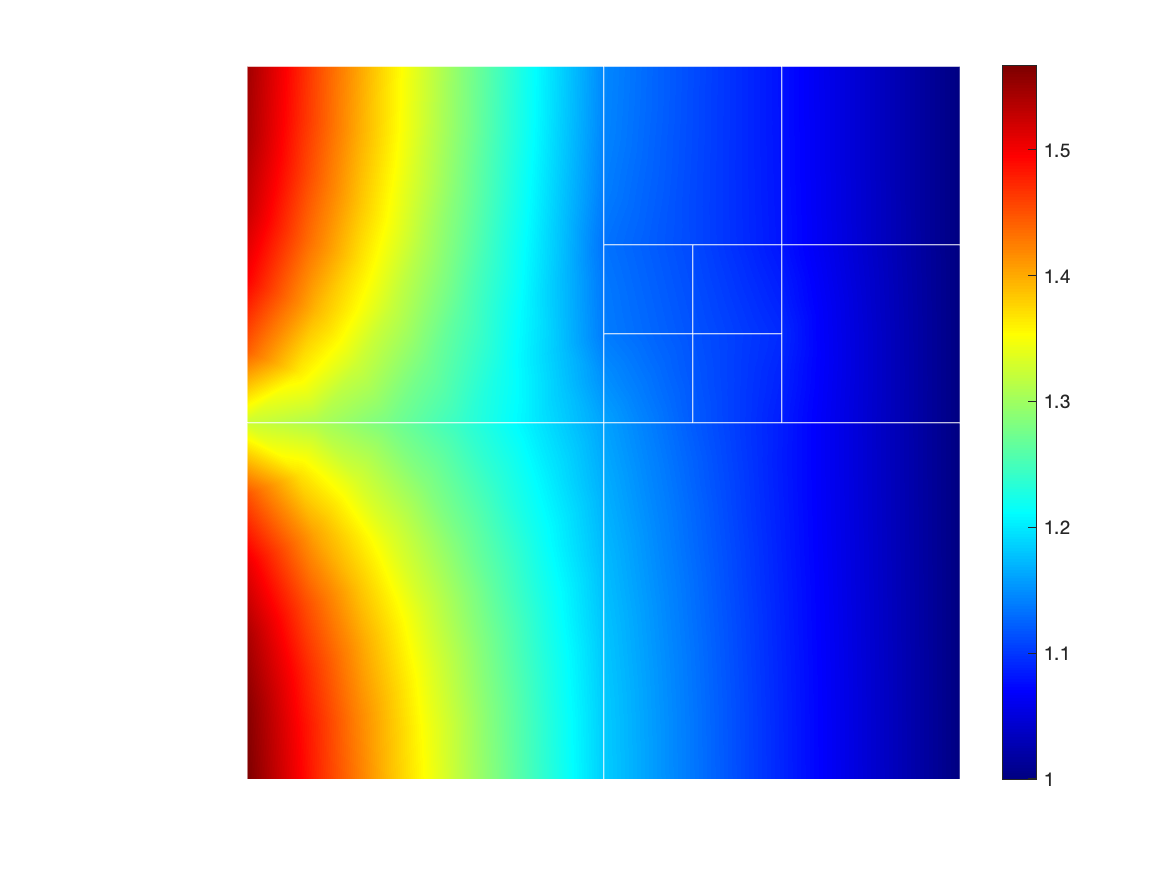}
  \caption{Contour of pressure}
 \end{subfigure}
 \begin{subfigure}[b]{0.45\textwidth}
  \includegraphics[width=\textwidth]{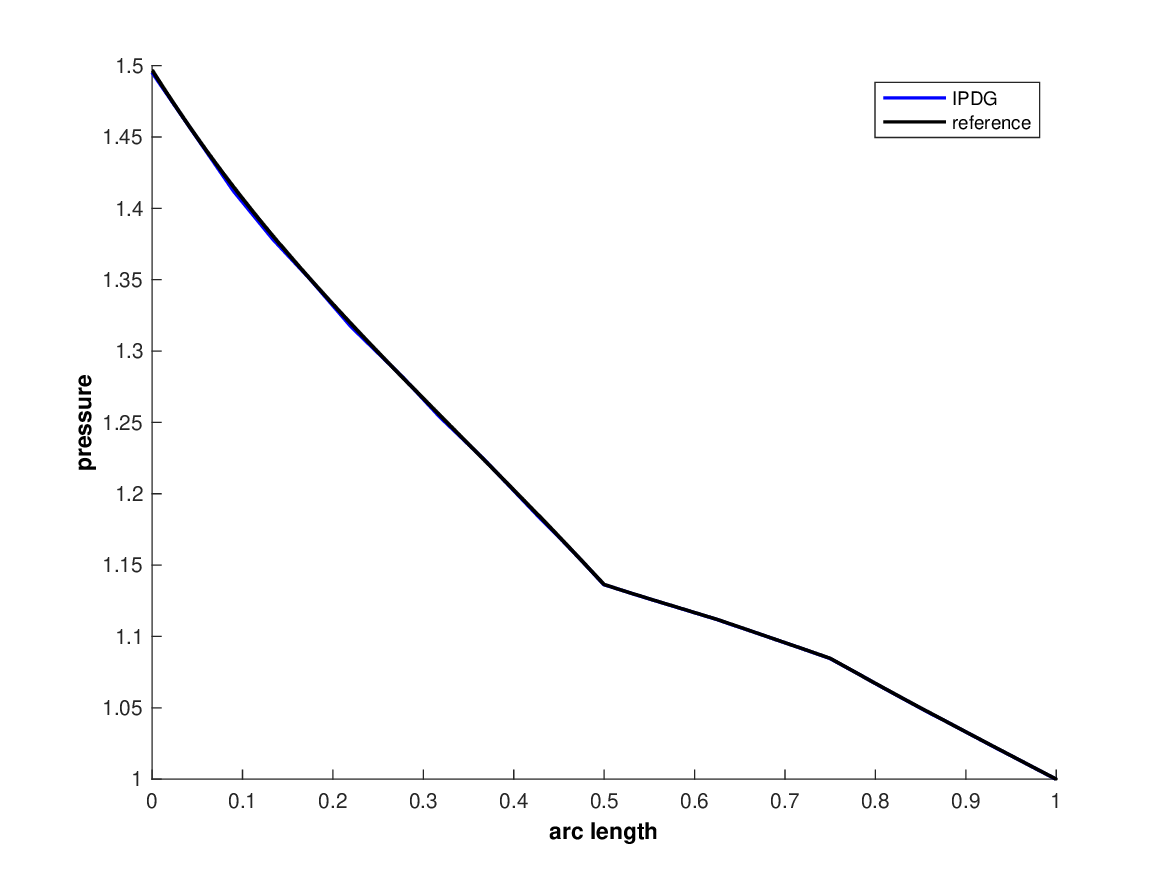}
  \caption{Slice of pressure - $1$}
 \end{subfigure}
 \begin{subfigure}[b]{0.45\textwidth}
  \includegraphics[width=\textwidth]{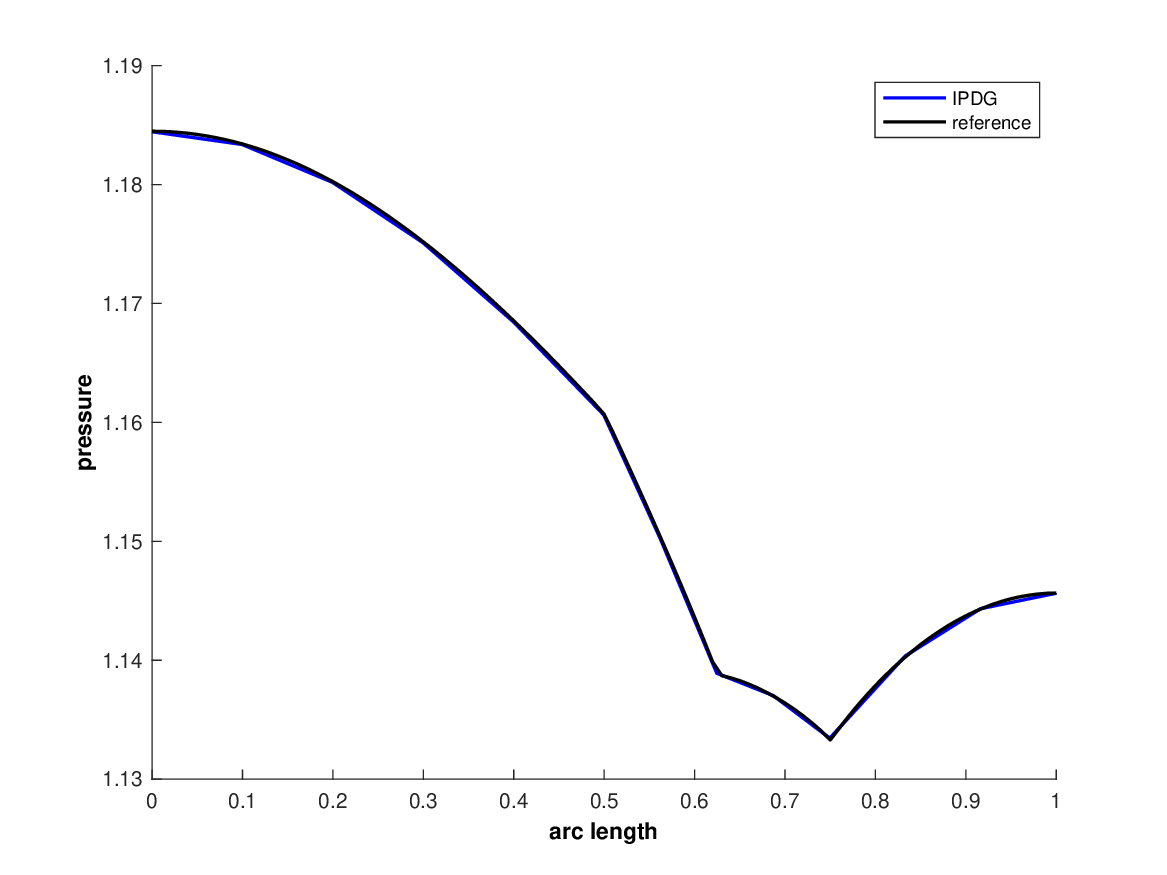}
  \caption{Slice of pressure - $2$}
 \end{subfigure}
 \caption{\textbf{Example \ref{ex:regular}: regular fracture network.} 
 Results of the $P^1$-SIPG method for conductive fractures computed on the grid are shown in Figure \ref{fig:grid_regular_P1}.
 The slices of pressure in (b) and (c) are taken along $y=0.7$ and $x=0.5$, respectively.
 The reference solution, provided by the authors of \cite{flemisch2018benchmarks}, is obtained from the mimetic finite difference method with $1, 136, 456$ cells for the equi-dimensional model.}
 \label{fig:regular_fractures_P1}
\end{figure}

\begin{figure}[!htbp]
 \centering
 \begin{subfigure}[b]{0.45\textwidth}
  \includegraphics[width=\textwidth]{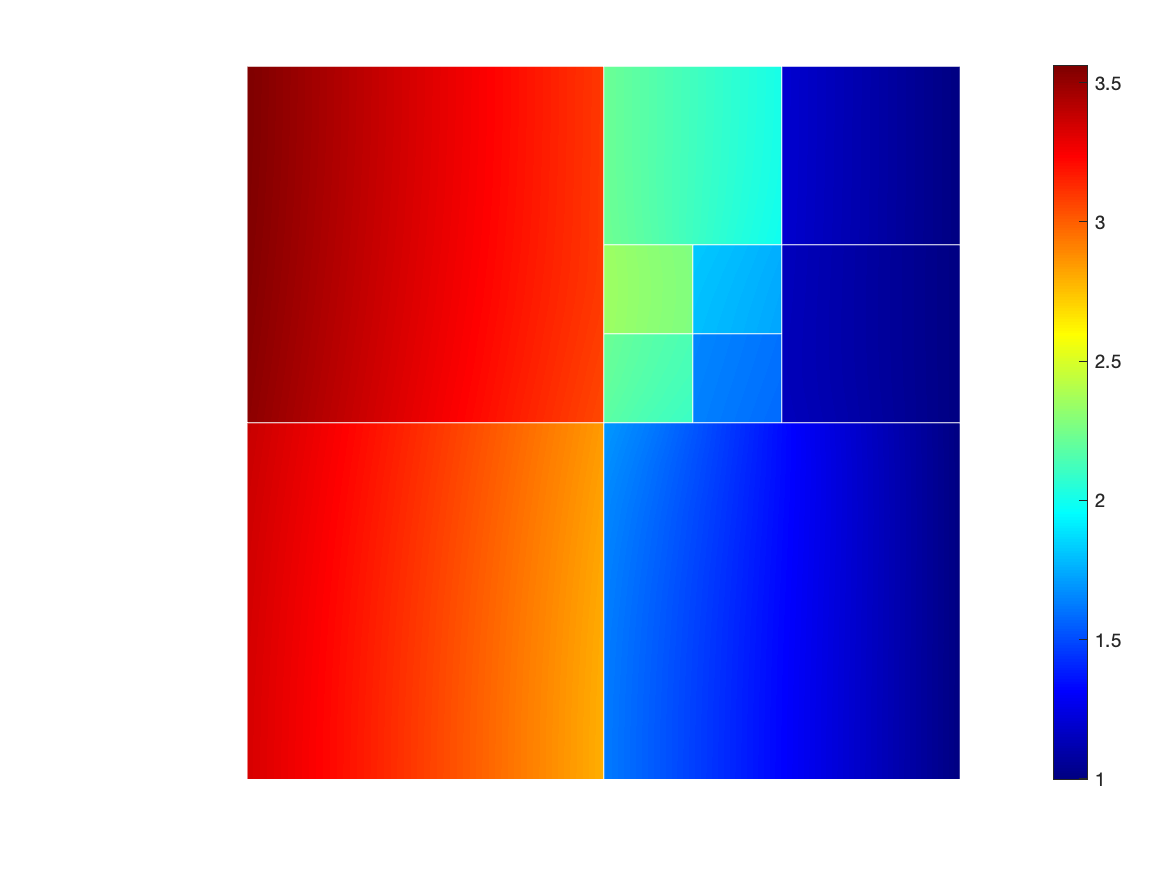}
  \caption{Contour of pressure}
 \end{subfigure}
 \begin{subfigure}[b]{0.45\textwidth}
  \includegraphics[width=\textwidth]{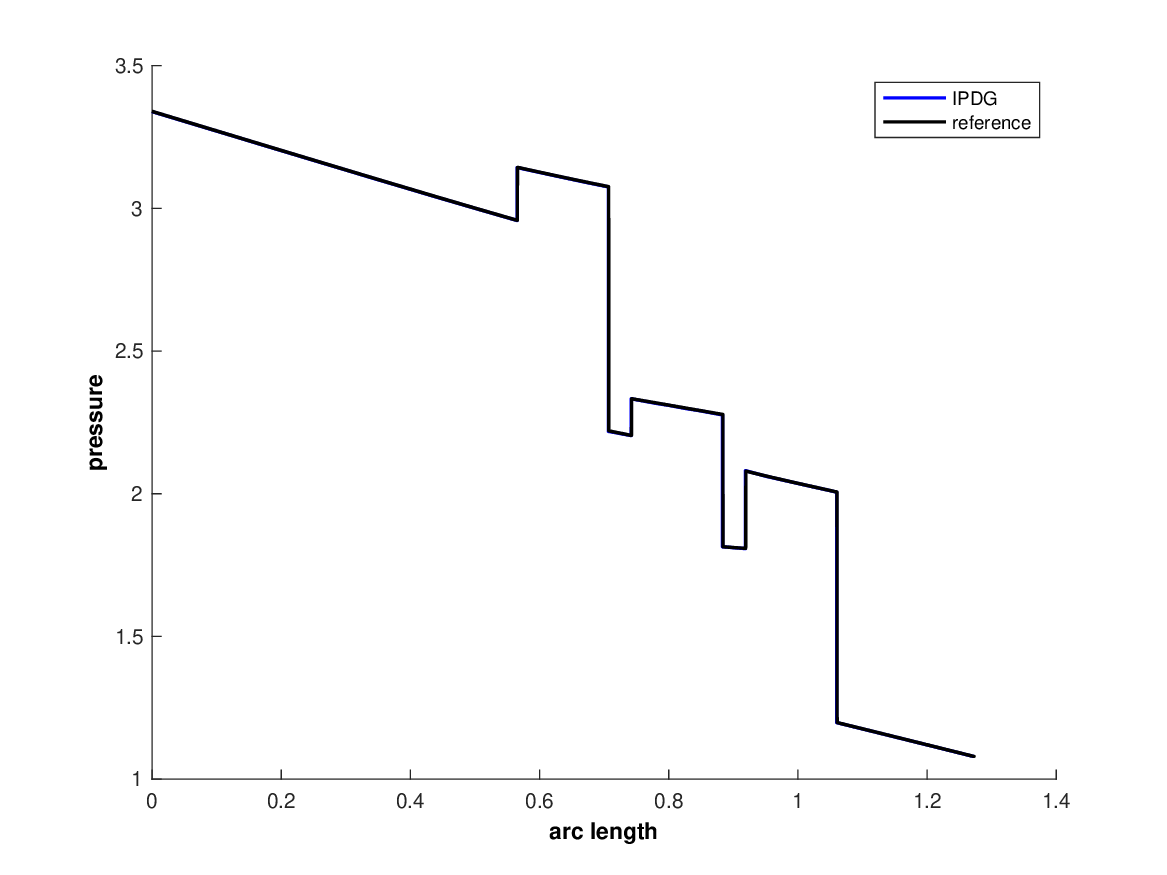}
  \caption{Slice of pressure}
 \end{subfigure}
 \caption{\textbf{Example \ref{ex:regular}: regular fracture network.} 
 Results of the $P^1$-SIPG method for blocking barriers computed on the grid are shown in Figure \ref{fig:grid_regular_P1}.
 The slice of pressure in (b) is taken along the line (0.0, 0.1) - (0.9, 1.0).
 The reference solution, provided by the authors of \cite{flemisch2018benchmarks}, is obtained from the mimetic finite difference method with $1, 136, 456$ cells for the equi-dimensional model.}
 \label{fig:regular_barriers_P1}
\end{figure}
    
\end{exmp}

\begin{exmp}\label{ex:complex}
\textbf{complex fracture network}

In this example, we test a complex fracture network where conductive fractures and blocking barriers coexist and intersect.
The computational domain is set to $\Omega=(0,1)^2$.
The fracture network in the domain contains $8$ conductive fractures and $2$ blocking fractures with a uniform aperture $a=10^{-4}$.
The permeabilities of the bulk matrix, conductive fractures, and blocking barriers are set to $\mathbf{K}_m=\mathbf{I}$, $k_f=10^{4}$ and $k_b=10^{-4}$, respectively.
The exact coordinates of the fractures were given in appendix C of \cite{flemisch2018benchmarks}.
See an illustration of the distribution of the fracture network and the computational grid in Figure \ref{fig:grid_complex_P1}.

Two cases are considered in the test.
In case (a), we investigate a predominantly vertical flow driven by the Dirichlet conditions $g_D=4$ and $g_D=1$ on the top and bottom boundaries, respectively, and the impermeable condition on the left and right boundaries.
In the case (b), we investigate a predominantly horizontal flow driven by the Dirichlet boundary conditions $g_D=4$ and $g_D=1$ on the left and right boundaries, respectively, and the impermeable condition on the top and bottom boundaries.

When handling the intersections of a fracture and a barrier, we prioritize the interface condition of the barrier at these intersections. 
Specifically, we assume that the fracture is divided into two separate, non-communicating parts by the barrier, allowing for discontinuous pressure across the barrier. 
This treatment aims to approximate the adoption of a harmonic average of the permeabilities of the fracture and barrier for the intersection cell, as used in the methods described in \cite{flemisch2018benchmarks}. 
In other situations, one may choose to honor the interface condition of the fracture at the intersections if the conductive fracture dominates.

We perform the computation for both cases on the grid shown in Figure \ref{fig:grid_complex_P1}. 
{We take the penalty parameters $\alpha_0=\tilde{\alpha}_0=10$ in both cases.}
The numerical results are presented in Figure \ref{fig:complex_t2b_P1} and Figure \ref{fig:complex_l2r_P1} for the cases of vertical and horizontal flows, respectively.
One can observe an excellent agreement of our results with the reference solutions provided by the authors of \cite{flemisch2018benchmarks}.

\begin{figure}[!hbpt]
 \centering
 \includegraphics[width=0.45\textwidth]{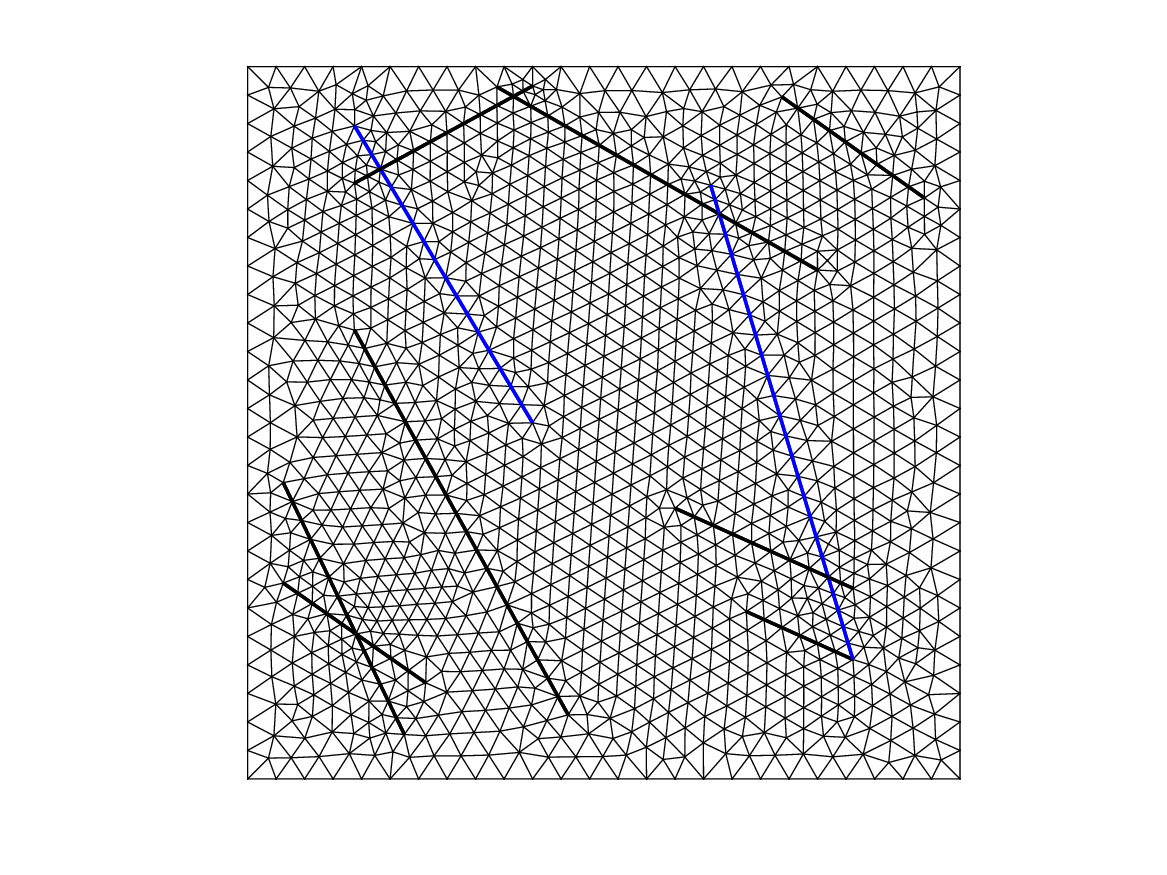}
  \setlength{\abovecaptionskip}{2pt} 
 \caption{\textbf{Example \ref{ex:complex}: complex fracture network.} 
 The distribution of fractures (indicated by the black and blue thick line segments for the conductive and blocking fractures, respectively) and the grid containing $2, 680$ cells used in the computation.}
 \label{fig:grid_complex_P1}
\end{figure}

\begin{figure}[!htbp]
 \centering
 \begin{subfigure}[b]{0.45\textwidth}
  \includegraphics[width=\textwidth]{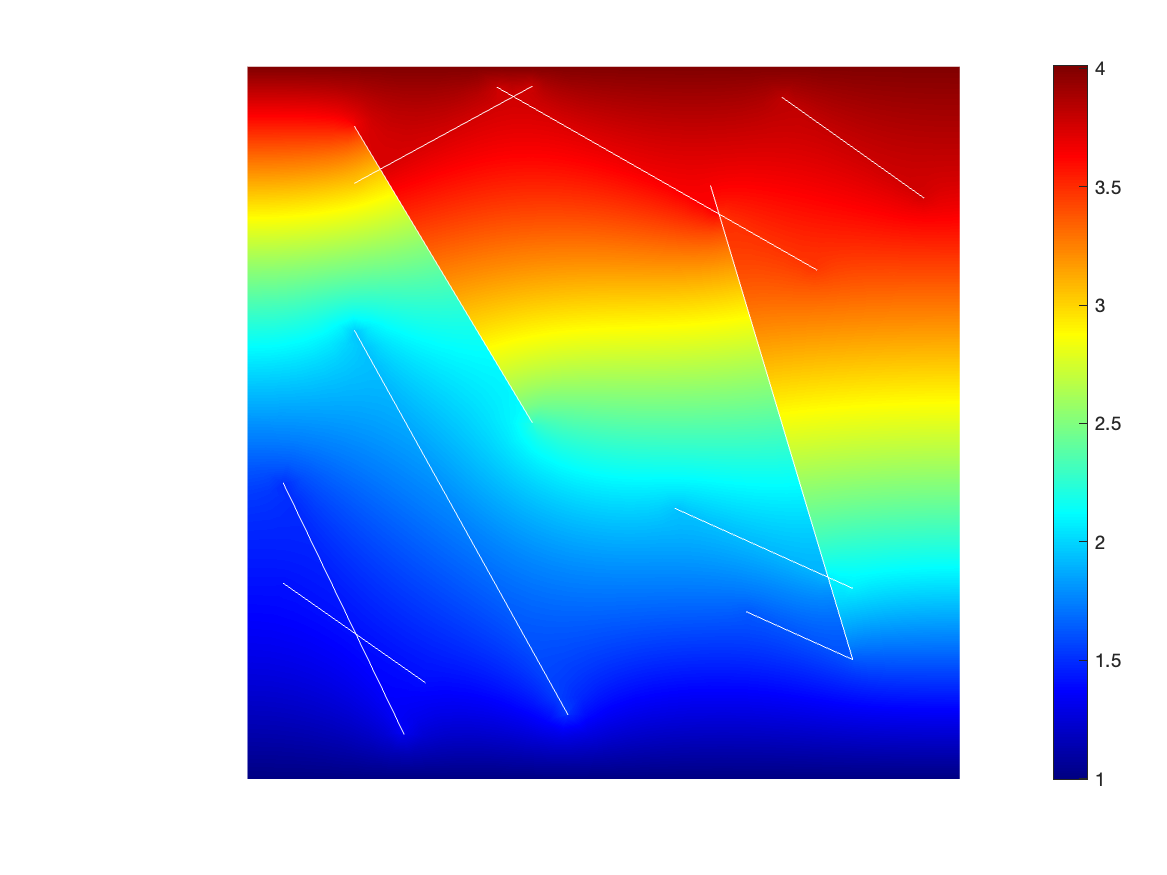}
  \caption{Contour of pressure}
 \end{subfigure}
 \begin{subfigure}[b]{0.45\textwidth}
  \includegraphics[width=\textwidth]{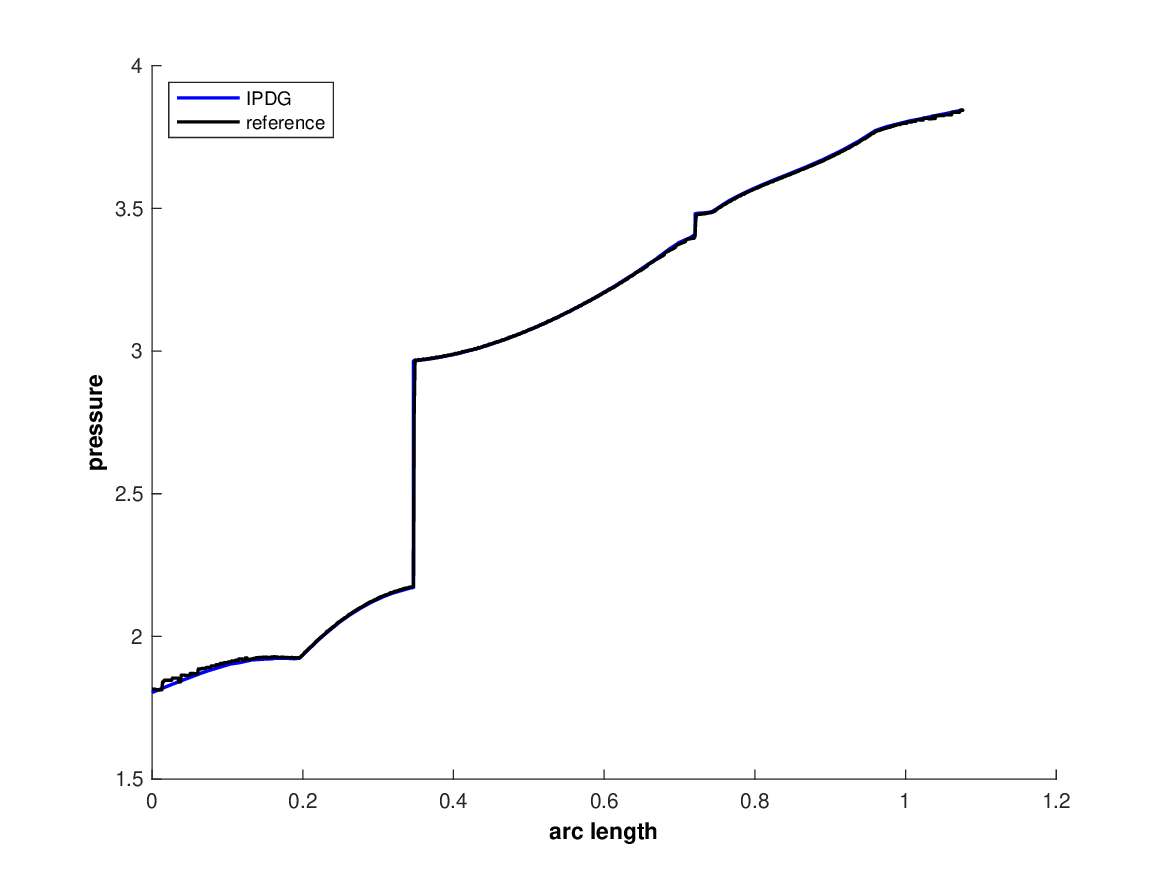}
  \caption{Slice of pressure}
 \end{subfigure}
 \caption{\textbf{Example \ref{ex:complex}: complex fracture network.} 
 Results of the $P^1$-SIPG method for the predominately vertical flow computed on the grid are shown in Figure \ref{fig:grid_complex_P1}.
 The slice of pressure in (b) is taken along the line (0.0, 0.5) - (1.0, 0.9).
 The reference solution, provided by the authors of \cite{flemisch2018benchmarks}, is obtained from the mimetic finite difference method with $2,260,352$ cells for the equi-dimensional model.}
 \label{fig:complex_t2b_P1}
\end{figure}

\begin{figure}[!htbp]
 \centering
 \begin{subfigure}[b]{0.45\textwidth}
  \includegraphics[width=\textwidth]{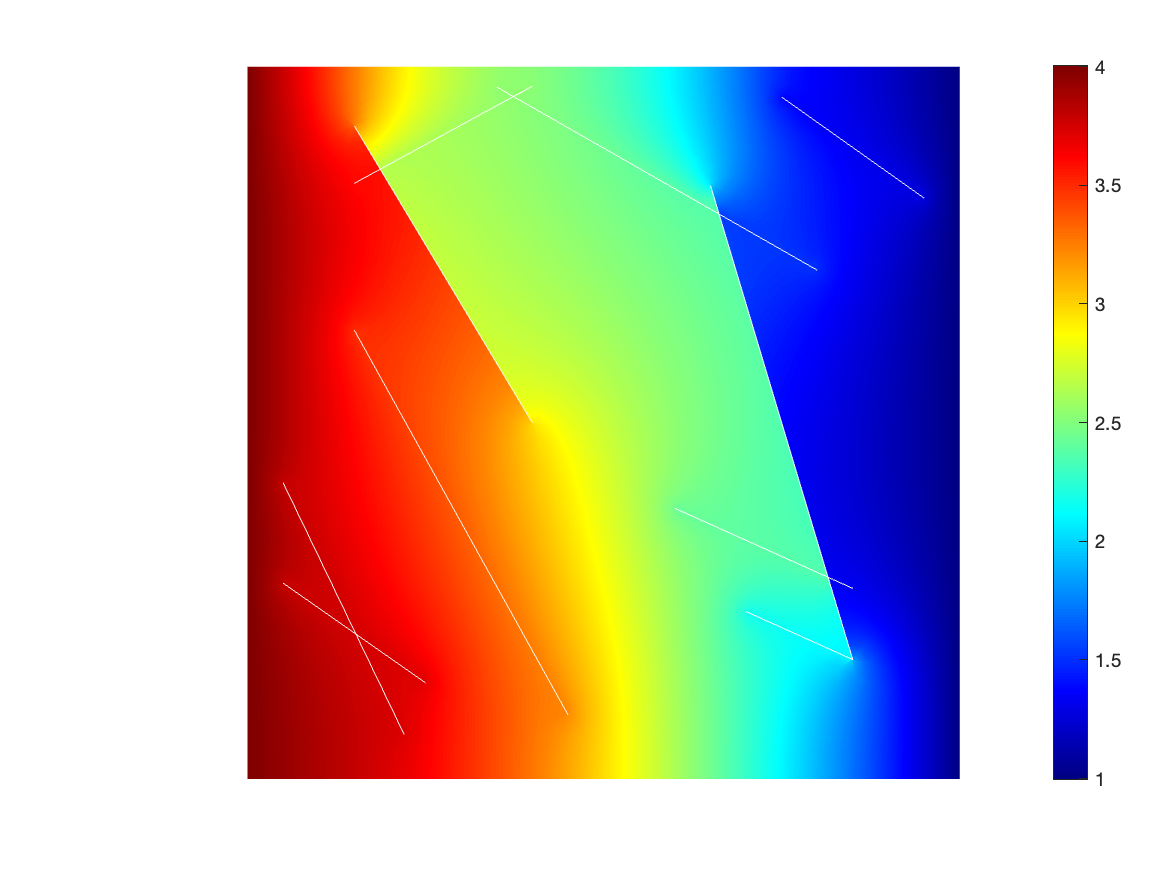}
  \caption{Contour of pressure}
 \end{subfigure}
 \begin{subfigure}[b]{0.45\textwidth}
  \includegraphics[width=\textwidth]{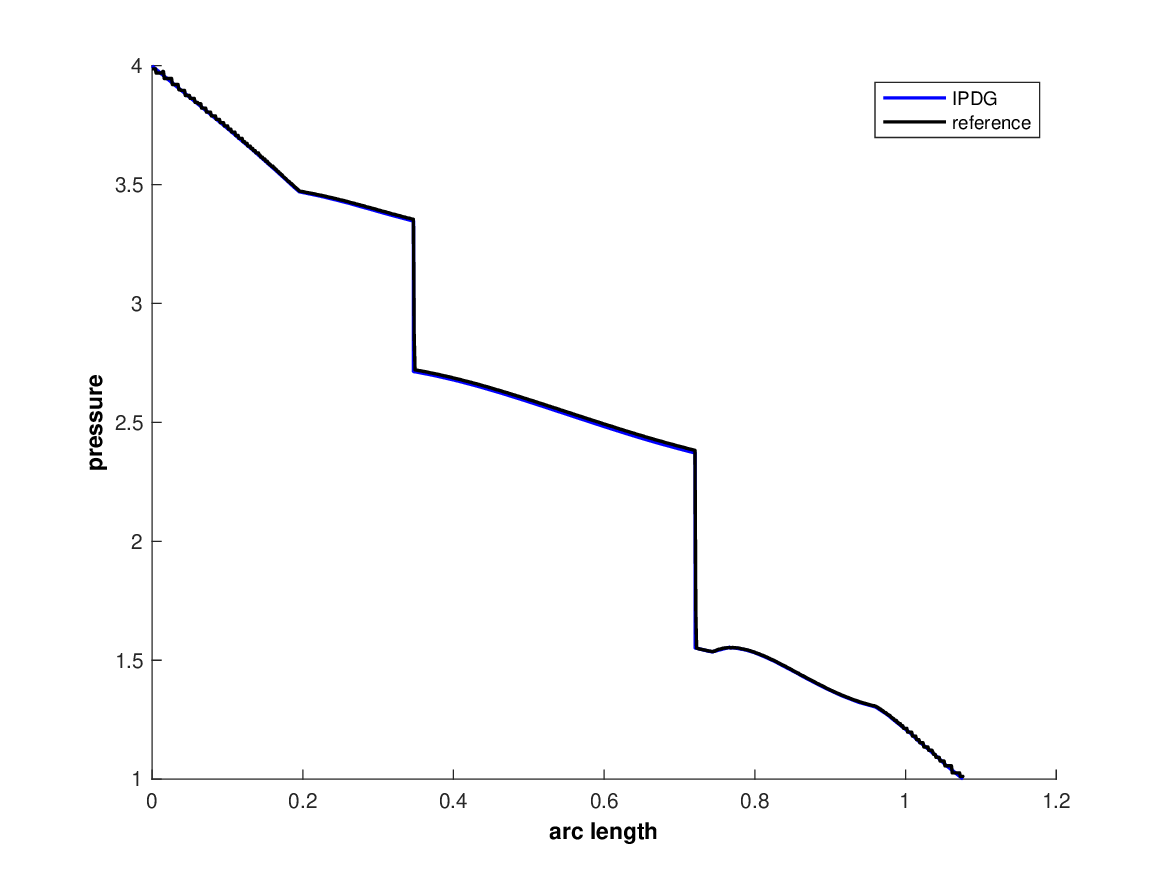}
  \caption{Slice of pressure}
 \end{subfigure}
 \caption{\textbf{Example \ref{ex:complex}: complex fracture network.} 
 Results of the $P^1$-SIPG method for the predominately horizontal flow computed on the grid are shown in Figure \ref{fig:grid_complex_P1}.
 The slice of pressure in (b) is taken along the line (0.0, 0.5) - (1.0, 0.9).
 The reference solution, provided by the authors of \cite{flemisch2018benchmarks}, is obtained from the mimetic finite difference method with $2,260,352$ cells for the equi-dimensional model.}
 \label{fig:complex_l2r_P1}
\end{figure}
\end{exmp}

\begin{exmp}\label{ex:real}
\textbf{realistic case}

In this example, we test a realistic case of a fracture network containing $64$ fractures.
The original setup with conductive fractures was given in \cite{flemisch2018benchmarks}, and its variant featuring blocking barriers was investigated in \cite{glaser2022comparison}.

In the case (a), all fractures are conductive with a uniform permeability $k_f=10^{-8}$ and aperture $a=10^{-2}$.
In the case (b), all fractures are blocking with a uniform permeability $k_b=10^{-18}$ and aperture $a=10^{-2}$.
In both cases, the computational domain is $\Omega=(0, 700)\times(0, 600)$ with the permeability of bulk matrix given as $\mathbf{K}_m=10^{-14}\mathbf{I}$.
The exact coordinates of the fracture network are available in the data repository shared in \cite{flemisch2018benchmarks}.
See an illustration of the distribution of the fractures and the computational grid in Figure \ref{fig:grid_real_P1}.
The flow in the domain is driven by Dirichlet conditions $g_D=1, 013, 250$ and $g_D=0$ on the left and right boundaries, respectively, and the top and bottom boundaries are impermeable.

We perform the computation for both cases on the grid shown in Figure \ref{fig:grid_real_P1}. 
{We take the penalty parameters $\alpha_0=\tilde{\alpha}_0=10^{-5}$ in case (a), and $\alpha_0=10^{-4}$ in case (b).}
The numerical results are presented in Figure \ref{fig:real_fractures_P1} and Figure \ref{fig:real_barriers_P1} for the cases of conductive and blocking fractures, respectively.
Unfortunately, there is no reference solutions available in the literature due to the complication of geometry.
Therefore, we compare our results on slices with those obtained from different DFMs in \cite{flemisch2018benchmarks, glaser2022comparison, xu2024extension}.
It can be observed from the comparison that our results don't have noticeable deviations from the solution range of other numerical methods.

\begin{figure}[!hbpt]
 \centering
 \includegraphics[width=0.45\textwidth]{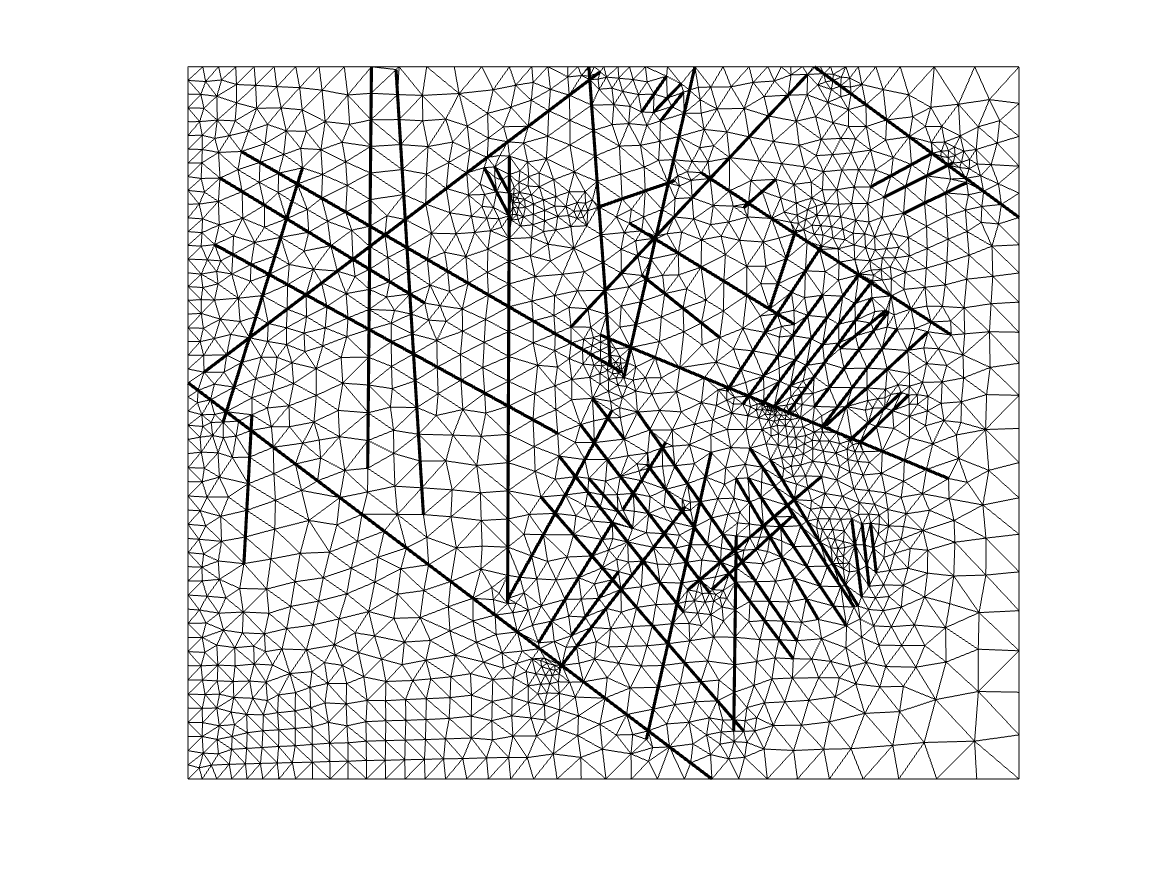}
  \setlength{\abovecaptionskip}{2pt} 
 \caption{\textbf{Example \ref{ex:real}: realistic case.} 
 The distribution of fractures (indicated by the black thick line segments) and the grid containing $3, 611$ cells used in the computation.}
 \label{fig:grid_real_P1}
\end{figure}

\begin{figure}[!htbp]
 \centering
 \begin{subfigure}[b]{0.45\textwidth}
  \includegraphics[width=\textwidth]{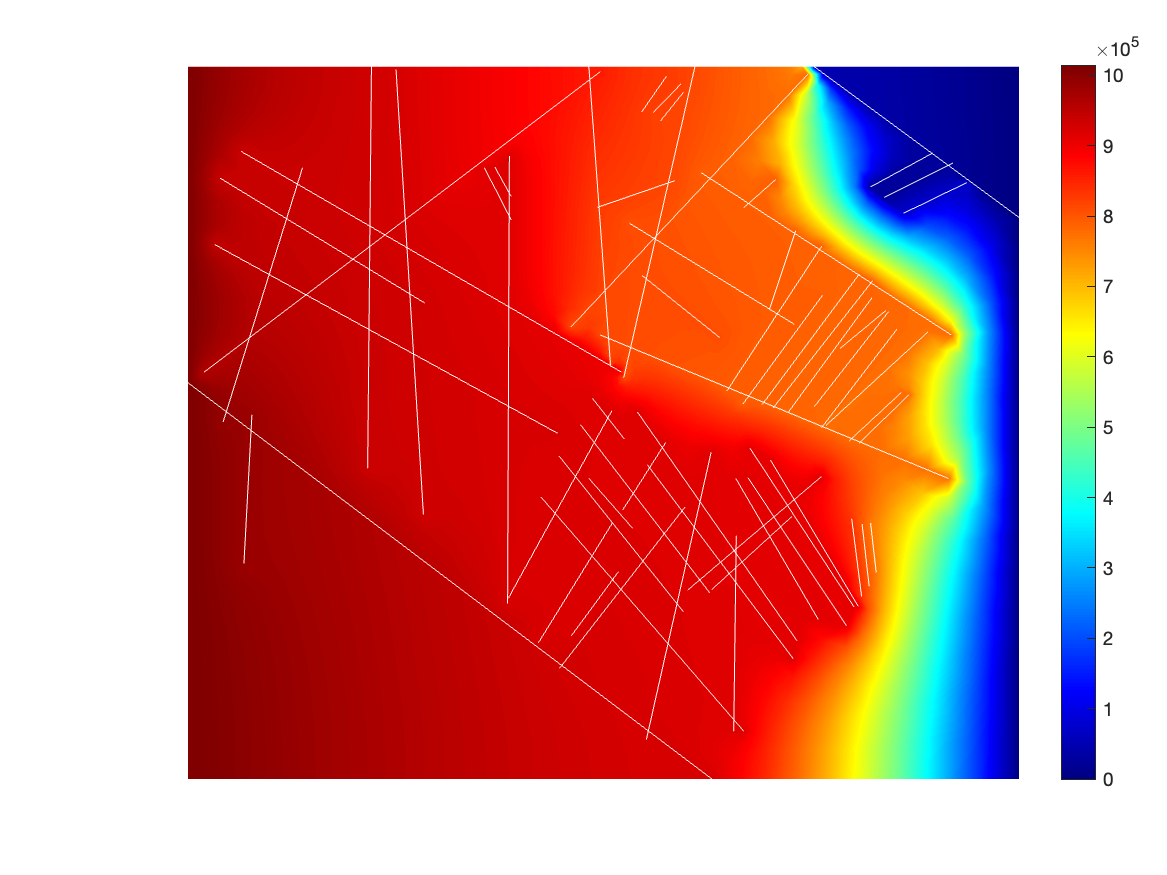}
  \caption{Contour of pressure}
 \end{subfigure}
 \begin{subfigure}[b]{0.45\textwidth}
  \includegraphics[width=\textwidth]{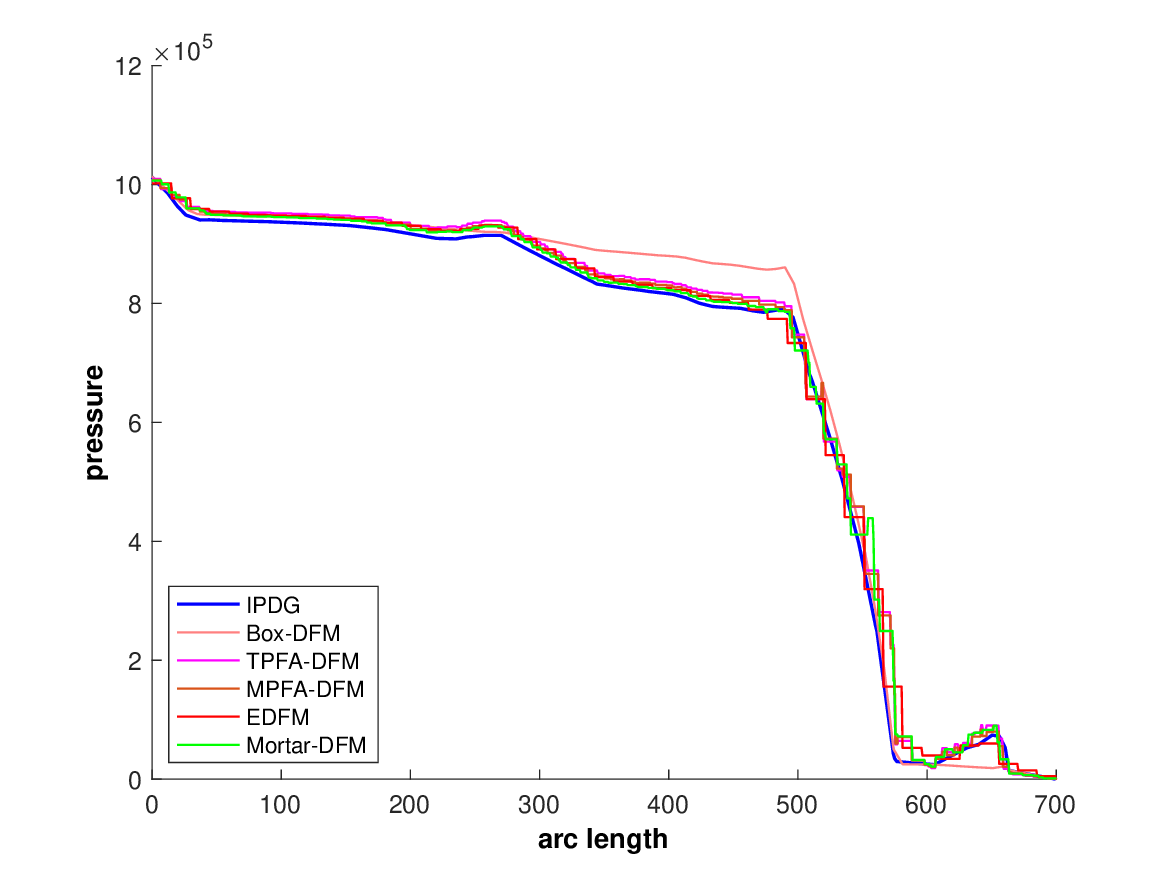}
  \caption{Slice of pressure - $1$}
 \end{subfigure}
 \begin{subfigure}[b]{0.45\textwidth}
  \includegraphics[width=\textwidth]{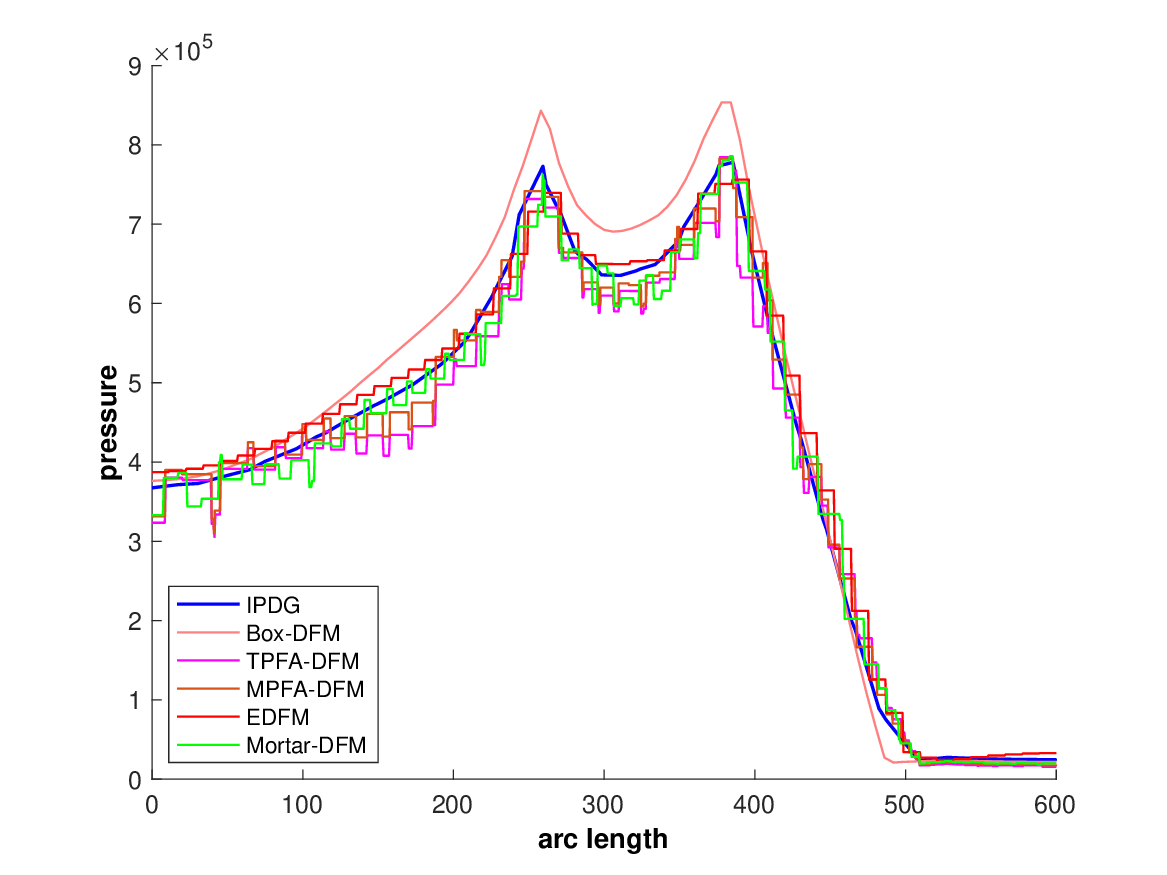}
  \caption{Slice of pressure - $2$}
 \end{subfigure}
 \caption{\textbf{Example \ref{ex:real}: realistic case.} 
 Results of the $P^1$-SIPG method for conductive fractures computed on the grid are shown in Figure \ref{fig:grid_real_P1}.
 The slices of pressure in (b) and (c) are taken along $y = 500$ and $x = 625$, respectively.
 The numerical methods that participate in the comparison, provided by the authors of \cite{flemisch2018benchmarks}, are the box method discrete fracture model (Box-DFM), cell-centered control volume discrete fracture model with two-point flux approximation (TPFA-DFM) and multi-point flux approximation (MPFA-DFM), embedded discrete fracture model (EDFM), and flux-mortar discrete fracture model (Mortar-DFM).}
 \label{fig:real_fractures_P1}
\end{figure}

\begin{figure}[!htbp]
 \centering
 \begin{subfigure}[b]{0.45\textwidth}
  \includegraphics[width=\textwidth]{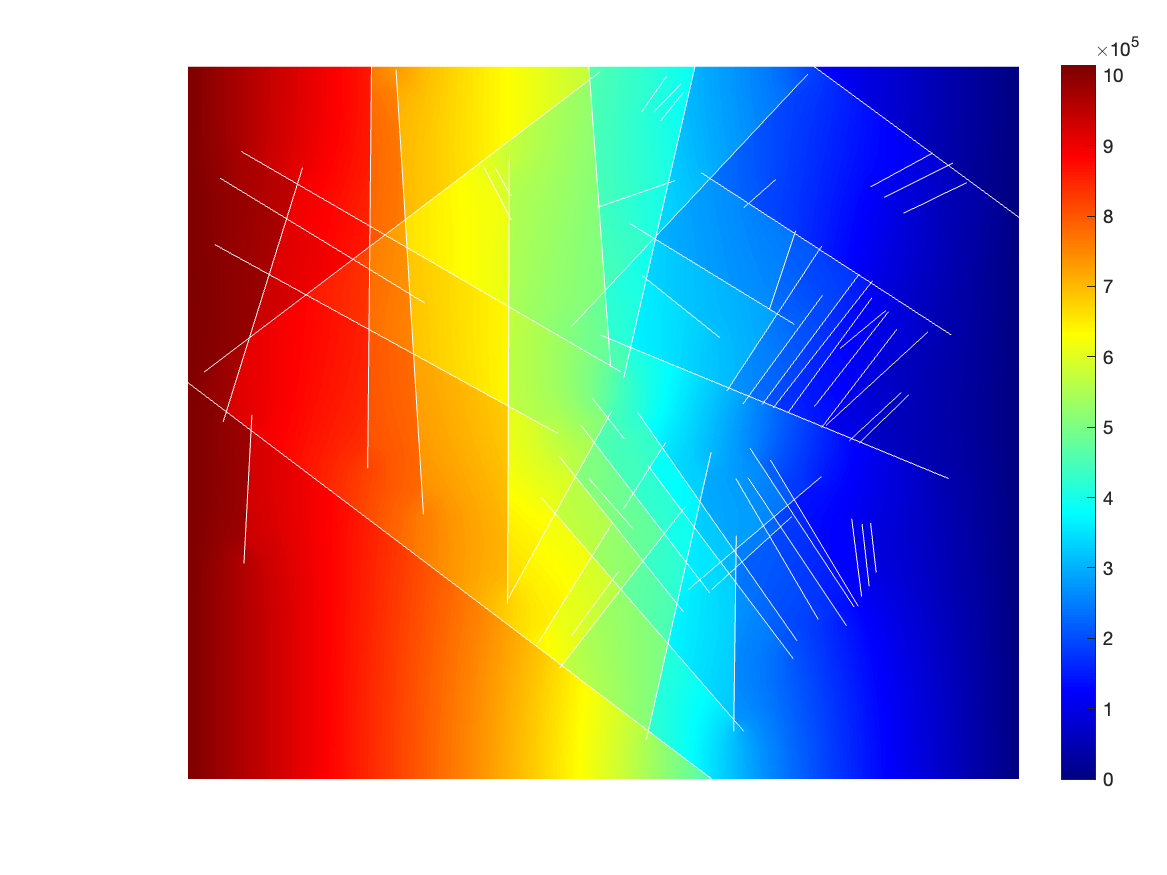}
  \caption{Contour of pressure}
 \end{subfigure}
 \begin{subfigure}[b]{0.45\textwidth}
  \includegraphics[width=\textwidth]{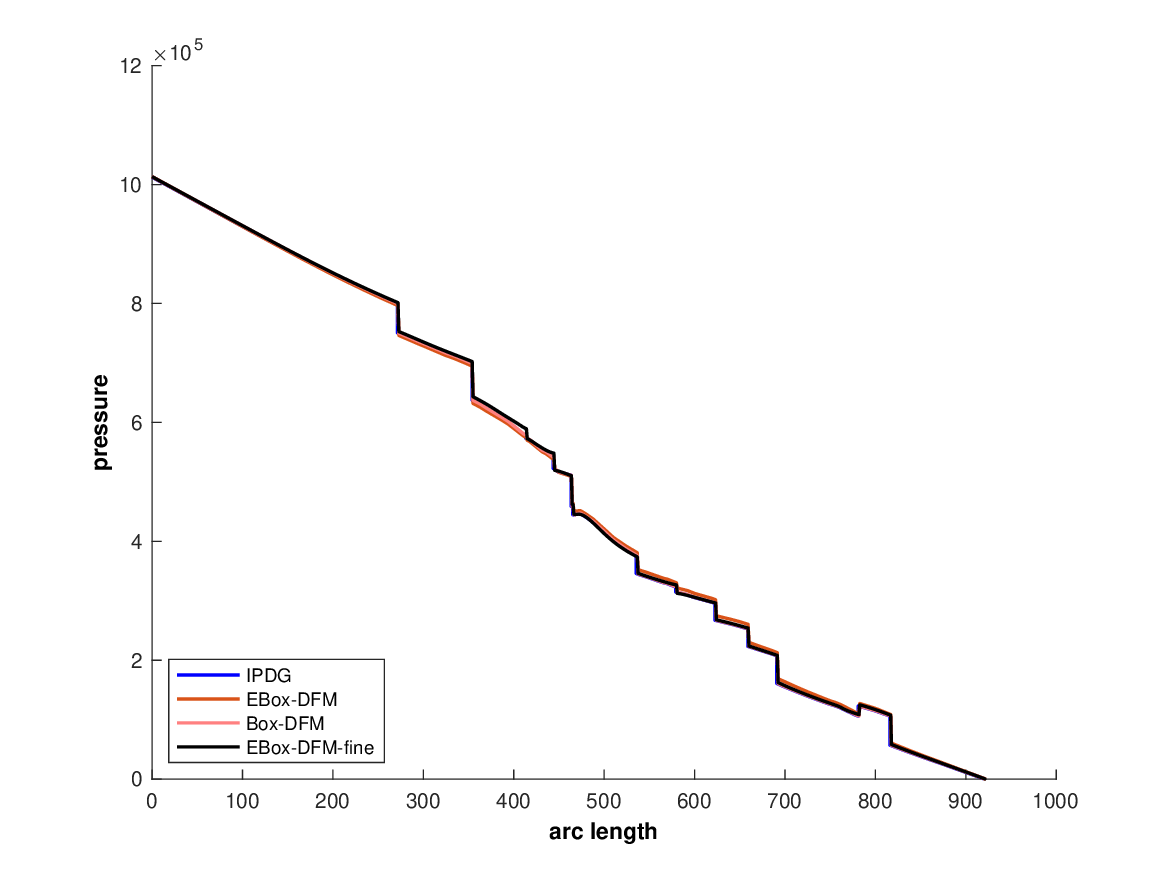}
  \caption{Slice of pressure - $1$}
 \end{subfigure}
 \begin{subfigure}[b]{0.45\textwidth}
  \includegraphics[width=\textwidth]{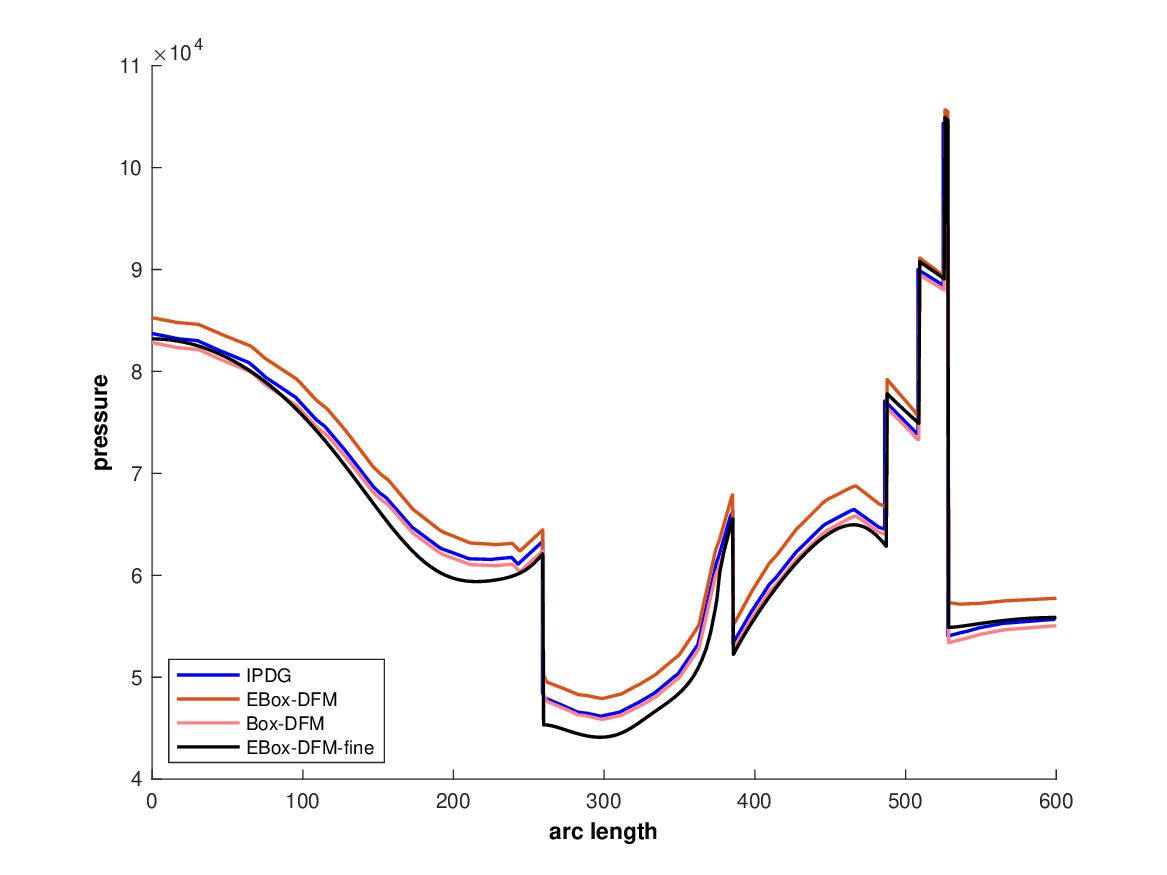}
  \caption{Slice of pressure - $2$}
 \end{subfigure}
 \caption{\textbf{Example \ref{ex:real}: realistic case.} 
 Results of the $P^1$-SIPG method for blocking barriers computed on the grid are shown in Figure \ref{fig:grid_real_P1}.
 The slices of pressure in (b) and (c) are taken along the line (0,0) - (700,600) and $x = 625$, respectively.
 The numerical methods that participate in the comparison are EBox-DFM \cite{glaser2022comparison} and Box-DFM \cite{xu2024extension}.
 The results of EBox-DFM-fine are obtained from the EBox-DFM on a fine grid containing $114, 721$ cells, while all other results are obtained on the same grid in Figure \ref{fig:grid_real_P1}.}
 \label{fig:real_barriers_P1}
\end{figure}

\end{exmp}

\section{Extension to two-phase flows}\label{Sect:two-phase}
The scheme established in Section \ref{Sect:scheme} can be directly extended to two-phase flows. 
In this section, we recall the equations for incompressible two-phase flows and apply the IPDG method for the interface model of two-phase flows in fractured porous media. 
A numerical test is presented last to demonstrate the validity of our algorithm.

\subsection{Interface model}
By neglecting the capillary pressure and gravitational effects, the governing equations of incompressible two-phase flows in porous media, known as the Buckley–Leverett equation, are given as \cite{karimi2001numerical},
\begin{equation*}
{-\nabla\cdot\left((\frac{k_{rn}}{\mu_n}+\frac{k_{rw}}{\mu_w})\mathbf{K}\nabla p\right)=q_n+q_w,}\quad (x,y)\in\Omega,
\end{equation*}
and 
\begin{equation*}
\frac{\partial(\phi s_w)}{\partial t}-\nabla\cdot(\frac{k_{rw}}{\mu_w}\mathbf{K}\nabla p)=q_w,\quad (x,y)\in\Omega,
\end{equation*}
where $k_{rn}=k_{rn}(s_w)$ and $k_{rw}=k_{rw}(s_w)$ are the relative permeabilities of the non-wetting and wetting-phase fluids, respectively, $\mu_n$ and $\mu_{w}$ are the viscosities of the non-wetting and wetting-phase fluids, respectively, $q_n$ and $q_w$ are source terms for the non-wetting and wetting-phase fluids, respectively, $s_w$ is the saturation of the wetting phase, $\phi$ is porosity of the rocks, and $p$ is the pressure of the fluid.

In the presence of fractures (refer to Figure \ref{fig:hybridDomain}), we consider the following interface problem for modeling the two-phase flows in fractured media,
\begin{equation}\label{eq:PDEmodel_pressure}
\begin{cases}
{-\nabla\cdot\left((\frac{k_{rn}}{\mu_n}+\frac{k_{rw}}{\mu_w})\mathbf{K}_m\nabla p\right)=q_n+q_w,} &(x,y)\in\Omega\setminus\left(\gamma_1\cup \gamma_2\right)\\
{\mathbf{u}^{-}\cdot\mathbf{n}_1^-+\mathbf{u}^{+}\cdot\mathbf{n}_1^+=\mathbf{0}}, & (x,y)\in\gamma_1\\
{\mathbf{u}^{-}\cdot\mathbf{n}_1^-=-(\frac{k_{rn}}{\mu_n}+\frac{k_{rw}}{\mu_w})k_b\frac{p^+-p^-}{a}}, & (x,y)\in\gamma_1\\
{p^+-p^-=0}, & (x,y) \in \gamma_2\\
{-\frac{\partial}{\partial \boldsymbol{\nu}_2}\left(a(\frac{k_{rn}}{\mu_{n}}+\frac{k_{rw}}{\mu_w})k_f\frac{\partial p^+}{\partial \boldsymbol{\nu}_2}\right)={  q_{f,n}+q_{f,w}+}\mathbf{u}^{-}\cdot\mathbf{n}_2^-+\mathbf{u}^{+}\cdot\mathbf{n}_2^+}, & (x,y) \in \gamma_2
\end{cases}
\end{equation}
and 
\begin{equation}\label{eq:PDEmodel_saturation}
\begin{cases}
\frac{\partial(\phi^m s_w)}{\partial t}-\nabla\cdot(\frac{k_{rw}}{\mu_w}\mathbf{K}_m\nabla p)=q_w, &(x,y)\in\Omega\setminus\left(\gamma_1\cup \gamma_2\right)\\
{\mathbf{u}_w^{-}\cdot\mathbf{n}_1^-+\mathbf{u}_w^{+}\cdot\mathbf{n}_1^+=\mathbf{0}}, & (x,y)\in\gamma_1\\
{\mathbf{u}_w^{-}\cdot\mathbf{n}_1^-=-\frac{k_{rw}}{\mu_w}k_b\frac{p^+-p^-}{a}}, & (x,y)\in\gamma_1\\
{s_w^+-s_w^-=0}, & (x,y) \in \gamma_2\\
\frac{\partial (a\phi^f s^+_w)}{\partial t}-\frac{\partial}{\partial \boldsymbol{\nu}_2}\left(a\frac{k_{rw}}{\mu_w}k_f\frac{\partial p^+}{\partial \boldsymbol{\nu}_2}\right)={ q_{f,w}+}\mathbf{u}_{w}^{-}\cdot\mathbf{n}^-_2+\mathbf{u}_{w}^{+}\cdot\mathbf{n}^+_2, & (x,y) \in \gamma_2
\end{cases}
\end{equation}
where $\mathbf{u}_{\alpha}^{\pm}:=-\frac{k_{r\alpha}}{\mu_\alpha}\mathbf{K}_{m}^{\pm}\nabla p^{\pm}$, $\alpha=n,w$ are Darcy's velocities of the fluid of the phase $\alpha$ evaluated from either side of $\gamma_i,i=1,2$, and $\mathbf{u}^{\pm}:=\mathbf{u}_{n}^{\pm}+\mathbf{u}^{\pm}_w$ is the total velocity of the two phases.
The other notations have the same meaning as those used in Section \ref{Sect:model}.
The governing equations \eqref{eq:PDEmodel_pressure}, \eqref{eq:PDEmodel_saturation} are subject to the boundary conditions
\begin{equation}\label{eq:BdrCond_pressure}
p=g_D ~\text{on}~ \Gamma_D, \quad \left((\frac{k_{rn}}{\mu_n}+\frac{k_{rw}}{\mu_w})\mathbf{K}_m\nabla p\right)\cdot\mathbf{n}= g_N ~\text{on}~\Gamma_N:=\partial\Omega\setminus\Gamma_D,
\end{equation}
and 
\begin{equation}\label{eq:BdrCond_saturation}
s_w=s_{D,w} ~\text{on}~ \Gamma_{D,\text{in}}, \quad 
\left(\frac{k_{rw}}{\mu_w}\mathbf{K}_m\nabla p\right)\cdot\mathbf{n}=g_{N,w} ~\text{on}~\Gamma_{N,\text{in}},
\end{equation}
and the initial condition 
\begin{equation}\label{eq:IniCond_saturation}
s_{w}(x,y,0)=s_{w0}(x,y),\quad (x,y)\in\Omega,
\end{equation}
where $\Gamma_{D,\text{in}}:=\{(x,y)\in\Gamma_D: (\mathbf{K}_m\nabla p)\cdot\mathbf{n}>0\}$ and $\Gamma_{N,\text{in}}:=\{(x,y)\in\Gamma_N: g_N>0\}$ are the inflow Dirichlet and Neumann boundaries, respectively.

\begin{rem}
We assume the continuity of saturation across conductive fractures in the model \eqref{eq:PDEmodel_saturation}. 
This assumption was adopted in \cite{kim2000finite, karimi2001numerical}. In the literature \cite{monteagudo2004control, reichenberger2006mixed}, the authors argued that models accounting for discontinuities of saturation are more accurate when the capillary pressure functions of fractures and the bulk matrix differ. 
However, since the capillary pressure is neglected in the Buckley–Leverett equations, we assume continuity of saturation in the model for simplicity.
\end{rem}

\subsection{Numerical scheme}
We adopt the same notations as those established in Section \ref{Sect:scheme}, unless otherwise stated.
The IPDG scheme for the interface model \eqref{eq:PDEmodel_pressure} - \eqref{eq:IniCond_saturation} is to find $p_h, s_{wh} \in V_{h,k}^{DG}$, such that
\begin{equation}\label{eq:scheme_pressure}
\tilde{a}_h(p_h, s_{wh}, \xi)+\tilde{b}_h(p_h, s_{wh}, \xi)=\tilde{F}_h(s_{wh}, \xi) +\tilde{G}_h(s_{wh}, \xi) \quad \forall\xi\in V_{h,k}^{DG},
\end{equation}
and
\begin{equation}\label{eq:scheme_saturation}
\pi_{h}((s_{wh})_t,\xi)+c_h(p_h, s_{wh}, \xi)+d_h(p_h, s_{wh}, \xi)=H_h(s_{wh}, \xi) + I_h(s_{wh}, \xi) \quad \forall\xi\in V_{h,k}^{DG},
\end{equation}
where
\begin{equation}\label{eq:a_tilde}
\begin{split}
    \tilde{a}_h(p_h, s_{wh}, \xi)=&((\frac{k_{rn}}{\mu_n}+\frac{k_{rw}}{\mu_w})\mathbf{K}_m\nabla p_h,\nabla \xi)_{\mathcal{T}_h}\\
    &-\la \{(\frac{k_{rn}}{\mu_n}+\frac{k_{rw}}{\mu_w})\mathbf{K}_m\nabla p_h\},\jl\xi\jr \ra_{\mathcal{E}^{0}\cup \mathcal{E}^D\cup \mathcal{E}^{\gamma_2}}\\
    &+\sigma \la 
    \jl p_h\jr , \{(\frac{k_{rn}}{\mu_n}+\frac{k_{rw}}{\mu_w})\mathbf{K}_m\nabla\xi\} \ra_{\mathcal{E}^0\cup \mathcal{E}^D\cup \mathcal{E}^{\gamma_2}}\\
    &+\la \alpha \{\frac{k_{rn}}{\mu_n}+\frac{k_{rw}}{\mu_w}\}\jl p_h\jr, \jl \xi\jr \ra_{\mathcal{E}^0\cup \mathcal{E}^D\cup \mathcal{E}^{\gamma_2}}+\la\{\frac{k_{rn}}{\mu_n}+\frac{k_{rw}}{\mu_w}\}\frac{k_b}{a}\jl p_h\jr,\jl\xi\jr \ra_{\mathcal{E}^{\gamma_1}},
\end{split}
\end{equation}
\begin{equation}\label{eq:b_tilde}
\begin{split}
    \tilde{b}_h(p_h, s_{wh}, \xi)=&\frac 12  \la a(\frac{k_{rn}^-}{\mu_n}+\frac{k_{rw}^-}{\mu_w})k_f \frac{\partial p_h^-}{\partial \boldsymbol{\nu}_2}, \frac{\partial \xi^-}{\partial \boldsymbol{\nu}_2}\ra_{\mathcal{E}^{\gamma_2}}+\frac 12 \la  a(\frac{k_{rn}^+}{\mu_n}+\frac{k_{rw}^+}{\mu_w})k_f \frac{\partial p_h^+}{\partial \boldsymbol{\nu}_2}, \frac{\partial \xi^+}{\partial \boldsymbol{\nu}_2}\ra_{\mathcal{E}^{\gamma_2}}\\
    &-\frac 12[\{a(\frac{k_{rn}}{\mu_n}+\frac{k_{rw}}{\mu_w})k_f\frac{\partial p_h}{\partial \boldsymbol{\nu}_2}\}_{P_{-}^{\star}}, \jl \xi \jr_{P_{-}^\star}]_{\mathcal{V}^\circ \cup \mathcal{V}^D} -\frac 12[\{a(\frac{k_{rn}}{\mu_n}+\frac{k_{rw}}{\mu_w})k_f\frac{\partial p_h}{\partial \boldsymbol{\nu}_2}\}_{P_{+}^{\star}}, \jl \xi \jr_{P_{+}^\star}]_{\mathcal{V}^\circ \cup \mathcal{V}^D} \\
    &+\frac{\sigma}{2}[\{a(\frac{k_{rn}}{\mu_n}+\frac{k_{rw}}{\mu_w})k_f\frac{\partial \xi }{\partial \boldsymbol{\nu}_2}\}_{P_{-}^{\star}}, \jl p_h \jr_{P_{-}^\star}]_{\mathcal{V}^\circ \cup \mathcal{V}^D} +\frac\sigma2[\{a(\frac{k_{rn}}{\mu_n}+\frac{k_{rw}}{\mu_w})k_f\frac{\partial \xi}{\partial \boldsymbol{\nu}_2}\}_{P_{+}^{\star}}, \jl p_h \jr_{P_{+}^\star}]_{\mathcal{V}^\circ \cup \mathcal{V}^D} \\
    &+ [\tilde{\alpha}\{\frac{k_{rn}}{\mu_n}+\frac{k_{rw}}{\mu_w}\}_{P^{\star}_{-}}\jl p_h \jr_{P_{-}^\star},\jl \xi \jr_{P_{-}^\star}]_{\mathcal{V}^\circ \cup \mathcal{V}^D} +[\tilde{\alpha}\{\frac{k_{rn}}{\mu_n}+\frac{k_{rw}}{\mu_w}\}_{P_{+}^{\star}}\jl p_h \jr_{P_{+}^\star},\jl \xi \jr_{P_{+}^\star}]_{\mathcal{V}^\circ \cup \mathcal{V}^D},
\end{split}
\end{equation}

\begin{equation}\label{eq:F_tilde}
\tilde{F}_h(s_{wh}, \xi)=(q_n,\xi)_{\mathcal{T}_h}+(q_w,\xi)_{\mathcal{T}_h}+\la g_N, \xi \ra_{\mathcal{E}^N} +\sigma \la g_D, (\frac{k_{rn}}{\mu_n}+\frac{k_{rw}}{\mu_w})\mathbf{K}_m\nabla\xi\cdot\mathbf{n}\ra_{\mathcal{E}^D} +\la \alpha(\frac{k_{rn}}{\mu_n}+\frac{k_{rw}}{\mu_w}) g_D,\xi \ra_{\mathcal{E}^D},
\end{equation}
\begin{equation}\label{eq:G_tilde}
\begin{split}
\tilde{G}_h(s_{wh}, \xi)=&\frac{\sigma}{2}[ak_f(\{(\frac{k_{rn}}{\mu_n}+\frac{k_{rw}}{\mu_w})\frac{\partial \xi}{\partial \boldsymbol{\nu}_2}\}_{P_{-}^\star}+\{(\frac{k_{rn}}{\mu_n}+\frac{k_{rw}}{\mu_w})\frac{\partial \xi}{\partial \boldsymbol{\nu}_2}\}_{P_{+}^\star}),\text{sign}(\boldsymbol{\nu}_2\cdot \mathbf{n})g_D]_{\mathcal{V}^D}\\
&+[\tilde{\alpha}((\frac{k_{rn}}{\mu_n}+\frac{k_{rw}}{\mu_w})\xi_{P_{-}^\star}+(\frac{k_{rn}}{\mu_n}+\frac{k_{rw}}{\mu_w})\xi_{P_{+}^\star}),g_D]_{\mathcal{V}^D}{ + \la q_{f,n}, \{\xi\} \ra_{\mathcal{E}^{\gamma_2}} + \la q_{f,w}, \{\xi\} \ra_{\mathcal{E}^{\gamma_2}}},    
\end{split}
\end{equation}
\begin{equation}\label{eq:pi}
\pi_h((s_{wh})_t,\xi) = (\phi^m (s_{wh})_t, \xi)_{\mathcal{T}_h} + \frac 12 \la a \phi^f (s_{wh}^-)_t, \xi^-\ra_{\mathcal{E}^{\gamma_2}} + \frac 12 \la a\phi^f (s_{wh}^{+})_t, \xi^{+}\ra_{\mathcal{E}^{\gamma_2}},
\end{equation}
\begin{equation}\label{eq:c}
\begin{split}
    c_h(p_h, s_{wh}, \xi)=&(\frac{k_{rw}}{\mu_w}\mathbf{K}_m\nabla p_h,\nabla \xi)_{\mathcal{T}_h}-\la \{\frac{k_{rw}}{\mu_w}\mathbf{K}_m\nabla p_h\},\jl\xi\jr \ra_{\mathcal{E}^{0}\cup \mathcal{E}^D\cup \mathcal{E}^{\gamma_2}}\\
&+\sigma \la 
    \jl p_h\jr , \{\frac{k_{rw}}{\mu_w}\mathbf{K}_m\nabla\xi\} \ra_{\mathcal{E}^0\cup \mathcal{E}^D\cup \mathcal{E}^{\gamma_2}} + \la \alpha\{\frac{k_{rw}}{\mu_w}\}\jl p_h\jr, \jl \xi\jr \ra_{\mathcal{E}^0\cup \mathcal{E}^D\cup \mathcal{E}^{\gamma_2}}\\
    &+\la\{\frac{k_{rw}}{\mu_w}\}\frac{k_b}{a}\jl p_h\jr,\jl\xi\jr \ra_{\mathcal{E}^{\gamma_1}}+\la \beta\jl s_{wh}\jr, \jl \xi\jr \ra_{\mathcal{E}^0\cup \mathcal{E}^{\gamma_1}\cup \mathcal{E}^{\gamma_2}\cup\mathcal{E}^{D,\text{in}}}
\end{split}
\end{equation}
\begin{equation}\label{eq:d}
\begin{split}
    d_h(p_h,s_{wh},\xi)=&\frac 12 \la a\frac{k_{rw}^-}{\mu_w}k_f \frac{\partial p_h^-}{\partial \boldsymbol{\nu}_2}, \frac{\partial \xi^-}{\partial \boldsymbol{\nu}_2}\ra_{\mathcal{E}^{\gamma_2}}+\frac 12 \la  a\frac{k_{rw}^+}{\mu_w}k_f \frac{\partial p_h^+}{\partial \boldsymbol{\nu}_2}, \frac{\partial \xi^+}{\partial \boldsymbol{\nu}_2}\ra_{\mathcal{E}^{\gamma_2}}\\
    &-\frac 12[\{a\frac{k_{rw}}{\mu_w}k_f\frac{\partial p_h}{\partial \boldsymbol{\nu}_2}\}_{P_{-}^{\star}}, \jl \xi \jr_{P_{-}^\star}]_{\mathcal{V}^\circ \cup \mathcal{V}^D} - \frac 12[\{a\frac{k_{rw}}{\mu_w}k_f\frac{\partial p_h}{\partial \boldsymbol{\nu}_2}\}_{P_{+}^{\star}}, \jl \xi \jr_{P_{+}^\star}]_{\mathcal{V}^\circ \cup \mathcal{V}^D} \\
    &{\color{black} +\frac \sigma 2[\{a\frac{k_{rw}}{\mu_w}k_f\frac{\partial \xi}{\partial \boldsymbol{\nu}_2}\}_{P_{-}^{\star}}, \jl p_h \jr_{P_{-}^\star}]_{\mathcal{V}^\circ \cup \mathcal{V}^D} +\frac \sigma 2[\{a\frac{k_{rw}}{\mu_w}k_f\frac{\partial \xi}{\partial \boldsymbol{\nu}_2}\}_{P_{+}^{\star}}, \jl p_h \jr_{P_{+}^\star}]_{\mathcal{V}^\circ \cup \mathcal{V}^D} }\\
    &+ [\tilde{\alpha} \{\frac{k_{rw}}{\mu_w}\}_{P^\star_{-}}\jl p_h \jr_{P_{-}^\star},\jl \xi \jr_{P_{-}^\star}]_{\mathcal{V}^\circ \cup \mathcal{V}^D} + [\tilde{\alpha}\{\frac{k_{rw}}{\mu_w}\}_{P^{\star}_{+}}\jl p_h \jr_{P_{+}^\star},\jl \xi \jr_{P_{+}^\star}]_{\mathcal{V}^\circ \cup \mathcal{V}^D}\\
    &+{\color{black}[\tilde{\beta}\jl s_{wh} \jr_{P_{-}^\star},\jl \xi \jr_{P_{-}^\star}]_{\mathcal{V}^\circ \cup \mathcal{V}^{D,\rm{in}}} + [\tilde{\beta}\jl s_{wh} \jr_{P_{+}^\star},\jl \xi \jr_{P_{+}^\star}]_{\mathcal{V}^\circ\cup \mathcal{V}^{D,\rm{in}}}}
\end{split}
\end{equation}
\begin{equation}\label{eq:H}
\begin{split}
H_h(s_{wh},\xi)=&(q_w,\xi)_{\mathcal{T}_h}
+\la g_{N,w}, \xi \ra_{\mathcal{E}^{N,\text{in}}} 
+\la \frac{k_{rw}}{\mu_w}(\frac{k_{rn}}{\mu_n}+\frac{k_{rw}}{\mu_w})^{-1}g_{N}, \xi \ra_{\mathcal{E}^{N,\text{out}}} \\
& + \sigma \la 
     g_D , \frac{k_{rw}}{\mu_w}\mathbf{K}_m\nabla\xi\cdot\mathbf{n} \ra_{ \mathcal{E}^D} + \la \alpha\frac{k_{rw}}{\mu_w} g_D,\xi \ra_{\mathcal{E}^{D}}+\la\beta s_{D,w},\xi \ra_{\mathcal{E}^{D,\text{in}}},
\end{split}
\end{equation}
and 
\begin{equation}\label{eq:I}
\begin{split}
  I_h(s_{wh}, \xi)=&{\color{black} \frac \sigma 2[\{a\frac{k_{rw}}{\mu_w}k_f\frac{\partial \xi}{\partial \boldsymbol{\nu}_2}\}_{P_{-}^{\star}}+\{a\frac{k_{rw}}{\mu_w}k_f\frac{\partial \xi}{\partial \boldsymbol{\nu}_2}\}_{P_{+}^{\star}}, \text{sign}(\boldsymbol{\nu}_2\cdot \mathbf{n})g_D]_{\mathcal{V}^D}}\\
  &+[\tilde{\alpha}(\frac{k_{rw}}{\mu_w}\xi_{P_{-}^\star}+\frac{k_{rw}}{\mu_w}\xi_{P_{+}^\star}),g_D]_{\mathcal{V}^{D}}{+[\tilde{\beta}s_{D,w}, (\xi_{P_{-}^\star}+\xi_{P_{+}^\star})]_{ \mathcal{V}^{D,\rm{in}}}}{ + \la q_{f,w}, \{\xi\} \ra_{\mathcal{E}^{\gamma_2}}},
\end{split}
\end{equation}
where we take $k_{r\alpha}=k_{r\alpha}(s_{D,w})$ and $k_{r\alpha}=k_{r\alpha}(s_{wh})$ on $\Gamma_{D,\text{in}}$ and $\Gamma_{D,\text{out}}:=\Gamma_D\setminus\Gamma_{D,\text{in}}$, respectively, for $\alpha=n,w$, and $\beta=\beta_0\|\mathbf{K}_m\nabla p\|_{L^\infty(\Omega)}$ and $\tilde{\beta}=\tilde{\beta}_0\|ak_f\frac{\partial p^\pm}{\partial\boldsymbol{\nu}_2} \|_{L^\infty(\gamma_2)}$ are penalty parameters for some positive numbers $\beta_0,\tilde{\beta}_0$.

We adopt the third-order strong stability preserving Runge-Kutta (SSP-RK3) method \cite{shu1988efficient} to evolve the system over time, with each stage solving the pressure equation \eqref{eq:scheme_pressure} implicitly and updating the saturation equation \eqref{eq:scheme_saturation} explicitly (implicit-pressure explicit-saturation method). The total variation bounded (TVB) \cite{cockburn1998runge} and bound-preserving \cite{zhang2010maximum, guo2021conservative} limiters are applied to control oscillations and preserve the physical bounds of the saturation.

\subsection{A numerical experiment}
\begin{exmp}\label{ex:two-phase}
\textbf{two-phase flows in the complex fracture network}

We present a numerical experiment to demonstrate the validity of the schemes \eqref{eq:scheme_pressure} - \eqref{eq:scheme_saturation}. 
For simplicity, we consider the same setup and grid as used in Example \ref{ex:complex} in Section \ref{Sect:tests}, where conductive fractures and blocking barriers coexist and intersect. 
Moreover, we set $\phi^m=0.2$, $\phi^f=1$, $\mu_n=\mu_w=1$, $k_{rn}(s_w)=1-s_{w}$, $k_{rw}(s_w)=s_w$, $s_{w0}=0$ and $s_{D,w}=1$ in the problem.

{We take the penalty parameters $\alpha_0=\tilde{\alpha}_0=10, \beta_0=\tilde{\beta}_0=2$ in both cases.}
The saturation in the predominantly vertical flow (as in the case (a) of Example \ref{ex:complex}) at times $T_1=0.01, T_2=0.02, T_3=0.03$, and $T_4=0.04$ is presented in Figure \ref{fig:C_complex_T2B}, where the blocking effect of barriers is not significant as the flow is almost parallel to the barriers.
The saturation in the predominantly horizontal flow (as in the case (b) of Example \ref{ex:complex}) at times $T_1=0.015, T_2=0.03, T_3=0.045$, and $T_4=0.06$ is presented in Figure \ref{fig:C_complex_L2R}. From this figure, we can clearly see the effects of conductive fractures and blocking barriers on the flow. 
\begin{figure}[!htbp]
 \centering
 \begin{subfigure}[b]{0.45\textwidth}
  \includegraphics[width=\textwidth]{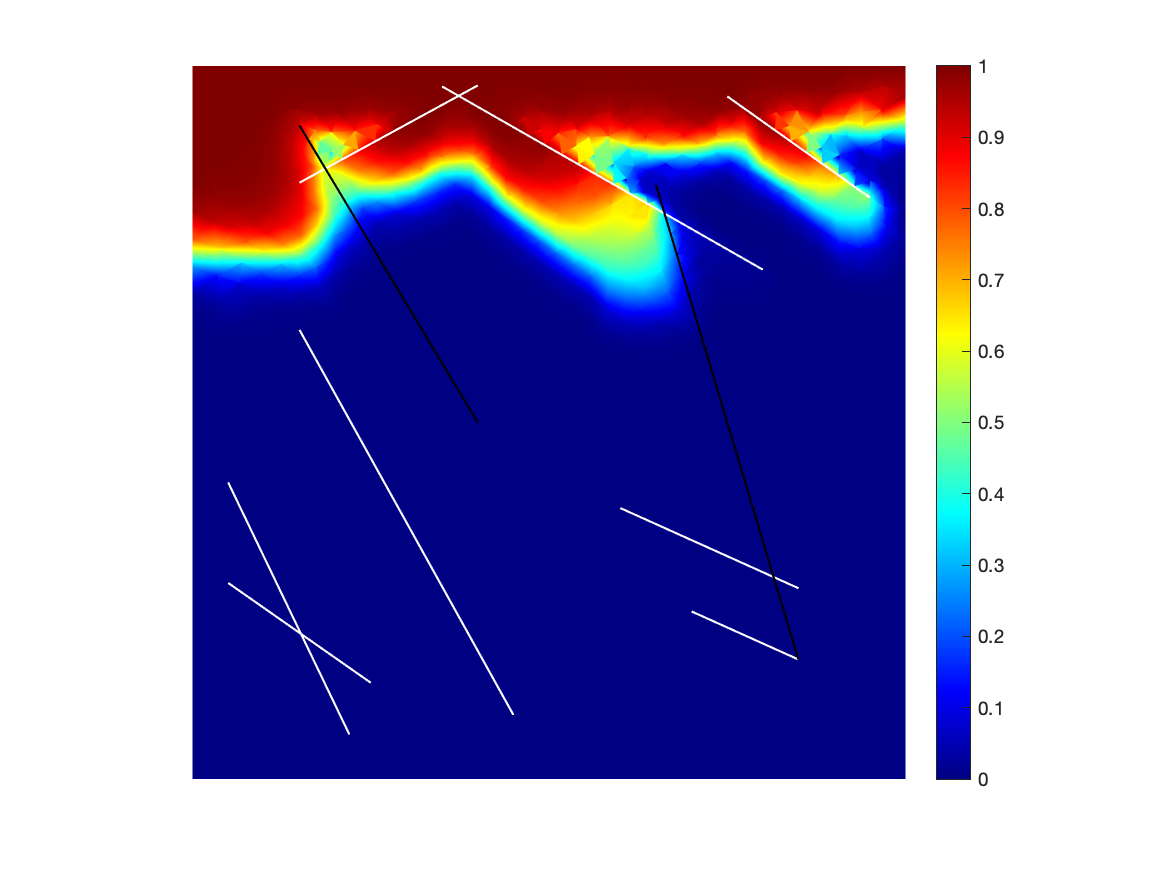}
  \caption{$T_1=0.01$}
 \end{subfigure}
 \begin{subfigure}[b]{0.45\textwidth}
  \includegraphics[width=\textwidth]{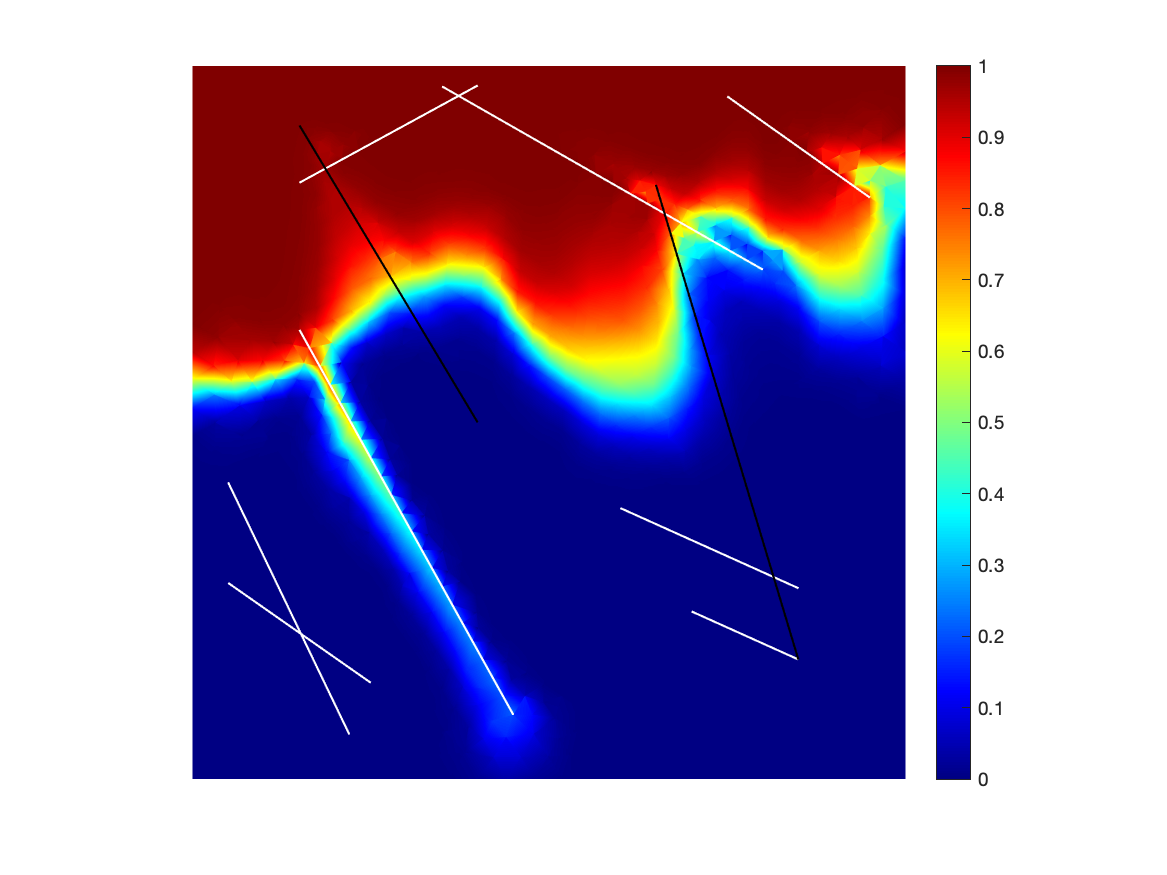}
  \caption{$T_2=0.02$}
 \end{subfigure}\\
 \begin{subfigure}[b]{0.45\textwidth}
  \includegraphics[width=\textwidth]{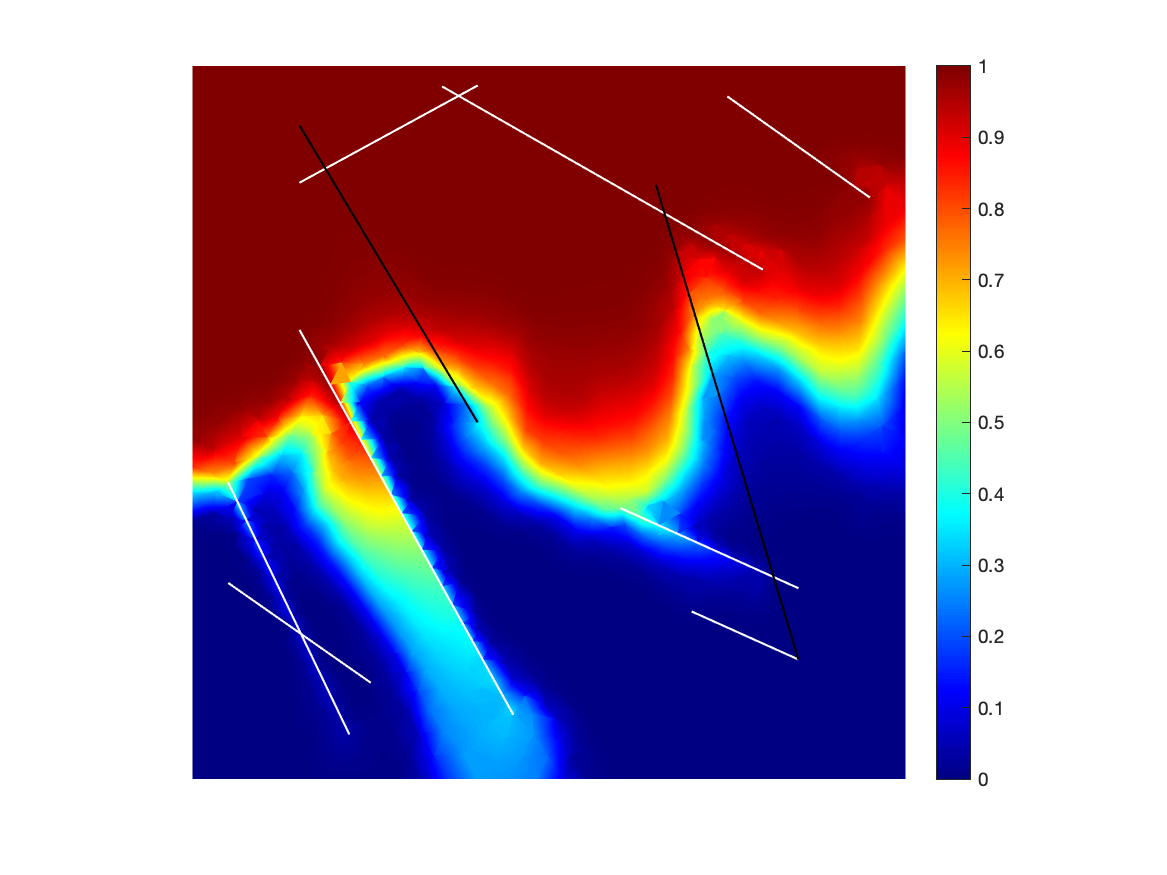}
  \caption{$T_3=0.03$}
 \end{subfigure}
 \begin{subfigure}[b]{0.45\textwidth}
  \includegraphics[width=\textwidth]{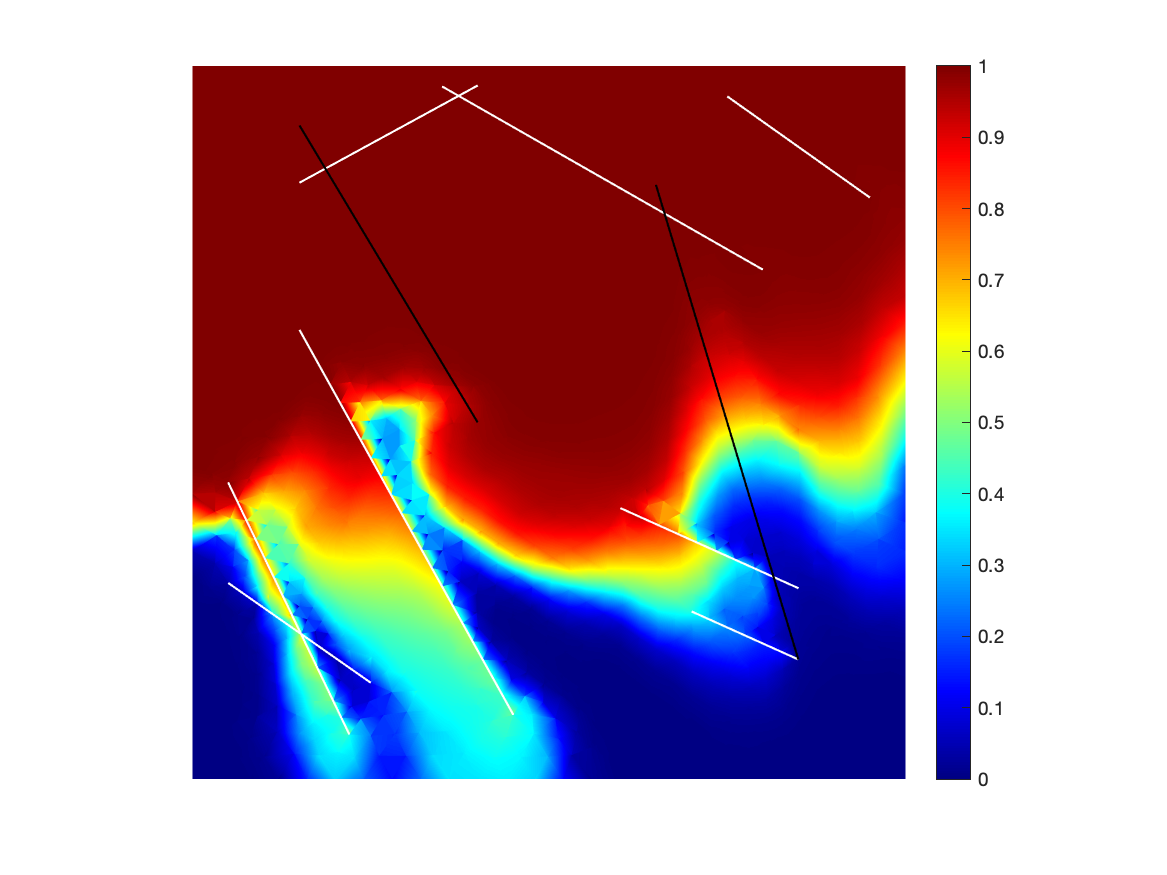}
  \caption{$T_4=0.04$}
 \end{subfigure}\\ 
\caption{
\textbf{Example \ref{ex:two-phase}: two-phase flow in complex fracture network.}
Saturation at different times, computed using the $P^1$-SIPG method, for the predominantly vertical two-phase flow.}
 \label{fig:C_complex_T2B}
\end{figure}
\begin{figure}[!htbp]
 \centering
 \begin{subfigure}[b]{0.45\textwidth}
  \includegraphics[width=\textwidth]{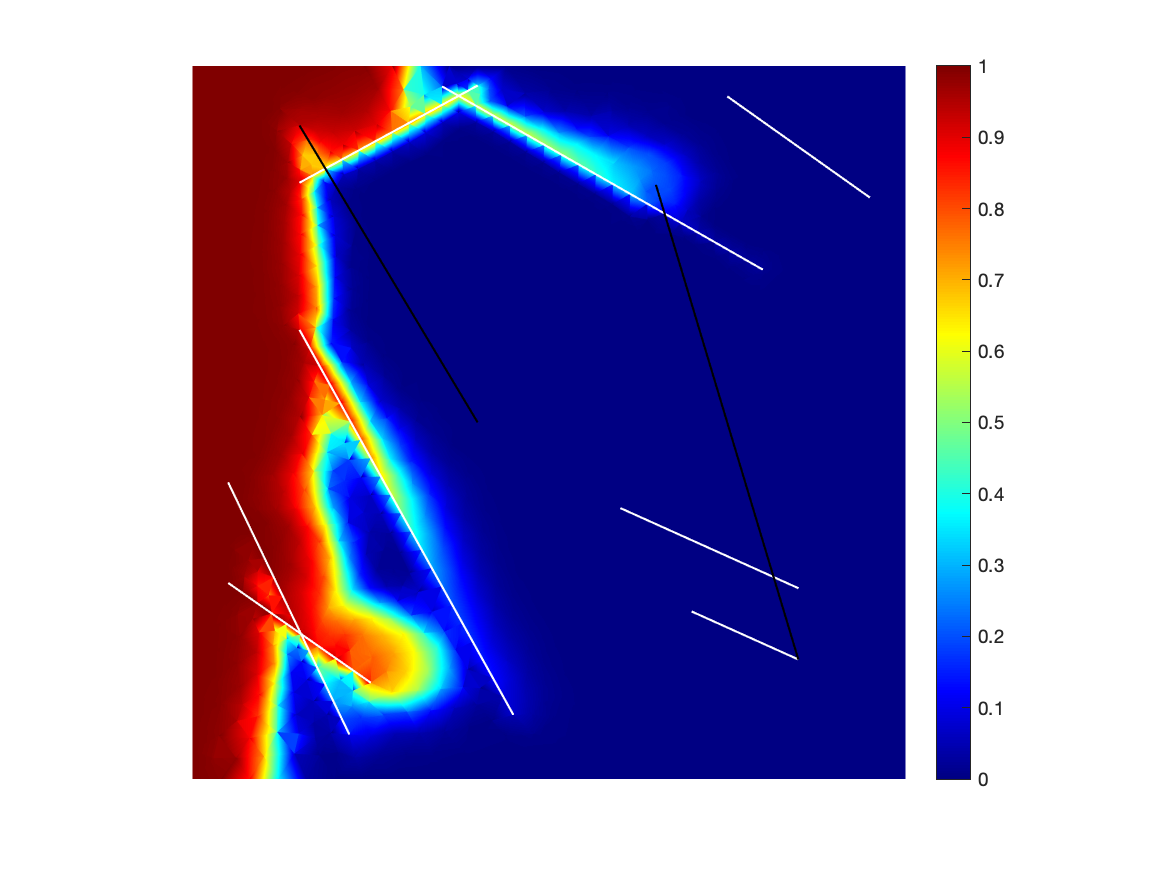}
  \caption{$T_1=0.015$}
 \end{subfigure}
 \begin{subfigure}[b]{0.45\textwidth}
  \includegraphics[width=\textwidth]{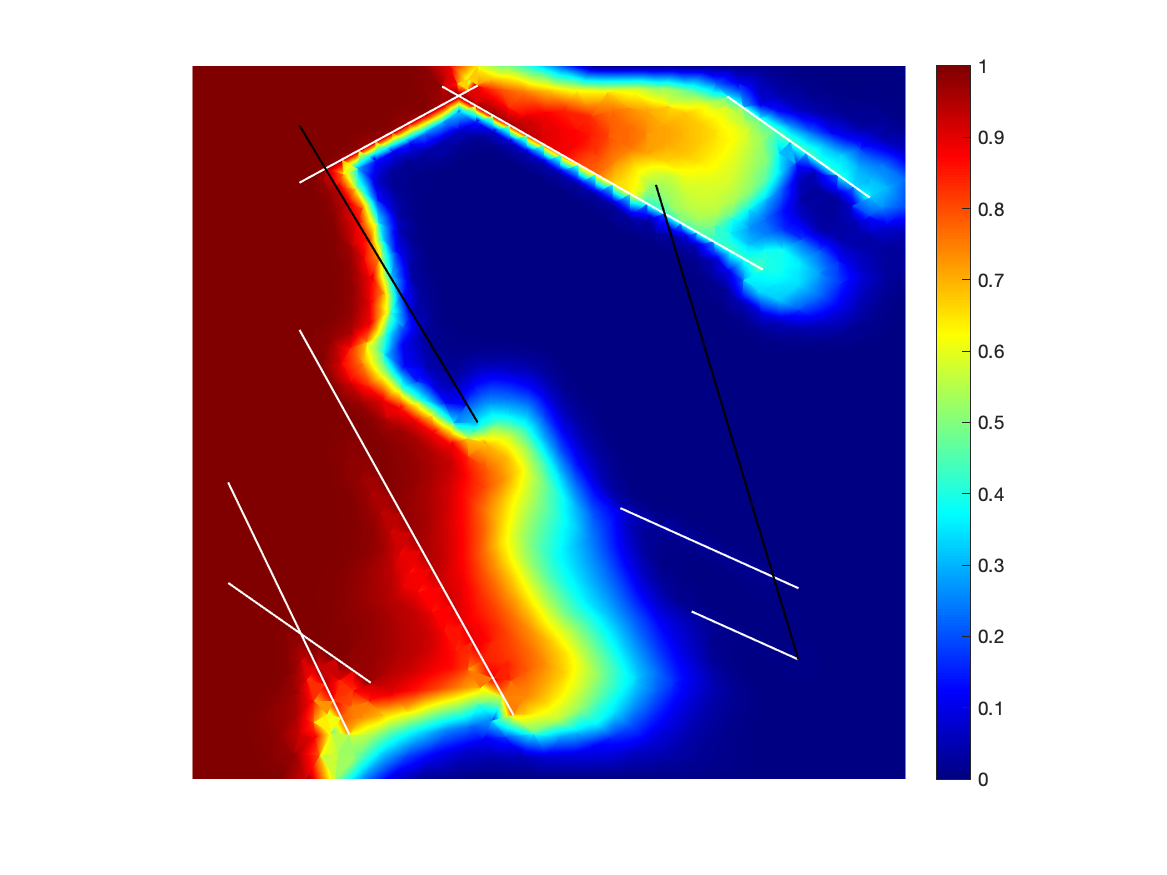}
  \caption{$T_2=0.03$}
 \end{subfigure}\\
 \begin{subfigure}[b]{0.45\textwidth}
  \includegraphics[width=\textwidth]{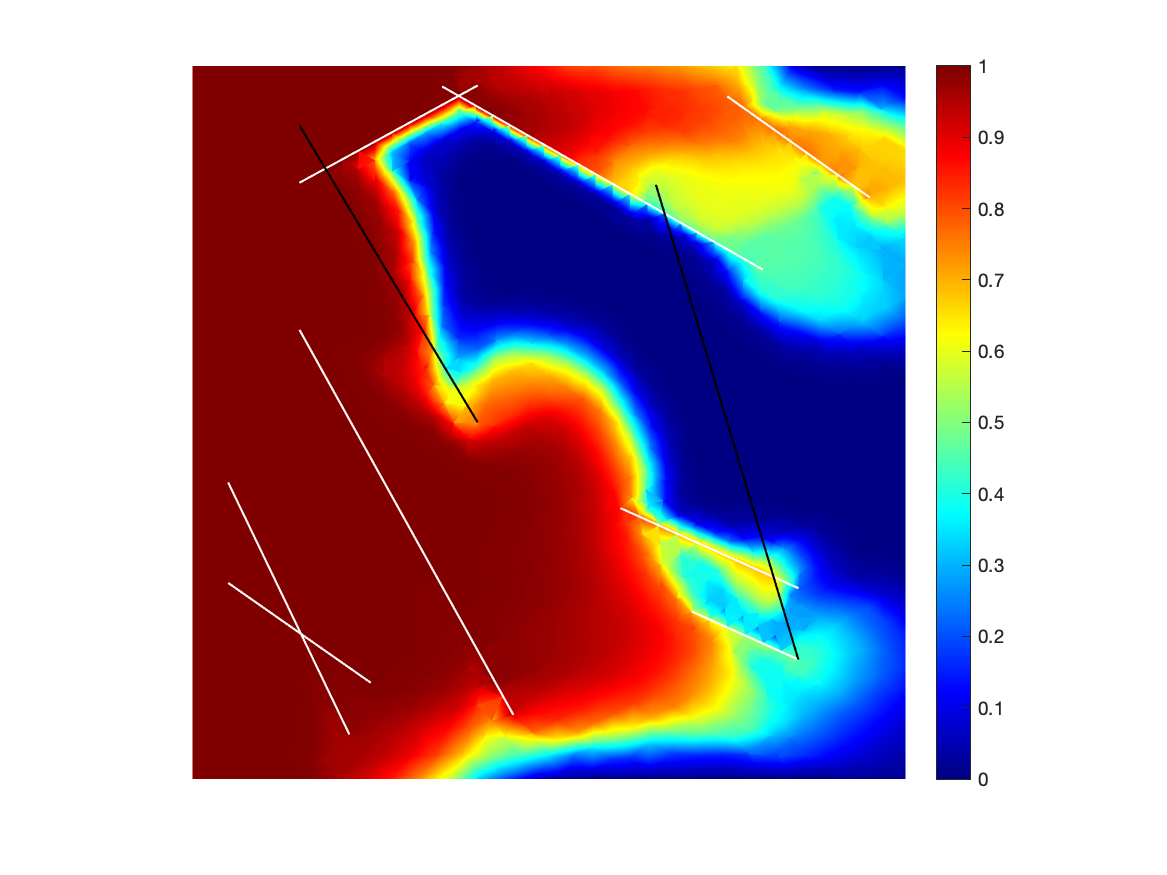}
  \caption{$T_3=0.045$}
 \end{subfigure}
 \begin{subfigure}[b]{0.45\textwidth}
  \includegraphics[width=\textwidth]{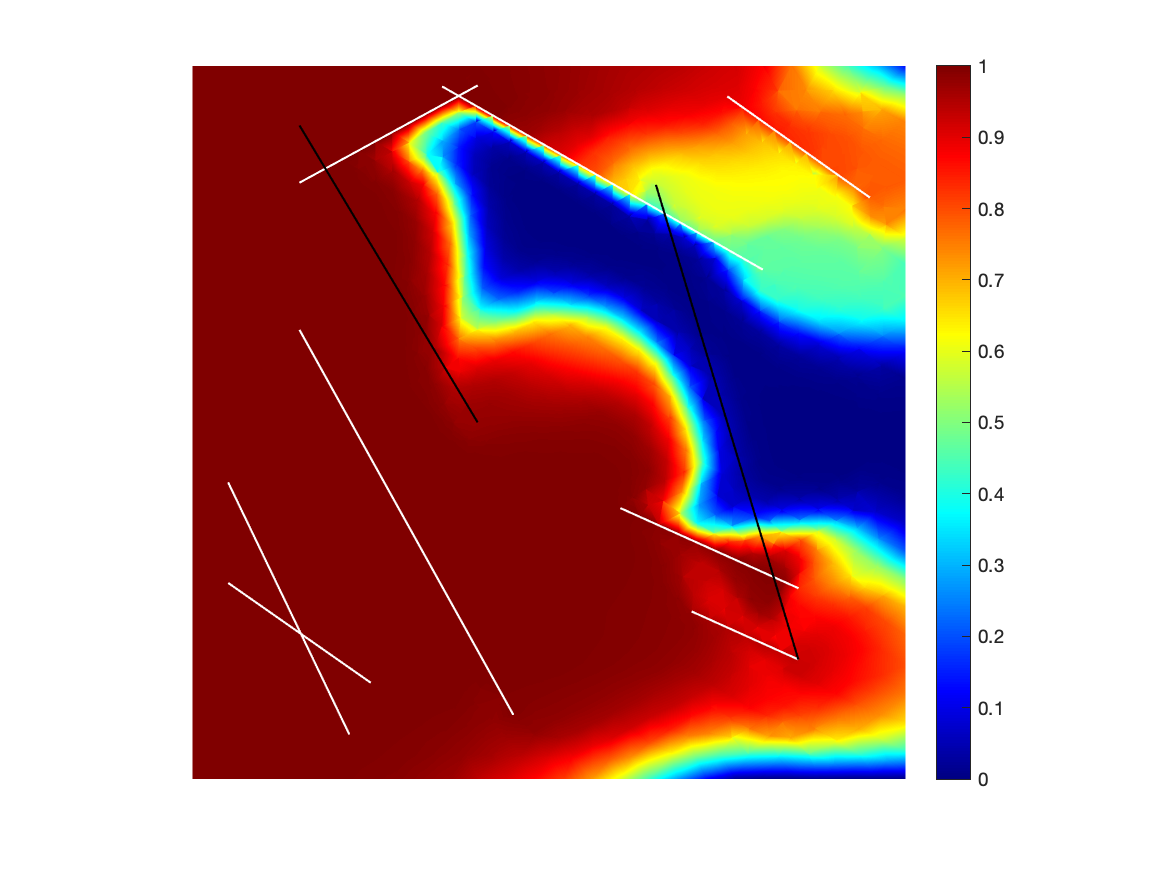}
  \caption{$T_4=0.06$}
 \end{subfigure}\\ 
 \caption{
 \textbf{Example \ref{ex:two-phase}: two-phase flow in complex fracture network.}
 Saturation at different times, computed using the $P^1$-SIPG method, for the predominantly horizontal two-phase flow.}
 \label{fig:C_complex_L2R}
\end{figure}
\end{exmp}

\section{Summary}\label{Sect:summary}
In this paper, we propose an interior penalty discontinuous Galerkin method for the simulation of fluid flow in fractured porous media.
An interface model accounting for both conductive and blocking fractures is discretized on fitted meshes without introducing additional degrees of freedom or equations on interfaces, thereby reducing computational costs. 
We provide stability analysis and error estimates for the scheme and derive optimal a priori error bounds in terms of mesh size $h$ and sub-optimal bounds regarding polynomial degree $k$, assessed in both the energy norm and $L^2$ norm.
Numerical experiments involving published benchmarks are conducted to validate the theoretical analysis and demonstrate the small model error of our method. 
Moreover, we extend our approach to two-phase flow and use numerical tests to demonstrate the validity of our algorithm.

\section*{Acknowledgements}
The authors are grateful to Dr. Dennis Gläser for sharing the data used in Section \ref{Sect:tests} for comparison.

\end{document}